\documentclass{article}

\usepackage[colorinlistoftodos,bordercolor=orange,backgroundcolor=orange!20,line
color=orange,textsize=scriptsize]{todonotes}
\catcode`\@=11
\catcode`\@=12
\usepackage{subcaption}

\usepackage{arxiv}
 \renewcommand{\citet}[1]{\cite{#1}}
 \usepackage{wrapfig}

\usepackage{amsmath, amsthm, amssymb, epsfig, sectsty, graphicx, float, url}
\usepackage{bm}
\usepackage{algorithm, algorithmicx}
\usepackage[noend]{algpseudocode}

\usepackage[utf8]{inputenc} 
\usepackage[T1]{fontenc}    
\usepackage{hyperref}       
\usepackage{url}            
\usepackage{booktabs}       
\usepackage{amsfonts}       
\usepackage{nicefrac}       
\usepackage{microtype}      
\usepackage{cleveref}

\usepackage{titlesec}
\titlespacing*{\paragraph} {0pt}{0.35ex plus 0.3ex minus .2ex}{1em}

\newcommand{\mcX}{\mathcal{X}}

\newcommand{\mbR}{\mathbb{R}}

\newcommand{\mcS}{\mathcal{S}}
\newcommand{\mbN}{\mathbb{N}}
\newcommand{\mcY}{\mathcal{Y}}

\newcommand{\mcC}{\mathcal{C}}
\newcommand{\mcV}{\mathcal{V}}
\newcommand{\mcI}{\mathcal{I}}

\DeclareMathOperator*{\argmin}{arg\,min}

\newcommand{\cbap}{{\sf CBA\textsuperscript{+}}}
\newcommand{\spcbap}{{\sf SP-CBA\textsuperscript{+}}}
\newcommand{\cba}{{\sf CBA}}
\newcommand{\rmp}{{\sf RM\textsuperscript{+}}}
\newcommand{\cfrp}{{\sf CFR\textsuperscript{+}}}
\newcommand{\rmm}{{\sf RM}}
\newcommand{\tmm}{{\sf time-max}}
\newcommand{\chp}{{\sf CHOOSEDECISION}}
\newcommand{\upp}{{\sf UPDATEPAYOFF}}

\usepackage{footnote}
\theoremstyle{plain}
\newtheorem{theorem}{Theorem}[section]
\newtheorem{lemma}[theorem]{Lemma}

\newtheorem{proposition}[theorem]{Proposition}

\theoremstyle{definition}

\newtheorem{remark}[theorem]{Remark}

\title{Solving optimization problems with Blackwell approachability}

\author{%
   Julien Grand-Cl{\'e}ment \\
ISOM Department\\
HEC Paris \\
   \texttt{grand-clement@hec.fr} \\
   \And
  Christian Kroer\\
  IEOR Department\\
Columbia University\\
  \texttt{christian.kroer@columbia.edu} \\
}

\begin{document}

\maketitle

\vspace{4mm}

\begin{abstract}
  We introduce the Conic Blackwell Algorithm$^+$ (\cbap) regret minimizer, a new parameter- and scale-free regret minimizer for general convex sets.
  \cbap{} is based on Blackwell approachability and attains $O(\sqrt{T})$ regret.
  We show how to efficiently instantiate \cbap{} for many decision sets of interest, including the simplex, $\ell_{p}$ norm balls, and ellipsoidal confidence regions in the simplex.
  Based on \cbap\, we introduce \spcbap\, a new parameter-free algorithm for solving convex-concave saddle-point problems, which achieves a $O(1/\sqrt{T})$ ergodic rate of convergence.  In our simulations, we demonstrate the wide applicability of \spcbap{} on several standard saddle-point problems, including matrix games, extensive-form games, distributionally robust logistic regression, and Markov decision processes. In each setting, \spcbap{} achieves state-of-the-art numerical performance, and outperforms classical methods, without the need for any choice of step sizes or other algorithmic parameters.
\end{abstract}

\section{Introduction}\label{sec:intro}
In this paper\footnote{A preliminary version of this paper has appeared as a conference paper by the same authors~\citep{grand2021conic}.}, we develop new algorithms for solving the following convex-concave \emph{saddle-point problems} (SPPs):
\begin{align}\label{eq:spp}
\min_{\bm{x} \in \mcX}  \max_{\bm{y} \in \mcY} F(\bm{x},\bm{y}),
\end{align}
where $\mcX \subset \mbR^{n}, \mcY \subset \mbR^{m}$ are convex, compact sets, and $F: \mcX \times \mcY \rightarrow \mbR$ is a subdifferentiable convex-concave function.
The optimization problem \eqref{eq:spp} arises in a number of practical problems. For example, the problem of computing a Nash equilibrium of a zero-sum game can be formulated as a convex-concave SPP, and this is the foundation of most methods for solving sequential zero-sum games~\citep{stengel1996efficient,zinkevich2007regret,tammelin2015solving,kroer2018faster}. Other instances include imaging~\citep{ChambollePock2011}, $\ell_{\infty}$-regression \citep{sidford2018coordinate}, Markov Decision Processes (MDPs) and robust MDPs~\citep{Iyengar,Kuhn,sidford2018coordinate}, market equilibrium~\citep{kroer2021computing} and distributionally robust logistic regression, where the $\max$ term represents the distributional uncertainty~\citep{namkoong2016stochastic,ben2015oracle}.
We introduce efficient algorithms for solving \eqref{eq:spp}, focusing on \textit{parameter-free} algorithms that do not require choosing, learning or tuning any step sizes.

\paragraph{Repeated game framework}
One way to solve convex-concave SPPs is by viewing the SPP as a repeated game between two players: at each iteration $t$, one player chooses $\bm{x}_t\in \mcX$, the other player chooses $\bm{y}_t\in \mcY$, and then the players observe the payoff $F(\bm{x}_t,\bm{y}_t)$.
If each player employs a regret-minimization algorithm, then a well-known theorem says that the uniform average of the decisions generated by the players converge to a solution to the SPP (see Theorem \ref{th:-folk-theorem} in Section \ref{sec:game-setup}).
We will call this the ``repeated game framework''.
There are already well-known algorithms for instantiating the above repeated game framework for \eqref{eq:spp}. 
For example, one can employ the \emph{online mirror descent} (OMD) algorithm~\citep{nemirovsky1983problem}, which generates iterates as follows for the first player (and similarly for the second player):
\begin{align}\label{eq:omd-intro}
       \bm x_{t+1} = \argmin_{\bm{x} \in \mcX} \langle \eta \bm{f}_{t}, \bm{x} \rangle + D(\bm{x}, \bm{x}_{t}),
\end{align}
where $\bm{f}_{t} \in \partial_{\bm{x}} F(\bm{x}_t, \bm{y}_t)$ ($\partial_{\bm{x}}$ denotes the set of subgradients as regards the variable $\bm{x}$), $\eta>0$ is an appropriate step size, and $D$ is a \emph{Bregman divergence} which measures distance between points.
Another example of a regret minimizer is Follow-The-Regularized-Leader (FTRL)~\citep{abernethy2009competing}, which generates updates as follows:
\begin{align}\label{eq:ftrl-intro}
       \bm x_{t+1} = \argmin_{\bm{x} \in \mcX} \langle \eta \sum_{\tau=1}^{t} \bm{f}_{\tau}, \bm{x} \rangle + D(\bm{x}, \bm{x}_{t}). 
\end{align}
The updates \eqref{eq:omd-intro} and \eqref{eq:ftrl-intro} can be computed efficiently for many decision sets $\mcX$ and one can achieve an average regret on the order of $O(1/\sqrt{T})$ after $T$ iterations. This regret can be achieved by choosing a fixed step size $\eta = \sqrt{2}\Omega /L \sqrt{T}$, where $L$ is an upper bound on the $\ell_{2}$-norms of the subgradients $\left( \bm{f}_{t} \right)_{t \geq 0}$ and $\Omega = \max \{ \| \bm{x} - \bm{x}' \|_{2} \; \vert \; \bm{x},\bm{x}' \in \mcX\}.$ 
Choosing the step size $\eta$ is problematic, as it requires choosing in advance the number of iterations $T$ and to know the upper bound $L$, which may be hard to obtain in many applications or too conservative in practice. 
 Alternatively, it is possible to choose changing step sizes $\eta_{t} = \alpha / \sqrt{t}$, for $\alpha>0$.  Still, adequately tuning the parameter $\alpha$ can be time- and resource-consuming.  This is not just a theoretical issue,  as we highlight in our numerical experiments (Section \ref{sec:simu}) and in the appendices (Appendices \ref{app:details-omd}). 

These issues can be addressed by employing \textit{adaptive} step sizes, which estimate the parameters through the observed subgradients, e.g., AdaHedge for the simplex setting~\citep{de2014follow} or AdaFTRL for general compact convex decisions sets~\citep{orabona2015scale}.
These adaptive variants have not seen practical adoption in large-scale game-solving, where variants based on Blackwell approachability are preferred (see the next paragraph).
As we show in our experiments, adaptive variants of OMD and FTRL perform much worse than our proposed algorithms.
While these adaptive algorithms are referred to as \emph{parameter-free}, this is only true in the sense that they are able to learn the necessary parameters. Our algorithm is parameter-free in the stronger sense that there are no parameters that even require learning. 

\paragraph{Blackwell approachability}
In this paper, we use the framework of \textit{Blackwell} approachability~\citep{blackwell1956analog} to develop novel parameter-free algorithms for solving the convex-concave saddle-point problem \eqref{eq:spp}. In principle, Blackwell approachability arises in the framework of repeated two-player games with vector-valued payoff: the goal of the first-player is to choose a sequence of decisions $\bm{x}_{1},\bm{x}_{2},...,$ such that the \textit{average} of the visited payoff converges to a known target set $\mcS$, while the second-player is typically playing adversarially. Blackwell's celebrated theorem~\citep{blackwell1956analog} provides an algorithm for constructing such a sequence of decisions $\bm{x}_{1},\bm{x}_{2},...,$ in the case where the target set $\mcS$ is \textit{half-space forceable} (see details in Section \ref{sec:game-setup}).

Blackwell approachability is a very general framework and the applications are numerous, ranging from stochastic games~\citep{milman2006approachable}, revenue management, market design, and submodular maximization~\citep{niazadeh2020online}, calibration~\citep{perchet2010approachability}, learning in games~\citep{aumann1995repeated}, and fair online learning~\citep{chzhen2021unified}. In particular,  Blackwell approachability can be used as a regret minimizer~\citep{abernethy2011blackwell}, and provides a no-regret algorithm, with a average regret of $O\left(1/\sqrt{T}\right)$ after $T$ iterations. Crucially, when applied to online regret minimization, Blackwell approachability can be instantiated without evaluating any of the smoothness or convexity parameters of the objective function $F$,  and the resulting no-regret algorithm does not use any step sizes: this is in contrast to classical regret minimizers such as OMD \eqref{eq:omd-intro} and FTRL \eqref{eq:ftrl-intro}, which require choosing step sizes.

Despite its appealing properties from a theoretical standpoint, in practice Blackwell approachability is not widely used to solve classical optimization problems. In fact, to the best of our knowledge, the only practical implementation of Blackwell approachability for solving \eqref{eq:spp} is for the case of bilinear games on the simplex, where $F(\bm{x},\bm{y}) = \langle \bm{x},\bm{Ay}\rangle$ for $\bm{A} \in \mbR^{n \times m}$, and $\mcX,\mcY$ are simplices. This simplex instantiation is also used for Extensive-Form Games (EFGs), via the aforementioned CFR decomposition~\citep{zinkevich2007regret,farina2019online}. 
In the simplex setting, a particular application of Blackwell approachability yields a no-regret algorithm called \textit{regret matching} (\rmm)~\citep{hart2000simple}. Combining \rmm{} with specific weighting, thresholding, and alternating schemes yields an algorithm called \emph{regret matching$^+$} (\rmp)~\citep{tammelin2015solving}.
\rmp\ has been used in \emph{every} case of solving extremely large-scale EFGs in practice, and in particular it was used in recent poker AI milestones, where poker AIs beat human poker players~\citep{bowling2015heads,moravvcik2017deepstack,brown2018superhuman,brown2019superhuman}.
In fact, \rmp{} routinely outperforms theoretically-superior methods, such as optimistic variants of OMD and FTRL~\citep{rakhlin2013online,chiang2012online}, which achieve $O\left(1/T\right)$ convergence rates in the repeated game framework.
Despite its very strong empirical performances, \rmp{} is only defined when the decision set is the simplex. However, many problems of the form \eqref{eq:spp} have convex sets $\mcX,\mcY$ that are not simplexes, e.g., box constraints or norm-balls for distributionally robust optimization~\citep{ben2015oracle}. 
Encouraged by the very  strong empirical performance of \rmp{} and \cfrp, we will construct parameter-free algorithms based on Blackwell approachability for solving more general instances of the saddle-point problem \eqref{eq:spp}.
 
\subsection{Our Contributions}
Our main contributions are as follows.
\begin{itemize}
\item 
\textit{Conic Blackwell Algorithm$^+$ (\cbap).} We start from the general reduction between regret minimization over general convex compact sets and Blackwell approachability~\citep{abernethy2011blackwell}. 
This yields a regret minimizer which we will refer to as the \emph{conic Blackwell algorithm} (\cba). Motivated by the practical performance of \rmp{} on simplexes, we construct a variant of \cba{} which uses a thresholding operation analogous to the one employed by {\rmp}. We call this regret minimizer \cbap{} (Algorithm \ref{alg:CBAp}). We show that \cbap{} achieves $O(1/\sqrt{T})$ average regret in the worst-case. A major selling point of \cbap\ is that it does not require any step size choices. Instead, \cbap{} implicitly adjusts to the structure of the domains and losses by being instantiations of a Blackwell approachability algorithm, which is itself parameter-free.
\item
 \textit{Impacts of weights and alternation.}
As regret minimizers, we show that both \cba{} and \cbap{} are compatible with increasing weighting schemes, that put more weights on more recent decisions and payoffs (Theorem \ref{th:cba-linear-averaging-both} and Theorem \ref{th:cbap-linear-averaging-only-policy}), where \cbap{} is compatible with different weighting schemes for the decisions and the payoffs.
We then introduce a new algorithm for solving convex-concave saddle-point problems by using \cbap\ in a repeated game framework with linear weights on the sequence of decisions and uniform weights on the payoffs (this is known as \emph{linear averaging} in other algorithms~\citep{tammelin2015solving,gao2021increasing}), as well as an alternating payoff scheme. We call this algorithm \spcbap.
We quantify the benefits of alternation for solving \eqref{eq:spp} (Theorem \ref{th:alternation-works-bilinear-case}), and show the first strict improvement guarantee for using alternation; note that prior results only showed that it does not slow the convergence~\citep{burch2019revisiting}.
\item \textit{Efficient implementation of \cbap.} We show how to implement \cba{} and \cbap{} when $\mcX$ and $\mcY$ are simplexes, $\ell_p$ balls, and intersections of the $\ell_2$ ball with a simplex,  which arises naturally as a confidence region. More generally, \cba{} and \cbap{} can be implemented when we can efficiently compute orthogonal projections onto the set $\mcX$ and $\mcY$.
Note that the general reduction of regret minimization and Blackwell approachability from \cite{abernethy2011blackwell} yields \cba, but does not yield a practically-implementable algorithm, as the authors do not consider which decision sets allow for efficient projections.

\item\textit{Practical performance of \spcbap.}
We highlight the practical efficacy of our algorithmic framework on several domains.
First, we apply \spcbap{} to two-player zero-sum matrix games, where the objective function is bilinear, and we compare with \rmp{}, as well as with AdaHedge and AdaFTRL, two adaptive first-order algorithms. We then apply \spcbap{} to extensive-form games (EFGs), where the \rmp{} regret minimizer combined with linear averaging, alternation, and a counterfactual regret (\cfrp) minimization scheme, leads to state-of-the-art practical algorithms~\citep{tammelin2015solving,kroer2018faster,gao2021increasing}. For EFGs, we find that \spcbap{} leads to comparable performance in terms of the iteration complexity, and for some games it slightly outperforms \cfrp. In the simplex setting we also find that \spcbap{} outperforms both AdaHedge and AdaFTRL.
These results show that \spcbap\ recovers the strong practical performance of \rmp\ and \cfrp{} in the only setting where these two methods apply.
Second, and more importantly, we show that \spcbap\ leads to strong practical performance in settings where \rmp{} and \cfrp{}  do not apply.  We consider instances of distributionally robust logistic regression and Markov decision processes (MDPs). For these two instances of saddle-point problems, we find that \spcbap{} performs orders of magnitude better than online mirror descent and follow-the-regularized leader, as well as their optimistic variants, when using their theoretically-correct fixed step sizes. 
Even when considering tuned step sizes for the other algorithms, \spcbap{} performs better, with only a few cases of comparable performance (at step sizes that lead to divergence for some of the other non-parameter-free methods). The fast practical performance of our algorithm, combined with its simplicity and the total lack of step sizes or parameters tuning, suggests that it should be seriously considered as a practical approach for solving  convex-concave optimization instances in various settings.
\end{itemize}

We conclude our introduction with a brief discussion on the average regret achieved by other methods, and resulting convergence to a saddle point. Our algorithm \spcbap{} has a rate of convergence towards a saddle point of  $O(1/\sqrt{T})$, similar to OMD and FTRL.  In theory, it is possible to obtain a faster $O \left( 1/T \right)$ rate of convergence when $F$ is differentiable with Lipschitz gradients, for example via mirror prox~\citep{nemirovski2004prox} or other primal-dual algorithms~\citep{ChambollePock16}. 
However, our experimental results show that \spcbap{} is faster than optimistic variants of FTRL and OMD~\citep{syrgkanis2015fast}, the latter being almost identical to the mirror prox algorithm, and both achieving $O(1/T)$ rate of convergence.
A similar conclusion has been drawn in the context of sequential game solving, where the \rmp-based algorithms have better practical performance than the theoretically-superior $O \left( 1/T \right)$-rate methods~\citep{kroer2018faster,kroer2018solving}.
In a similar vein, using \emph{error-bound conditions}, it is possible to achieve a linear rate, e.g., when solving bilinear saddle-point problems over polyhedral decision sets, by using the extragradient method~\citep{tseng1995linear} or optimistic gradient descent-ascent~\citep{wei2020linear}. However, these linear rates rely on unknown constants, and may not be indicative of practical performance.
\section{Repeated game framework and Blackwell approachability}\label{sec:game-setup}
We will solve \eqref{eq:spp} using a repeated game framework. 
There are $T$ iterations with indices $t=1,\ldots,T$. In this framework, each iteration $t$ consists of the following steps:
\begin{enumerate}
    \item Each player chooses strategies $\bm x_t\in \mcX, \bm y_t \in \mcY$.
    \item The first player observes $\bm{f_t} \in \partial_{\bm x} F(\bm x_t,\bm y_t)$ and uses $\bm{f}_{t}$ when computing the next strategy.
    \item The second player observes $\bm{g_t} \in \partial_{\bm y} F(\bm x_t,\bm y_t)$ and uses $\bm{g}_{t}$ when computing the next strategy.
\end{enumerate}
In the repeated game framework described above, the first player chooses strategies from $\mcX$ to minimize the sequence of payoffs in the repeated game, while the second player chooses strategies from $\mcY$ in order to maximize payoffs.
The goal of each player is to minimize their regret $R_{T,\bm{x}},R_{T,\bm{y}}$ across the $T$ iterations:
\begin{align*}
R_{T,\bm{x}}  = \sum_{t=1}^{T} \langle \bm f_t, \bm x_t \rangle - \min_{\bm{x} \in \mcX} \sum_{t=1}^{T} \langle \bm f_t, \bm x \rangle ,\quad
R_{T,\bm{y}} = \max_{\bm{y} \in \mcY} \sum_{t=1}^{T} \langle \bm g_t ,\bm{y} \rangle - \sum_{t=1}^{T} \langle \bm g_t,\bm{y}_{t} \rangle.
\end{align*}
The reason this repeated game framework leads to a solution to the SPP problem~\eqref{eq:spp} is the following well-known theorem (e.g., Theorem 1, \cite{kroer2020ieor8100}). Relying on $F$ being convex-concave and subdifferentiable,  it connects the regret incurred by each player to the duality gap in \eqref{eq:spp}.
\begin{theorem}\label{th:-folk-theorem}
Let $\left( \bar{\bm{x}}_{T},\bar{\bm{y}}_{T} \right) = \dfrac{1}{T}\sum_{t=1}^{T} \left(\bm{x}_{t}, \bm{y}_{t} \right)$ for any $\left(\bm{x}_{t}\right)_{t \geq 1},\left(\bm{y}_{t}\right)_{t \geq 1}$.
Then 
\[\max_{\bm{y} \in \mcY} F(\bar{\bm{x}}_{T},\bm{y}) - \min_{\bm{x} \in \mcX} F(\bm{x},\bar{\bm{y}}_{T}) \leq (R_{T,\bm{x}} + R_{T,\bm{y}})/T. \]
\end{theorem}
Therefore,  when each player uses a regret minimizer that guarantees regret on the order of $O(\sqrt{T})$,   $\left(\bar{\bm{x}}_{T}, \bar{\bm{y}}_{T} \right)_{T \geq 0}$ converges to a solution to \eqref{eq:spp} at a rate of $O\left(1/\sqrt{T}\right)$.
Later we will show a generalization of Theorem \ref{th:-folk-theorem} that will allow us to incorporate increasing averaging schemes that put additional weight on the later iterates.
Given the repeated game framework, the next question becomes which algorithms to employ in order to minimize regret for each player. As mentioned in Section \ref{sec:intro}, for matrix games and EFGs, variants of Blackwell approachability are used in practice (via the CFR decomposition for EFGs).

\paragraph{Blackwell Approachability}
In Blackwell approachability, a decision maker repeatedly takes decisions $\bm x_t$ from some convex decision set $\mcX$ (this set plays the same role as $\mcX$ or $\mcY$ in \eqref{eq:spp}). After taking decision $\bm x_t$ the player observes a vector-valued affine payoff function $\bm u_t(\bm x) \in \mathbb R^n$. 
The goal for the decision maker is to force the average payoff $\frac{1}{T} \sum_{t = 1}^{T} \bm{u}_{t}(\bm x_{t})$ to approach some convex target set $\mathcal S$.
Blackwell proved that a convex target set $\mathcal S$ can be approached if and only if for every halfspace $\mathcal H\supseteq \mathcal S$, there exists $\bm x \in \mcX$ such that for every possible payoff function $\bm u(\cdot)$, $\bm u(\bm x)$ is guaranteed to lie in $\mathcal H$. The action $\bm x$ is said to \emph{force} $\mathcal H$.
Blackwell's proof is via an algorithm: at iteration $t$, his algorithm projects the average payoff $\bar{\bm{u}}_{t} = \frac{1}{t-1} \sum_{\tau = 1}^{t-1} \bm{u}_\tau(\bm x_\tau)$ onto $\mathcal S$, and then the decision maker chooses an action $\bm x_t$ that forces the tangent halfspace to $\mathcal S$ generated by the normal vector $\bar{\bm{u}}_{t} - \pi_{\mathcal S}(\bar{\bm{u}}_{t})$, where $\pi_{\mathcal S}(\bar{\bm{u}}_{t})$ is the orthogonal projection of $\bar{\bm{u}}_{t}$ onto $\mathcal S$.
We call this algorithm \emph{Blackwell's algorithm}; it approaches $\mathcal S$ at a rate of $O(1/\sqrt{T})$~\citep{blackwell1956analog}. In particular, for $d(\bar{\bm{u}}_{T},\mcS)$ defined as $d(\bar{\bm{u}}_{T},\mcS) = \min \{ \| \bar{\bm{u}}_{T} - \bm{z} \|_{2} \; \vert \bm{z} \in \mcS \},$
we have $d(\bar{\bm{u}}_{T},\mcS) = O \left(1/\sqrt{T} \right)$. 
Blackwell's algorithm is really a meta-algorithm, rather than a concrete algorithm. Even within the context of the Blackwell approachability problem, one needs to devise a way to compute the forcing actions needed at each iteration, i.e., to compute $\pi_{\mathcal S}(\bar{\bm{u}})$. 
To the best of our knowledge, prior to this paper, the only practical implementation of Blackwell approachability for solving \eqref{eq:spp} is on the simplex for solving bilinear saddle-point problems and extensive-form games, which leads to \rmm{} and \rmp.

\paragraph{Details on Regret Matching}
Let $\Delta(n)$ be the $n$-dimensional probability simplex.
Regret Matching (\rmm) arises by instantiating Blackwell approachability with the decision space $\mcX$ equal to $\Delta(n)$, the target set $\mathcal S$ equal to the nonpositive orthant $\mathbb R_{-}^n$, and the vector-valued payoff function $\bm u_t(\bm x_t) = \bm f_t - \langle \bm f_t,\bm x_t \rangle \bm{e}$ equal to the regret associated to each of the $n$ actions (which correspond to the corners of $\Delta(n)$). 
Here $\bm{e} \in \mbR^{n}$ is the all one vector.
 \cite{hart2000simple} showed that with this setup, playing each action with probability proportional to its positive regret up to time $t$ satisfies the forcing condition needed in Blackwell's algorithm.
 Formally, regret matching (\rmm) keeps a running sum $\bm r_t = \sum_{\tau = 1}^{t} \left( \bm f_{\tau} - \langle\bm f_{\tau}, \bm x_\tau \rangle \bm{e} \right)$, and then action $i$ is played with probability $\bm x_{t+1,i} = [\bm r_{t,i}]^+ / \sum_{i=1}^n [\bm r_{t,i}]^+$, where $[\cdot]^+$ denotes thresholding at zero.
 By Blackwell's approachability theorem, this algorithm converges to zero average regret at a rate of $O(1/\sqrt{T})$.
 In zero-sum game-solving, it was discovered that a variant of regret matching leads to extremely strong practical performance (but the same theoretical rate of convergence). In regret matching$^+$ (\rmp), the running sum is thresholded at zero at every iteration: $\bm r_t = [\bm r_{t-1} + \bm f_{t} - \langle\bm f_{t}, \bm x_t \rangle \bm{e} ]^+$, and then actions are again played proportional to $\bm r_t$.
In the next section, we describe a framework by \cite{abernethy2011blackwell} for using Blackwell's algorithm to construct regret minimizers for more general convex sets $\mcX$; this will lead to the \cba\ algorithm, from which we will construct \cbap.
While we use the framework of \cite{abernethy2011blackwell}, we note that the \emph{Lagrangian Hedging} framework of \cite{gordon2007no} could also be used as the basis for developed a general class of Blackwell-approachability-style algorithms.
It would be interesting to construct a \cbap-like algorithm and efficient projection approaches for such a framework as well.

\section{Conic Blackwell Algorithm}\label{sec:cba}
\subsection{Our algorithm}\label{sec:cba-algorithm}
In this section we introduce our main regret minimizer, Conic Blackwell Algorithm Plus (\cbap), which uses a variation of Blackwell's approachability procedure~\citep{blackwell1956analog} to perform regret minimization on a general convex compact decision set $\mcX$.
 We will assume that losses are coming from a bounded set; this occurs, for example, if there exists $L_x,L_y$ (that we do not need to know), such that
\begin{equation}\label{eq:definition-Lx-Lx}
\|  \bm{f} \| \leq L_x, \| \bm{g} \| \leq L_y, \; \forall \; \bm{x}\in \mcX, \bm{y} \in \mcY, \forall \; \bm{f} \in \partial_{\bm{x}} F(\bm{x}, \bm{y}), \forall \; \bm{g} \in \partial_{\bm{y}} F(\bm{x}, \bm{y}).
\end{equation}
We will simply write $L$ for $L_{x}$ or $L_{y}$ when we focus on the regret of a single player.
We will also use the notation $\kappa = \max_{\bm{x} \in \mcX} \| \bm{x} \|_{2}$ (recall that $\mcX$ is compact).
\cbap{} is best understood as a combination of two steps. The first is the basic \cba{} algorithm, derived from Blackwell's algorithm, which we describe next.
To convert Blackwell's algorithm to a regret minimizer on $\mcX$, we use the reduction from \citep{abernethy2011blackwell}, which considers the conic hull $\mcC = \textrm{cone}(\{\kappa\} \times \mcX) \subset \mbR^{n+1}$.
The Blackwell approachability problem is then instantiated with $\mcX$ as the decision set, the target set equal to the polar  $\mcC^\circ = \{ \bm{z} : \langle \bm{z}, \bm{\hat z} \rangle \leq 0, \forall \bm{\hat z} \in \mcC\}$ of $\mcC$, and payoff vectors $(\langle\bm{f},\bm{x} \rangle, -\bm{f}) \in \mbR^{n+1}$.
The conic Blackwell algorithm (\cba) is implemented by computing the projection $\pi_\mcC(\bm{u})$ of the average payoff vector $\bm{u}$ onto $\mcC$, noting that the projection can be written as $\alpha(\kappa,\bm{x})$ where $\alpha \geq 0$ is a scalar, and playing the action $\bm{x}$.
The second step in {\cbap} is to replace the average payoff vector $\bm{u}$ with a running \textit{projected} aggregation of the payoffs, where we always add the newest payoff to the aggregate, and then project the aggregate onto $\mcC$.

More concretely, pseudocode for {\cbap} is given in Algorithm \ref{alg:CBAp}.
This pseudocode relies on two functions:
$\chp_{\cbap}: \mbR^{n+1} \rightarrow \mbR^{n}$, which maps the aggregate payoff vector $\bm u_t$ to a decision in $\mcX$, and $\upp_{\cbap}$ which controls how we aggregate payoffs.
Given an aggregate payoff vector $\bm{u}=(\tilde{u},\hat{\bm{u}}) \in \mbR \times \mbR^{n}$,  we have
\[\chp_{\cbap}(\bm{u})=(\kappa/\tilde{u}) \hat{\bm{u}}.\]
If $\tilde{u}=0$, we just let $\chp_{\cbap}(\bm{u})=\bm{x}_{0}$ for some arbitrary $\bm{x}_{0} \in \mcX$. 
The function $\upp_{\cbap}$ is implemented by adding the most recent payoff to the aggregate payoffs, and then projecting onto $\mcC$.
More formally, it is defined as
\[
\upp_{\cbap}(\bm{u}, \bm{x},\bm{f},\omega) =\pi_{\mcC} \left( \bm{u} + \omega \left( \langle \bm{f},\bm{x} \rangle / \kappa, 
- \bm{f} \right) \right),
\]
where $\omega$ is the weight assigned to the most recent payoff. Because of the projection step in $\upp_{\cbap}$, we always have $\bm{u} \in \mcC$,  which in turn guarantees that $\chp_{\cbap}(\bm{u}) \in \mcX$, since $\mcC = \textrm{cone}(\{\kappa\} \times \mcX)$.
\begin{algorithm}[t]
\caption{Conic Blackwell Algorithm Plus (\cbap)}\label{alg:CBAp}
\begin{algorithmic}[1]
\State \textbf{Input} A convex, compact set $\mcX \subset \mbR^{n}$, $\kappa = \max \{ \| \bm{x} \|_{2} \; \vert \; \bm{x} \in \mcX \}$.
\State \textbf{Algorithm parameters} Weights $\left(\omega_{\tau} \right)_{\tau \geq 1} \in \mbR^{\mbN}$.
\State \textbf{Initialization} $t=1$,  $\bm{x}_{1} \in \mcX$. 
\State Observe $\bm{f}_{1}$ then set
$\bm{u}_{1} =  \omega_{1}\left( \langle \bm{f}_{1},\bm{x}_{1} \rangle / \kappa,  - \bm{f}_{1} \right) \in \mbR \times \mbR^{n}$.
\For{$t \geq 1$}
\State Choose $\bm{x}_{t+1} = \chp_{\cbap}(\bm{u}_{t})$. \label{step:alg:projection}
\State Observe the loss $\bm{f}_{t+1} \in \mbR^{n}$.
\State Update $\bm{u}_{t+1} = \upp_{\cbap}(\bm{u}_{t}, \bm{x}_{t+1},\bm{f}_{t+1},\omega_{t+1}).$ \label{alg:step:update-u}
\State Increment $t \leftarrow t+1.$
\EndFor
\end{algorithmic}
\end{algorithm}

Let us give some intuition on the effect of projection onto $\mcC$. For a geometric intuition, it is easier to visualize the dynamics in $\mbR^{2}$.
Figure \ref{fig:projection-cone} illustrates the projection step $\pi_{\mcC}(\cdot)$ of $\cbap$. At a high level, from $\bm{u}_{t}$ to $\bm{u}_{t+1}$, an \emph{instantaneous payoff vector}
\[ \bm{v}_{t+1}=\omega_{t+1}\left(\langle\bm{f}_{t+1},\bm{x}_{t+1}\rangle/\kappa, - \bm{f}_{t+1}\right) \]
is first added to $\bm{u}_{t}$, and then the resulting vector $\bm{u}^{+}_{t}=\bm{u}_{t}+\bm{v}_{t+1}$ is projected onto $\mcC$. The projection $\pi_{\mcC}(\cdot)$ moves the vector $\bm{u}^{+}_{t}$ along the edges of the cone $\mcC^{\circ}$, preserving the (orthogonal) distance $d$ to $\mcC^{\circ}$.
Intuitively, from a game-theoretic perspective in the usual case where $\mcC=\mbR^{2}_{+}$, the projection eliminates the negative components of the payoffs, meaning that we do not remember ``negative regrets.''

\begin{figure}[hbt]
\centering
   \begin{subfigure}{0.4\textwidth}

\centering
         \includegraphics[width=1.0\linewidth]{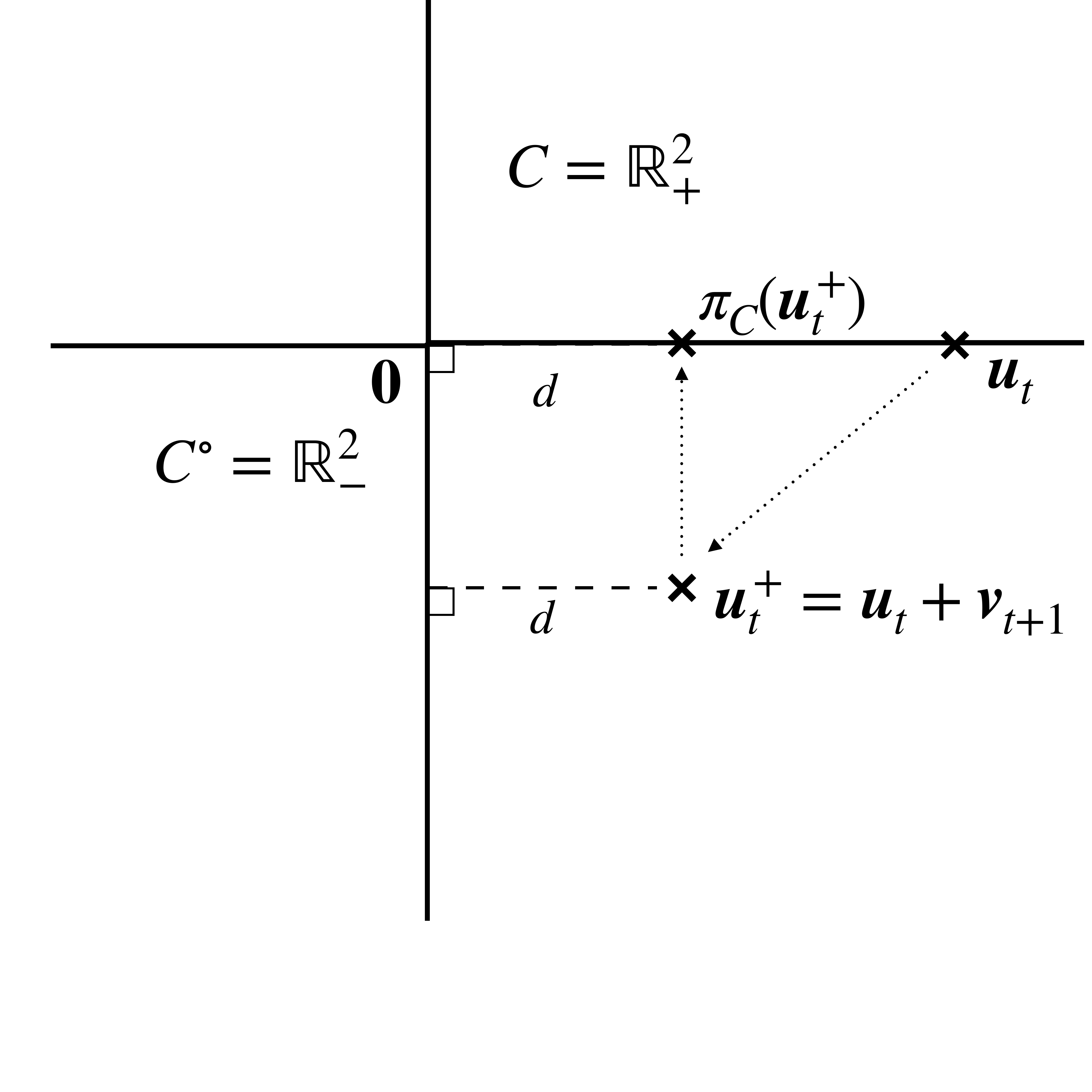}
  \end{subfigure}
     \begin{subfigure}{0.4\textwidth}
\centering
         \includegraphics[width=1.0\linewidth]{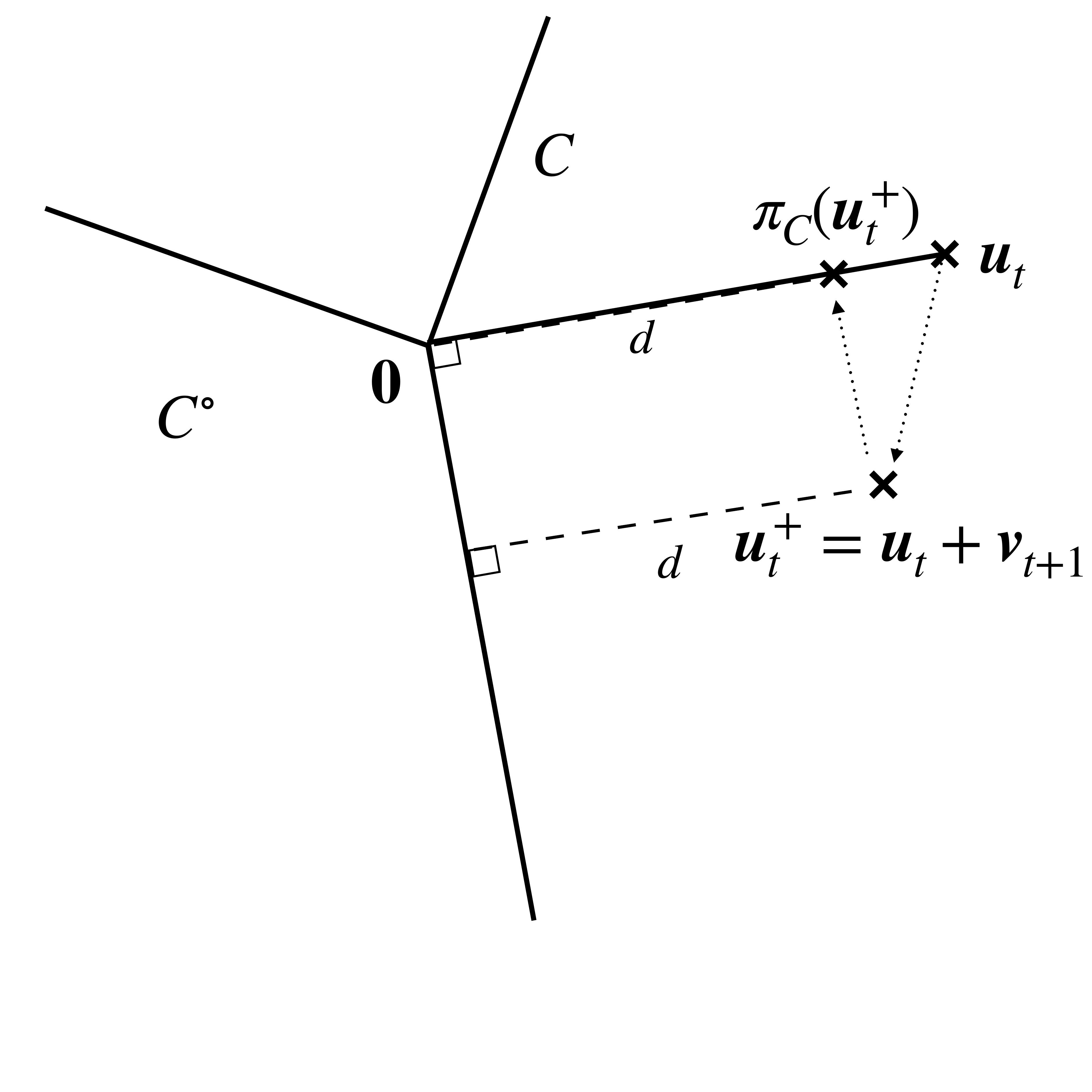}
  \end{subfigure}
  \caption{Illustration of $\pi_{\mcC}(\cdot)$ for $\mcC=\mbR_{+}^{2}$ (left-hand side) and $\mcC$ any cone in $\mbR^{2}$ (right-hand side).}
  \label{fig:projection-cone}
\end{figure}
Let us also note the difference between \cbap{} and  the algorithm introduced in \cite{abernethy2011blackwell}, which we have called \cba.  
\cba{} uses different \upp{} and \chp{} functions. 
In \cba{} the payoff update is defined as \[\upp_{\cba}(\bm{u}, \bm{x},\bm{f},\omega) = \bm{u} +  \omega \left( \langle \bm{f} ,\bm{x} \rangle /\kappa, 
- \bm{f} \right).\]
Note in particular the lack of projection as compared to \cbap{}, this is analogous to the difference between \rmm{} and \rmp. The $\chp_{\cba}$ function then requires a projection onto $\mcC$: \[\chp_{\cba}(\bm{u}) = \chp_{\cbap}\left(\pi_{\mcC}(\bm{u}) \right).\]
Based upon the analysis in \cite{blackwell1956analog},  \cite{abernethy2011blackwell} show that \cba{} with uniform weights (both on payoffs and decisions) guarantees $O(1/\sqrt{T})$ average regret.

\subsection{Regret bounds for \cba{} and \cbap}\label{sec:regret-bounds}
In this section we investigate the theoretical performance guarantees of \cba\ and \cbap\ when we vary the weights on decisions and payoffs. This is motivated by practical performance, where it has been observed in several other settings that increasing weights usually perform better~\citep{gao2021increasing,tammelin2015solving,brown2019solving}, and that \emph{alternating} update schemes are helpful~\citep{tammelin2015solving,kroer2020ieor8100}.
First, we show that \cba\ and \cbap\ are both compatible with varying weights $\left(\omega_{t}\right)_{t \geq 1}$, when those weights are used on both decisions and payoffs. 
Second, we show that \cbap{} is compatible with different weights $\left(\omega_{t}\right)_{t \geq 1}$ on payoffs and weights $\left(\theta{t}\right)_{t \geq 1}$ on decisions.

We start with the following theorem, which shows that \cba{} with weights on both decisions and payoffs is a no-regret algorithm.
This generalizes the result of \citep{abernethy2011blackwell}, which shows that \cba\ works for uniform weights.
\begin{theorem}\label{th:cba-linear-averaging-both}
Let $\left( \bm{x}_{t} \right)_{t \geq 1}$ be the sequence of decisions generated by \cba{} with payoff weights $\left( \omega_{t} \right)_{t \geq 1}$ and let $S_{t} = \sum_{\tau=1}^{t} \omega_{\tau}$ for any $t \geq 1$. 
Then
\[ \sum_{t=1}^{T} \omega_{t} \langle \bm{f}_{t},\bm{x}_{t} \rangle -  \min_{\bm{x} \in \mcX} \sum_{t=1}^{T} \omega_{t} \langle \bm{f}_{t},\bm{x} \rangle = O \left( \kappa \cdot  d(\bm{u}_{T},\mcC^{\circ}) \right).\]
Additionally,
\[ d(\bm{u}_{T},\mcC^{\circ}) = O \left(L \cdot \sqrt{ \sum_{t=1}^{T} \omega_{t}^{2}}\right).\]
Overall, the average regret is such that
\[\dfrac{\sum_{t=1}^{T} \omega_{t} \langle \bm{f}_{t},\bm{x}_{t} \rangle -  \min_{\bm{x} \in \mcX} \sum_{t=1}^{T} \omega_{t} \langle \bm{f}_{t},\bm{x} \rangle}{S_{T}} \leq \sqrt{2} \kappa L \frac{\sqrt{ \sum_{t=1}^{T} \omega_{t}^{2}}}{\sum_{t=1}^{T} \omega_{t} }.\]
\end{theorem}

The proof of Theorem \ref{th:cba-linear-averaging-both} uses the following facts from conic optimization. Several of these are direct consequences of Moreau's decomposition theorem. We provide proofs of all statements in Lemma \ref{lem:conic-opt} in Appendix \ref{app:proof-lemma-conic-opt}.

\begin{lemma}\label{lem:conic-opt}
Let $\mcC \subset \mbR^{n+1} $ be a closed convex cone and $\mcC^{\circ}$ its polar.
\begin{enumerate}
\item If $\bm{u} \in \mbR^{n+1}$, then $\bm{u} - \pi_{\mcC^{\circ}}(\bm{u}) = \pi_{\mcC}(\bm{u})$, $\langle \bm{u} - \pi_{\mcC^{\circ}}(\bm{u}),\pi_{\mcC^{\circ}}(\bm{u}) \rangle = 0,$ and $\| \bm{u} - \pi_{\mcC^{\circ}}(\bm{u}) \|_{2} \leq \| \bm{u} \|_{2}$. \label{lem:statement:u-minus-proj-in-C-polar}
\item If $\bm{u} \in \mbR^{n+1}$ then 
\[ d(\bm{u},\mcC)   = \max_{\bm{w} \in \mcC^{\circ} \cap B_{2}(1) } \langle\bm{u},\bm{w} \rangle,\]
\label{lem:statement:dist-to-cone}
where $B_{2}(1) = \{ \bm{w} \in \mbR^{n+1} \; \vert \; \| \bm{w} \|_{2} \leq 1\}$.
\item If $\bm{u} \in \mcC$, then $d(\bm{u},\mcC^{\circ}) = \| \bm{u} \|_{2}$. \label{lem:statement:norm2-C}
\item Assume that $\mcC = \textrm{cone}(\{ \kappa \} \times \mcX)$ with $\mcX \subset \mbR^{n}$ convex compact and $\kappa = \max_{\bm{x} \in \mcX} \| \bm{x} \|_{2}$. Then $\mcC^{\circ}$ is a closed convex cone.  Additionally, if $\bm{u} \in \mcC$ we have $-\bm{u} \in \mcC^{\circ}$. \label{lem:statement:negative-cone-in-polar}
\item \label{lem:statement:partial order}
Let us write $\leq_{\mcC^{\circ}}$ for the ordering induced by $\mcC^{\circ}:\bm{x} \leq_{\mcC^{\circ}} \bm{y} \iff \bm{y} - \bm{x} \in \mcC^{\circ}$.
Then 
\begin{align}
\bm{x} \leq_{\mcC^{\circ}} \bm{y}, \bm{x}' \leq_{\mcC^{\circ}} \bm{y}' & \Rightarrow \bm{x} + \bm{x}' \leq_{\mcC^{\circ}} \bm{y} + \bm{y}',  &\forall \; \bm{x},\bm{x}',\bm{y},\bm{y}' \in \mbR^{n+1}, \label{eq:partial-order-additivity}\\
\bm{x} + \bm{x}'  \leq_{\mcC^{\circ}} \bm{y} & \Rightarrow \bm{x} \leq_{\mcC^{\circ}} \bm{y}, &\forall \; \bm{x},\bm{y} \in \mbR^{n+1}, \forall \; \bm{x}' \in \mcC^{\circ}, \label{eq:partial-order-minus-x-prime}
\end{align}
\item \label{lem:statement:x-y-dist} Assume that $\bm{x} \leq_{\mcC^{\circ}} \bm{y}$ for $\bm{x},\bm{y} \in \mbR^{n+1}$. Then $d(\bm{y},\mcC^{\circ}) \leq \| \bm{x} \|_{2}$.
\end{enumerate}
\end{lemma}
We are now ready to prove Theorem \ref{th:cba-linear-averaging-both}.
\begin{proof}[Proof of Theorem \ref{th:cba-linear-averaging-both}] The proof proceeds in two steps.
We start by proving 
\[ \sum_{t=1}^{T} \omega_{t} \langle \bm{f}_{t},\bm{x}_{t} \rangle -  \min_{\bm{x} \in \mcX} \sum_{t=1}^{T} \omega_{t} \langle \bm{f}_{t},\bm{x} \rangle = O \left(\kappa \cdot  d(\bm{u}_{T},\mcC^{\circ}) \right).\]
We have
\begin{align}
d(\bm{u}_{T},\mcC^{\circ}) &  = \max_{\bm{w} \in \textrm{cone}(\{ \kappa \} \times \mcX) \bigcap B_{2}(1) } \langle \sum_{t=1}^{T} \omega_{t} \bm{v}_{t},\bm{w} \rangle \label{eq:proof-step-0} \\
 & \geq  \max_{\bm{x} \in \mcX} \langle  \sum_{t=1}^{T}  \omega_{t} \bm{v}_{t},\dfrac{(\kappa,\bm{x})}{\| (\kappa,\bm{x}) \|_{2}} \rangle \nonumber \\
 & \geq \max_{\bm{x} \in \mcX} \dfrac{\sum_{t=1}^{T} \omega_{t} \langle \bm{f}_{t},\bm{x}_{t} \rangle - \sum_{t=1}^{T} \omega_{t} \langle \bm{f}_{t},\bm{x} \rangle}{ \| (\kappa,\bm{x})\|_{2}} \label{eq:proof-lin-avg-step-0},
\end{align}
where \eqref{eq:proof-step-0} follows from Statement \ref{lem:statement:dist-to-cone} in Lemma \ref{lem:conic-opt}, and \eqref{eq:proof-lin-avg-step-0} follows from \cba{} maintaining
\[ \bm{u}_{t} = \left( \sum_{\tau=1}^{t} \omega_{\tau} \frac{ \langle \bm{f}_{\tau},  \bm{x}_{\tau}\rangle}{\kappa} ,  - \sum_{\tau=1}^{t} \omega_{\tau} \bm{f}_{\tau} \right), \forall \; t \geq 1.\]
Since $\| (\kappa,\bm{x})\|_{2} \leq \sqrt{2}\kappa$, we can conclude that
\[ \sqrt{2} \kappa d(\bm{u}_{T},\mcC^{\circ}) \geq \sum_{t=1}^{T} \omega_{t} \langle \bm{f}_{t},\bm{x}_{t} \rangle -  \min_{\bm{x} \in \mcX} \sum_{t=1}^{T} \omega_{t} \langle \bm{f}_{t},\bm{x} \rangle.\]
We now prove that \[ d(\bm{u}_{T},\mcC^{\circ}) \leq L \sqrt{\sum_{\tau=1}^{T} \omega_{\tau}^{2}}.\]
We have
\begin{align}
d(\bm{u}_{t+1},\mcC^{\circ})^{2} & = \min_{\bm{z} \in \mcC^{\circ}} \| \bm{u}_{t+1} - \bm{z} \|_{2}^{2} \nonumber \\
& \leq \| \bm{u}_{t+1} - \pi_{\mcC^{\circ}}(\bm{u}_{t}) \|_{2}^{2} \nonumber \\
& \leq \|  \bm{u}_{t} +\omega_{t+1} \bm{v}_{t+1} - \pi_{\mcC^{\circ}}(\bm{u}_{t})\|_{2}^{2} \nonumber \\
& \leq   \| \bm{u}_{t} - \pi_{\mcC^{\circ}}(\bm{u}_{t}) \|_{2}^{2} + \omega_{t+1}^{2} \| \bm{v}_{t+1} \|_{2}^{2} \nonumber \\
& + 2 \omega_{t+1} \langle \bm{u}_{t} - \pi_{\mcC^{\circ}}(\bm{u}_{t}), \bm{v}_{t+1}  \rangle \nonumber \\
& \leq  \| \bm{u}_{t} - \pi_{\mcC^{\circ}}(\bm{u}_{t}) \|_{2}^{2} + \omega_{t+1}^{2} \| \bm{v}_{t+1}\|_{2}^{2}, \label{eq:proof-lin-avg-step-1}
\end{align}
where \eqref{eq:proof-lin-avg-step-1} follows from
\begin{equation}\label{eq:blackwell-forcing}
\langle \bm{u}_{t} - \pi_{\mcC^{\circ}}(\bm{u}_{t}), \bm{v}_{t+1}\rangle = 0.
\end{equation} 
This is one of the crucial components of Blackwell's approachability framework: the current decision is chosen to force the next payoff to lie in the hyperplane generated by projecting the aggregate payoffs onto the target set. To see this, first note that $\bm{u}_{t} - \pi_{\mcC^{\circ}}(\bm{u}) = \pi_{\mcC}(\bm{u}_{t})$. Let us write $\bm{\pi} = \left(\tilde{\pi},\hat{\bm{\pi}} \right) =  \pi_{\mcC}(\bm{u}_{t})$. Note that by definition, $\bm{x}_{t+1} =  (\kappa/\tilde{\pi})\hat{\bm{\pi}}$, and $\bm{v}_{t+1} = \left(\langle \bm{f}_{t+1},\bm{x}_{t+1} \rangle / \kappa, - \bm{f}_{t+1} \right)$.
Therefore,
\begin{align*}
\langle \bm{u}_{t} - \pi_{\mcC^{\circ}}(\bm{u}_{t}), \bm{v}_{t+1} \rangle & = \langle \bm{\pi}, \bm{v}_{t+1} \rangle \\
& = \langle \left(\tilde{\pi},\hat{\bm{\pi}} \right), \left(\langle \bm{f}_{t+1},\bm{x}_{t+1} \rangle / \kappa, - \bm{f}_{t+1}  \right) \rangle \\
&  = \langle \left(\tilde{\pi},\hat{\bm{\pi}} \right), \left(\langle \bm{f}_{t+1},(\kappa/\tilde{\pi})\hat{\bm{\pi}} \rangle / \kappa, - \bm{f}_{t+1} \right) \rangle \\
& = \langle \hat{\bm{\pi}}, \bm{f}_{t+1} \rangle - \langle \hat{\bm{\pi}}, \bm{f}_{t+1} \rangle \\
& = 0.
\end{align*}
Next,
recall that $d(\bm{u}_{t},\mcC^{\circ})^{2}=\| \bm{u}_{t} - \pi_{\mcC^{\circ}}(\bm{u}_{t}) \|_{2}^{2}$. Applying \eqref{eq:proof-lin-avg-step-1} inductively we obtain
\[ d(\bm{u}_{t},\mcC^{\circ})^{2} \leq  \sum_{\tau=1}^{t}  \omega_{\tau}^{2} \| \bm{v}_{\tau} \|_{2}^{2} \leq  L^2 \cdot \sum_{\tau=1}^{t}  \omega_{\tau}^{2},\]
where the last inequality follows from the definition of $\bm{v}_{t}$ and $L$.
\end{proof}
 In the next theorem, we show a result that may seem surprising: 
 \cbap\ allows us to use two separate and different weighting schemes for the decisions in the regret definition and the aggregate payoffs.
This result is analogous to the fact that for the simplex case, \rmp\ is compatible with polynomial averaging schemes on the decision while using constant weights on the aggregate payoffs~\citep{tammelin2015solving,brown2019solving}.

\begin{theorem}\label{th:cbap-linear-averaging-only-policy}
 Consider $\left(\bm{x}_{t} \right)_{t \geq 1}$ generated by \cbap{} with aggregate payoff weights $\left(\omega_{t}\right)_{t \geq 1}$, when regret is measured using decision weights $\left(\theta_{t} \right)_{t \geq 1}$, and $S_{T} = \sum_{t=1}^{T} \theta_{t}$. Assume that $\frac{\theta_{t+1}}{\theta_{t}} \geq \frac{\omega_{t+1}}{\omega_{t}}, \forall \; t \geq 1$.
Then 
\[ \dfrac{\sum_{t=1}^{T} \theta_{t} \langle \bm{f}_{t},\bm{x}_{t} \rangle -  \min_{\bm{x} \in \mcX} \sum_{t=1}^{T} \theta_{t} \langle \bm{f}_{t},\bm{x} \rangle}{S_{T}} \leq \sqrt{2} \kappa L \frac{\theta_{T}}{\omega_{T}} \frac{\sqrt{\sum_{t=1}^{T} \omega^{2}_{t}}}{\sum_{t=1}^{T} \theta_{t}}.\]
\end{theorem}
Our proof heavily relies on the sequence of payoffs belonging to the cone $\mcC$ at every iteration ($\bm{u}_{t} \in \mcC, \forall \; t \geq 1$), and for this reason it does not extend to \cba. We also note that the use of conic optimization somewhat simplifies the argument compared to the proof that \rmp{} is compatible with polynomial averaging on decisions and uniform weights on payoffs. 
\begin{proof}[Proof of Theorem \ref{th:cbap-linear-averaging-only-policy}]
Recall that $\bm{v}_{t} = \left(\langle \bm{f}_{t},\bm{x}_{t} \rangle / \kappa, - \bm{f}_{t} \rangle \right)$. By construction and following the same argument as for the proof of Theorem \ref{th:cba-linear-averaging-both}, we have
\begin{align}
\sum_{t=1}^{T} \theta_{t} \langle \bm{f}_{t},\bm{x}_{t} \rangle -  \min_{\bm{x} \in \mcX} \sum_{t=1}^{T} \theta_{t} \langle \bm{f}_{t},\bm{x} \rangle \leq \sqrt{2} \kappa \cdot d \left(\sum_{t=1}^{T} \theta_{t} \bm{v}_{t},\mcC^{\circ} \right) . 
\end{align}
Additionally,  we always have
\begin{equation}\label{eq:v-u-C-polar}
 \omega_{t+1}\bm{v}_{t+1} \geq_{\mcC^{\circ}} \bm{u}_{t+1} - \bm{u}_{t}.
\end{equation}
This is because
\begin{align*}
\omega_{t+1}\bm{v}_{t+1} - \bm{u}_{t+1} + \bm{u}_{t} & = \bm{u}_{t} +\omega_{t+1}\bm{v}_{t+1} - \bm{u}_{t+1} \\
& = \bm{u}_{t} +\omega_{t+1}\bm{v}_{t+1} -\pi_{\mcC}\left(\bm{u}_{t} +\omega_{t+1}\bm{v}_{t+1}\right)\\
& = \pi_{\mcC^{\circ}}\left(\bm{u}_{t} +\omega_{t+1}\bm{v}_{t+1}\right) \in \mcC^{\circ}.
\end{align*}
Therefore, multiplying \eqref{eq:v-u-C-polar} by $\theta_{t+1}$ and dividing by $\omega_{t+1}$, we obtain
\[ \theta_{t+1}\bm{v}_{t+1} \geq_{\mcC^{\circ}} \frac{\theta_{t+1}}{\omega_{t+1}}(\bm{u}_{t+1} -  \bm{u}_{t}).\]
Reformulating the right-hand side we obtain
\[ \theta_{t+1} \bm{v}_{t+1} \geq_{\mcC^{\circ}} \frac{\theta_{t+1}}{\omega_{t+1}}\bm{u}_{t+1} - \frac{\theta_{t}}{\omega_{t}} \bm{u}_{t} - \left(\frac{\theta_{t+1}}{\omega_{t+1}}-\frac{\theta_{t}}{\omega_{t}}\right)\bm{u}_{t}.\]
Summing up the previous inequalities from $t=1$ to $t=T-1$ and using $\bm{u}_{1}=\bm{v}_{1}$ we obtain
\[ \sum_{t=1}^{T} \theta_{t} \bm{v}_{t} \geq_{\mcC^{\circ}} \frac{\theta_{T}}{\omega_{T}}\bm{u}_{T} - \sum_{t=1}^{T-1} \left(\frac{\theta_{t+1}}{\omega_{t+1}}-\frac{\theta_{t}}{\omega_{t}}\right) \bm{u}_{t}.\]
Note that $\sum_{t=1}^{T-1} \left(\frac{\theta_{t+1}}{\omega_{t+1}}-\frac{\theta_{t}}{\omega_{t}}\right) \bm{u}_{t} \in \mcC$, because $\frac{\theta_{t+1}}{\omega_{t+1}}-\frac{\theta_{t}}{\omega_{t}} \geq 0$. Therefore, Statement \ref{lem:statement:negative-cone-in-polar} in Lemma \ref{lem:conic-opt} shows that $- \sum_{t=1}^{T-1} \left(\frac{\theta_{t+1}}{\omega_{t+1}}-\frac{\theta_{t}}{\omega_{t}} \right) \bm{u}_{t} \in \mcC^{\circ}$.
Now, by applying \eqref{eq:partial-order-minus-x-prime} in Statement~\ref{lem:statement:partial order} of Lemma \ref{lem:conic-opt},  we have
\[ \sum_{t=1}^{T} \theta_{t} \bm{v}_{t} \geq_{\mcC^{\circ}} \frac{\theta_{T}}{\omega_{T}}\bm{u}_{T} - \sum_{t=1}^{T-1} \left(\frac{\theta_{t+1}}{\omega_{t+1}}-\frac{\theta_{t}}{\omega_{t}}\right) \bm{u}_{t} \Rightarrow \sum_{t=1}^{T} \theta_{t} \bm{v}_{t} \geq_{\mcC^{\circ}} \frac{\theta_{T}}{\omega_{T}} \bm{u}_{T}. \]
Since $ \frac{\theta_{T}}{\omega_{T}} \bm{u}_{T} \in \mcC$, Statement \ref{lem:statement:x-y-dist} shows that
\begin{align}
   d \left(\sum_{t=1}^{T} \theta_{t} \bm{v}_{t} ,\mcC^{\circ} \right) \leq \bigg\| \frac{\theta_{T}}{\omega_{T}} \bm{u}_{T} \bigg\|_{2}.
   \label{eq:iterate weighted payoff less than ut}
\end{align}
By construction $\bm{u}_{T}$ is the sequence of aggregated payoffs generated by \cbap{} with weights $\left(\omega_{t}\right)_{t \geq 1}$. We now show that $d(\bm{u}_{T},\mcC^{\circ}) = \| \bm{u}_{T} \|_{2} =  O\left(L\sqrt{\sum_{t=1}^{T} \omega^{2}_{t}}\right)$.
We have
\begin{align}
 \| \bm{u}_{t+1} \|_{2}^{2}
& = \| \pi_{\mcC}\left( \bm{u}_{t} + \omega_{t+1} \bm{v}_{t+1} \right) \|_{2}^{2} \nonumber \\
& \leq \| \bm{u}_{t} + \omega_{t+1} \bm{v}_{t+1} \|_{2}^{2} \label{eq:norm-proj-smaller-than-norm}
\end{align}
where \eqref{eq:norm-proj-smaller-than-norm} follows from Statement \ref{lem:statement:u-minus-proj-in-C-polar} in Lemma \ref{lem:conic-opt}.
Therefore,
\begin{align*}
\| \bm{u}_{t+1}\|_{2}^{2} & \leq  \| \bm{u}_{t} \|^{2}_{2} + \omega_{t+1}^{2} \| \bm{v}_{t+1} \|^{2}_{2} + 2 \omega_{t+1} \langle \bm{u}_{t},\bm{v}_{t+1} \rangle .
\end{align*}
By construction and for the same reason as for \eqref{eq:blackwell-forcing}, $\langle \bm{u}_{t},\bm{v}_{t+1} \rangle=0$. 
Therefore, we have the recursion
\[ \| \bm{u}_{t+1}\|_{2}^{2} \leq \| \bm{u}_{t} \|^{2}_{2} + \omega_{t+1}^{2} \| \bm{v}_{t+1} \|^{2}_{2}.\]
By telescoping the inequality above we obtain
\[ \| \bm{u}_{t}\|_{2}^{2} \leq  \sum_{\tau=1}^{t} \omega_{\tau}^{2} \| \bm{v}_{\tau} \|^{2}_{2} .\]
By definition of $L$, we conclude that
\[ \| \bm{u}_{T} \|_{2} \leq L\sqrt{\sum_{t=1}^{T} \omega^{2}_{t}}.\]
Therefore, by \eqref{eq:iterate weighted payoff less than ut},  $d(\sum_{t=1}^{T} \theta_{t} \bm{v}_{t} ,\mcC^{\circ}) \leq L \frac{\theta_{T}}{\omega_{T}} \sqrt{\sum_{t=1}^{T} \omega^{2}_{t}}  .$
This shows that 
\[ \dfrac{\sum_{t=1}^{T} \theta_{t} \langle \bm{f}_{t},\bm{x}_{t} \rangle -  \min_{\bm{x} \in \mcX} \sum_{t=1}^{T} \theta_{t} \langle \bm{f}_{t},\bm{x} \rangle}{S_{T}} \leq \sqrt{2} \kappa L \frac{\theta_{T}}{\omega_{T}} \frac{\sqrt{\sum_{t=1}^{T} \omega^{2}_{t}}}{\sum_{t=1}^{T} \theta_{t}}.\]
\end{proof}

\subsection{Convergence bounds for saddle-point problems}\label{sec:regret-bound-duality-gap}
In this section, we show how the regret bounds from the previous section translate into convergence rates for solving convex-concave saddle-point problems in the repeated game framework.
In particular, the following theorem gives the convergence rate of \cbap{} and \cba{} for solving saddle-point problems of the form ~\eqref{eq:spp}, based on our bounds on the regret of each player under various weighting schemes. The proof is in Appendix \ref{app:proof-folk-final}.
\begin{theorem}\label{th:folk-final}
Let $L = \max \{ L_{x},L_{y}\}$ defined in \eqref{eq:definition-Lx-Lx} and $\kappa = \max \{ \max \{\| \bm{x} \|_{2},\| \bm{y} \|_{2}\} \; \vert \; \bm{x} \in \mcX,\bm{y} \in \mcY\}$.
\begin{enumerate}
\item 
Let $\left( \bar{\bm{x}}_{T},\bar{\bm{y}}_{T} \right) = \sum_{t=1}^{T} \omega_{t} \left(\bm{x}_{t}, \bm{y}_{t} \right)/S_{T},$ where $\left(\bm{x}_{t} \right)_{t \geq 1}, \left(\bm{y}_{t}\right)_{t \geq 1}$ are generated by the repeated game framework with \cba{} with weights $\left(\omega_{\tau}\right)_{t \geq 1}$ on both decisions and payoffs and $S_{T} = \sum_{t=1}^{T} \omega_{t}$. Assume that $\omega_{t} = t^{p}, \forall t \geq 1$.
Then 
\[ \max_{\bm{y} \in \mcY} F(\bar{\bm{x}}_{T},\bm{y}) - \min_{\bm{x} \in \mcX} F(\bm{x},\bar{\bm{y}}_{T}) = O \left(\frac{\kappa L \sqrt{p+1} }{\sqrt{T}} \right). \]
\item 
Let $p,q \in \mbN$ with $q \geq p$. Let $\left( \bar{\bm{x}}_{T},\bar{\bm{y}}_{T} \right) = \sum_{t=1}^{T} \theta_{t} \left(\bm{x}_{t}, \bm{y}_{t} \right)/S_{T},$ where $\left(\bm{x}_{t} \right)_{t \geq 1}, \left(\bm{y}_{t}\right)_{t \geq 1}$ are generated by the repeated game framework with \cbap{} with aggregate payoff weights $\left(\omega_{t}\right)_{t \geq 1}$, and decision weights $\left(\theta_{t}\right)_{t \geq 1}$  and $S_{T} = \sum_{t=1}^{T} \theta_{t}$. Assume that $\theta_{t} = t^{q}, \omega_{t} = t^p, \forall t \geq 1$.
Then 
\[ \max_{\bm{y} \in \mcY} F(\bar{\bm{x}}_{T},\bm{y}) - \min_{\bm{x} \in \mcX} F(\bm{x},\bar{\bm{y}}_{T}) = O \left( \frac{ \kappa L (q+1) }{\sqrt{p+1}\sqrt{T}} \right). \]
\end{enumerate}
\end{theorem}
We note that larger weights lead to slightly worse worst-case convergence guarantees. In contrast to this, we will see in our numerical simulations that the strongest empirical performances for \cbap{} are obtained for $q=1,p=0$, i.e., linear weights on the decisions and uniform weights on the payoffs.

Let us compare our bounds with the regret bounds of classical first-order methods. We  consider $p,q=0$.
\cba{} and \cbap{} achieve $O\left(\kappa L /\sqrt{T}\right)$ average regret, whereas online mirror descent (OMD)~\citep{nemirovsky1983problem,BenTal-Nemirovski} and follow-the-regularized-leader (FTRL)~\citep{abernethy2009competing,mcmahan2011follow} achieve $O\left(\Omega L /\sqrt{T}\right)$ average regret, where $\Omega = \max \{ \| \bm{x} - \bm{x}'\|_{2} \rvert \bm{x},\bm{x}' \in \mcX\}.$ 
We can always recenter $\mcX$ to contain $\bm{0}$, in which case the bounds for OMD/FTRL and \cbap{} are equivalent since $\kappa \leq \Omega \leq 2 \kappa$. 
The bound on the average regret for \emph{optimistic} OMD (OOMD, \cite{chiang2012online}) and \emph{optimistic} FTRL (OFTRL, \cite{rakhlin2013online})  is $O\left(\Omega^{2} L /T\right)$ in the repeated game framework, a priori better than the bound for \cbap{} as regards the number of iterations $T$. Nonetheless, we will see in Section \ref{sec:simu} that the empirical performance of \cbap{} is better than that of $O(1/T)$ methods.  
A similar situation occurs for \rmp{} compared to OOMD and OFTRL for solving extensive-form games such as poker~\citep{farina2019optimistic,kroer2018faster}.

\subsection{Improved convergence bounds using alternation}\label{sec:regret-vound-duality-gap-alternating-updates}
Alternation is a simple variation of the repeated game framework from Section \ref{sec:game-setup}. 
Alternation is known to lead to significant speedup for \rmp{}~\citep{tammelin2015solving}, and we will observe in our simulations (Section~\ref{sec:simu}) that this holds for \cbap\ as well.
In the repeated game framework with alternation, at iteration $t$, the second player is provided with the decision $\bm{x}_{t}$ of the first player for iteration $t$. 
Because alternation is defined the same way for both \cba{} and \cbap, we omit the subscripts in \chp{} and \upp{}.
In particular, at iteration $t$ of the repeated game framework with alternation, the players choose $\bm{x}_{t}$ and $\bm{y}_{t}$ as follows:
 \begin{enumerate}
 \item Both players start with aggregate payoffs $\bm{u}^{x}_{t-1},\bm{u}^{y}_{t-1}$.
 \item The first player chooses a decision $\bm{x}_{t}$ based on $\bm{u}^{x}_{t-1}$:
 \[\bm{x}_{t} = \chp(\bm{u}^{x}_{t-1}).\]
 \item For $\bm{g}_{t-1} = \partial_{y} F(\bm{x}_{t},\bm{y}_{t-1})$, the second player updates its aggregate payoff:
 \[  \bm{u}^{y}_{t} = \upp\left(\bm{u}_{t-1}^{y}, \bm{y}_{t-1},\bm{g}_{t-1},\omega_{t}\right).\]
 \item The second player chooses a decision $\bm{y}_{t}$ based on $\bm{u}^{y}_{t}$:  \[\bm{y}_{t} = \chp(\bm{u}^{y}_{t}).\]
 \item For $\bm{f}_{t} = \partial_{x} F(\bm{x}_{t},\bm{y}_{t} )$, the first player updates its aggregate payoff:
\[  \bm{u}^{x}_{t} = \upp\left(\bm{u}_{t-1}^{x}, \bm{x}_{t},\bm{f}_{t},\omega_{t}\right).\]
 \end{enumerate}

Recall that we use the repeated game framework to solve \eqref{eq:spp} because we can bound the duality gap by the sum of the average regrets of each player using Theorem~\ref{th:-folk-theorem}. It is known that in the repeated game framework with alternation, it is possible to construct decisions such that Theorem \ref{th:-folk-theorem} fails to hold, because of the mismatch in the sequences of decisions of the players~\citep{farina2019online}. 
That said, it was later shown that a modified version of Theorem~\ref{th:-folk-theorem} holds~\citep{burch2019revisiting}.
Here we state a more general version of that result, which was first shown in a set of lecture notes~\citep{kroer2020ieor8100}. In particular, the following bound holds on the duality gap. For the sake of completeness, we provide the proof in Appendix \ref{app:proof-folk-alternation}.
 \begin{theorem}\label{th:folk-theorem-alternation}
Consider some weights $\left(\theta_{t}\right)_{t \geq 1}$  and $S_{T} = \sum_{t=1}^{T} \theta_{t+1}$. 
Let $\left( \bar{\bm{x}}_{T},\bar{\bm{y}}_{T} \right) = \sum_{t=1}^{T} \theta_{t+1} \left(\bm{x}_{t+1}, \bm{y}_{t} \right)/S_{T},$ where $\left(\bm{x}_{t} \right)_{t \geq 1}, \left(\bm{y}_{t}\right)_{t \geq 1}$ are generated by the repeated game framework with alternation.
Then 
\begin{align*}
\max_{\bm{y} \in \mcY} F(\bar{\bm{x}}_{T},\bm{y}) - \min_{\bm{x} \in \mcX} F(\bm{x},\bar{\bm{y}}_{T}) & \leq   \frac{1}{S_{T}} \left( \max_{\bm{y} \in \mcY}  \theta_{t+1} \sum_{t=1}^{T} \langle \bm g_t ,\bm{y} \rangle - \sum_{t=1}^{T}  \theta_{t+1} \langle \bm g_t,\bm{y}_{t} \rangle  \right)\\
& +  \frac{1}{S_{T}}\left( \sum_{t=1}^{T}  \theta_{t+1} \langle \bm f_t, \bm x_t \rangle - \min_{\bm{x} \in \mcX} \sum_{t=1}^{T} \theta_{t+1} \langle \bm f_t, \bm x \rangle \right)\\
& + \frac{1}{S_{T}} \left(\sum_{t=1}^{T} \theta_{t+1} \left(F(\bm{x}_{t+1},\bm{y}_{t}) - F(\bm{x}_{t},\bm{y}_{t})\right) \right).
\end{align*}
\end{theorem}
From Theorem \ref{th:folk-theorem-alternation}, we see that alternation guarantees convergence to a solution of \eqref{eq:spp}, if 
\begin{equation}\label{eq:alternation-is-improving}
\sum_{t=1}^{t+1} \theta_{t+1} \left(F(\bm{x}_{t+1},\bm{y}_{t}) - F(\bm{x}_{t},\bm{y}_{t})\right) \leq 0.
\end{equation}
In the framework of \rmm{} and \rmp, we have $\mcX = \Delta(n), \mcY = \Delta(m)$ and the objective function is bilinear. In this case, it is shown in \cite{burch2019revisiting} that \eqref{eq:alternation-is-improving} holds. In particular, for any $t \in [T]$, it holds that $F(\bm{x}_{t+1},\bm{y}_{t}) - F(\bm{x}_{t},\bm{y}_{t}) \leq 0$.
We provide the following stronger result for \cbap{} in the case of an objective function $F$ that is linear in one of the two variables, with any convex compact decision sets $\mcX$ and $\mcY$. The proof is presented in Appendix \ref{app:proof-alternation-is-improving}.
\begin{theorem}\label{th:alternation-works-bilinear-case}
Assume that $\left(\bm{x},\bm{y} \right) \mapsto F\left(\bm{x},\bm{y} \right)$ is linear in $\bm{x}$.
\begin{enumerate}
\item In the framework of Theorem \ref{th:folk-theorem-alternation}, suppose that $\left(\bm{x}_{t}\right)_{t \geq 1},\left(\bm{y}_{t}\right)_{t \geq 1}$ are generated by \cbap{} with weights $\left(\omega_{t}\right)_{t \geq 1}$ on the payoffs. We have, for $t \geq 1$,
\[F(\bm{x}_{t+1},\bm{y}_{t})- F(\bm{x}_{t},\bm{y}_{t}) \leq - \frac{\kappa}{\omega_{t} \cdot \|\bm{u}_{t}^{x} \|_{\infty}} \| \bm{u}_{t}^{x} - \bm{u}_{t-1}^{x} \|_{2}^{2} .\]
\item In the framework of Theorem \ref{th:folk-theorem-alternation}, suppose that $\left(\bm{x}_{t}\right)_{t \geq 1},\left(\bm{y}_{t}\right)_{t \geq 1}$ are generated by \cba{} with weights $\left(\omega_{t}\right)_{t \geq 1}$ on the payoffs. We have, for $t \geq 1$,
\[F(\bm{x}_{t+1},\bm{y}_{t}) -  F(\bm{x}_{t},\bm{y}_{t}) \leq - \frac{\kappa}{ \omega_{t} \cdot  \|\pi_{\mcC}\left(\bm{u}_{t}^{x}\right) \|_{\infty}} \| \pi_{\mcC}\left(\bm{u}_{t}^{x}\right) - \pi_{\mcC}\left(\bm{u}_{t-1}^{x}\right) \|_{2}^{2} .\]
\end{enumerate}
\end{theorem}
Note that our results in Theorem \ref{th:alternation-works-bilinear-case} for \cba{} and \cbap{} improve upon the analogous results for \rmm{} and \rmp{}~\citep{burch2019revisiting}, because Theorem \ref{th:alternation-works-bilinear-case} guarantees a strict improvement from alternation, where \citep{burch2019revisiting} only show that it does not hurt.
Secondly, their result is for the case of a bilinear objective function, whereas we only require linearity in one of the variables. 
Our assumption that the objective function  is linear in one of the decision variable is satisfied for many important decision problems, e.g., markov decision processes, distributionally robust logistic regression, and matrix games, as we will see in our simulations in Section \ref{sec:simu}.
\section{Efficient implementations of CBA}\label{sec:efficient-implementation}
We now turn to efficiently implementing \cba{} and \cbap.
The main bottleneck of both \cbap{} and \cba{} is to efficiently compute $\pi_{\mcC}(\bm{u})$, the orthogonal projection of a vector $\bm{u}$ on the cone  $\mcC=\textrm{cone}(\{\kappa\} \times \mcX)$:
\begin{equation}\label{eq:projection-onto-C}
\pi_{\mcC}(\bm{u}) \in \arg \min_{\bm{y} \in\mcC} \| \bm{y} - \bm{u} \|_{2}^{2}.
\end{equation}
Note that this issue is not discussed in \cite{abernethy2011blackwell}, which do not provide an efficient implementation of \cba.
In this section, we show how to efficiently solve \eqref{eq:projection-onto-C} for many important decision sets $\mcX$. One of the critical components of our proofs is \textit{Moreau's Decomposition Theorem}~\citep{combettes2013moreau} (Statement \ref{lem:statement:u-minus-proj-in-C-polar} in Lemma \ref{lem:conic-opt}), which states that $\pi_{\mcC}(\bm{u})$ can be recovered from $\pi_{\mcC^{\circ}}(\bm{u})$ and vice versa, because for any convex cone $\mcC$, we have $\pi_{\mcC}(\bm{u})+\pi_{\mcC^{\circ}}(\bm{u})=\bm{u}.$ All the proofs for this section are presented in Appendices \ref{app:efficient-implementation}.

\subsection{Simplex}\label{sec:projection-simplex}
Assume that $\mcX = \Delta(n)$. This setting is standard for matrix games. 
It is also used for extensive-form games, because CFR decomposes regret minimization over the tree-like decision space into a set of local regret minimizations over simplexes~\citep{zinkevich2007regret}. 
In the game setting, $n$ is the number of actions of a player and $\bm{x} \in \Delta(n)$ represents a randomized strategy. 
When $\mcX = \Delta(n)$, we show that $\pi_{\mcC^{\circ}}(\bm{u})$ can be computed in $O(n\log(n))$ using a sorting trick similar to that for the standard simplex projection, and therefore $\pi_{\mcC}(\bm{u})$ can be computed in $O\left(n\log(n)\right)$ using Moreau's decomposition.
In particular, we provide the following closed-form expression for the polar cone $\mcC^{\circ}$.
\begin{lemma}\label{lem:simplex-C-polar} 
Let $\mcC = \textrm{cone}\left(\{1\} \times \Delta(n)\right)$. Then $ \mcC^{\circ} = \{ \left(\tilde{y},\hat{\bm{y}} \right) \in \mbR^{n+1} \; \vert \; \max_{i \in [n]} \hat{y}_{i} \leq - \tilde{y} \}.$
\end{lemma}

Therefore, computing $\pi_{\mcC^{\circ}}(\bm{u})$ is equivalent to solving
\begin{equation}\label{eq:projection-onto-C-polar-simplex}
\min \{ ( \tilde{y}-\tilde{u} )^{2} +  \| \hat{\bm{y}}-\hat{\bm{u}} \|_{2}^{2}  \; \vert \; (\tilde{y},\hat{\bm{y}}) \in \mbR^{n+1}, \max_{i \in [n]} \hat{y}_{i} \leq - \tilde{y} \}. 
\end{equation}
We prove  the following proposition in Appendix \ref{app:efficient-implementation}.
\begin{proposition}\label{prop:proj-simplex}
Let $\mcX = \Delta(n)$.
An optimal solution $\pi_{\mcC^{\circ}}(\bm{u})$ to \eqref{eq:projection-onto-C-polar-simplex} can be computed in $O(n \log(n))$ arithmetic operations. Therefore, $\pi_{\mcC}(\bm{u})$ can be computed in $O\left(n\log(n)\right)$ arithmetic operations.
\end{proposition} 

 \subsection{$\ell_{p}$ balls}\label{sec:projection-ell-p-balls}
For $p  \geq 1$ and $p=\infty$, we consider the $\ell_p$ balls $\mcX = \{ \bm{x} \in \mbR^{n} \; \vert \; \| \bm{x} \|_{p} \leq 1\}$. This type of decision set appears in many problems in optimization, including robust optimization~\citep{ben2015oracle}, distributionally robust logistic regression~\citep{namkoong2016stochastic}, $\ell_{\infty}$ regression \citep{sidford2018coordinate} and saddle-point reformulation of Markov Decision Processes~\citep{jin2020efficiently}. We first reformulate the cones $\mcC$ and $\mcC^{\circ}$. Recall that $\kappa = \max \{ \| \bm{x} \|_{2} \; \vert \; \bm{x} \in \mcX\}$. 
\begin{lemma}\label{lem:ball-p-C-polar}
Let $\mcX = \{ \bm{x} \in \mbR^{n} \; \vert \; \| \bm{x} \|_{p} \leq 1\},$ with $p  \geq 1$ or $p=\infty$. Let $q \in \mbR \bigcup \{+ \infty\}$ be such that $1/p + 1/q =1$. Then 
\begin{align*}
\mcC & = \{ (\tilde{y},\bm{y}) \in \mbR \times \mbR^{n} \; \vert \; \| \bm{y} \|_{p} \leq  \tilde{y}/\kappa \}, \\
\mcC^{\circ} & = \{ (\tilde{y},\bm{y}) \in \mbR \times \mbR^{n} \; \vert \; \| \bm{y} \|_{q} \leq - \kappa \tilde{y} \}.
\end{align*}
\end{lemma}
Based on Lemma \ref{lem:ball-p-C-polar}, we can prove the following propositions.
\begin{proposition}\label{prop:proj-ball-1-C-polar}
Let $\mcX = \{ \bm{x} \in \mbR^{n} \; \vert \; \| \bm{x} \|_{p} \leq 1\}$ for $p\in \{1,\infty\}$. Then $\pi_{\mcC}(\bm{u})$ can be computed in $O\left(n\log(n)\right)$ operations.
\end{proposition}
\begin{proposition}\label{prop:proj-ball-2-C-polar}
Let $\mcX = \{ \bm{x} \in \mbR^{n} \; \vert \; \| \bm{x} \|_{2} \leq 1\}.$ Then $\pi_{\mcC}(\bm{u})$ can be computed in $O\left(n\right)$ operations.
\end{proposition}
 \subsection{Ellipsoidal confidence region in the simplex}\label{sec:projection-confidence-regions}
Here, $\mcX$ is an \textit{ellipsoidal subregion of the simplex}, defined as 
$\mcX = \{ \bm{x} \in \Delta(n) \; \vert \; \| \bm{x} - \bm{x}_{0} \|_{2} \leq \epsilon_{x} \}$. 
This type of decision set is widely used because it is associated with confidence regions when estimating a probability distribution from observed data~\citep{Iyengar,bertsimas2019probabilistic}. It can also be used in the Bellman update for robust Markov Decision Processes~\citep{Iyengar,Kuhn,GGC}. 
We also assume that the confidence region is ``entirely contained in the simplex'': $\{ \bm{x} \in \mbR^{n} \vert \bm{x}^{\top}\bm{e}=1 \} \bigcap \{ \bm{x} \in \mbR^{n} \; \vert \; \| \bm{x}- \bm{x}_{0} \|_{2} \leq \epsilon_{x} \} \subseteq \Delta(n)$, to avoid degenerate components. In this case, using a change of basis we show that it is possible to compute $\pi_{\mcC}(\bm{u})$ in closed-form, i.e., in $O(n)$ arithmetic operations.
\begin{proposition}\label{prop:proj-confidence-regions}
Let $\mcX = \{ \bm{x} \in \Delta(n) \; \vert \; \| \bm{x} - \bm{x}_{0} \|_{2} \leq \epsilon_{x} \}$ and assume that $\{ \bm{x} \in \mbR^{n} \vert \bm{x}^{\top}\bm{e}=1 \} \bigcap \{ \bm{x} \in \mbR^{n} \; \vert \; \| \bm{x}- \bm{x}_{0} \|_{2} \leq \epsilon_{x} \} \subseteq \Delta(n)$. Then $\pi_{\mcC}(\bm{u})$ can be computed in $O\left(n\right)$ arithmetic operations.
\end{proposition}

\subsection{Other decision sets via bisection}\label{sec:projection-other sets}
In the case where we can not find an exact solution of $\pi_{\mcC}(\bm{u})$ or $\pi_{\mcC^{\circ}}(\bm{u})$, it is possible to resort to bisection to obtain an approximate solutions. 
In particular, $\pi_{\mcC}(\bm{u})$ is the solution to the following optimization program:
\begin{equation}\label{eq:pi-C-optimization-program}
\min \{ ( \alpha \kappa -\tilde{u} )^{2} +  \| \alpha \bm{x}-\hat{\bm{u}} \|_{2}^{2}  \; \vert \; \alpha \geq 0, \bm{x} \in \mcX \}.
\end{equation}
If we fix $\alpha > 0$, then an optimal $\bm{x}(\alpha)$ is a solution to 
\begin{equation}\label{eq:pi-C-optimization-program-fixed-alpha}
\min \{  \| \bm{x}-\hat{\bm{u}}/\alpha \|_{2}^{2}  \; \vert \; \bm{x} \in \mcX \}.
\end{equation}
Therefore, if we can efficiently compute orthogonal projections on the set $\mcX$, it is possible to perform bisection on $\alpha$ to compute an $\epsilon$-approximation of $\pi_{\mcC}(\bm{u})$ in $O\left(\log(\epsilon^{-1})\right)$ iterations, i.e., solving \eqref{eq:pi-C-optimization-program-fixed-alpha} only $O\left(\log(\epsilon^{-1})\right)$ times.
\section{Numerical experiments}\label{sec:simu}
In this section we compare the performances of \spcbap{} on real and synthetic instances of classical saddle-point problems. 
We focus on bilinear matrix games, extensive-form games, distributionally robust logistic regression, and Markov decision processes (MDPs).
Recall that we have defined \spcbap{} by combining the repeated game framework from Section \ref{sec:game-setup} with \cbap{} as a regret minimizer, along with uniform weights on the payoffs, linear weights on the decisions and the alternating updates from Section \ref{sec:regret-vound-duality-gap-alternating-updates}.
We start by examining the performance of \spcbap\ on matrix and extensive-form games; for these games the \rmp\ algorithm is already known to perform extremely well empirically, and the goal of these experiments is to see whether \spcbap\ retains that very strong empirical performance. The experiments on distributionally robust logistic regression and MDPs then show the performance on new domains where no Blackwell-based algorithms were known prior to this paper.

\subsection{Matrix games}\label{sec:simu-bspp}
Matrix games are saddle-point problems with a bilinear objective function and simplexes as decision sets:
\begin{equation}\label{eq:matrix-games}
\min_{\bm{x} \in \Delta(n) } \max_{\bm{y} \in \Delta(m) } \langle \bm{x},\bm{Ay} \rangle
\end{equation}
where $\bm{A} \in \mbR^{n \times m}$ is the matrix of payoffs of the game. We can view \eqref{eq:matrix-games} as a zero-sum game between the first player and the second player, where the coefficient $A_{ij} \in \mbR$ represents payoff obtained by the second player when the first player chooses action $i$ and the second player chooses action $j$. 

\paragraph{Experimental setup} 
We generate 100 synthetic matrices $\bm{A}$ of size $\mbR^{n \times m}$ with $(n,m)=(100,50)$. Similarly as in \cite{ChambollePock16,nesterov2005smooth}, for the coefficients of $\bm{A}$ we consider a uniform distribution in $[0,1]$ or a normal distribution of mean $0$ and variance $1$. We compare \spcbap{} with \rmp{}, which is known to achieve the best empirical performance compared to a wide range of algorithms, including Hedge and other first-order methods~\citep{kroer2020ieor8100,kroer2018solving,farina2019optimistic}.   
We also compare with two other scale-free and parameter-free no-regret algorithms, AdaHedge~\citep{de2014follow} and AdaFTRL~\citep{orabona2015scale}, with the $\ell_{2}$ norm as the Bregman divergence. Similarly as for \spcbap, for \rmp{} we use the repeated game framework with alternation, along with linear averaging on the decisions and uniform averaging on the payoffs. In Figures \ref{fig:bilinear-games-period}-\ref{fig:bilinear-games-times}, we compare the performance of the four algorithms (\spcbap, \rmp, AdaHedge and AdaFTRL) for solving \eqref{eq:matrix-games}. In Figure \ref{fig:bilinear-games-duality-uniform} and Figure \ref{fig:bilinear-games-duality-normal}, we let the four algorithms run for $T=1000$ iterations, and we show the duality gap of the current running average as a function of the number of iterations. This shows the progress made by the algorithms toward solving \eqref{eq:matrix-games} at each iteration. In Figure \ref{fig:bilinear-games-time-uniform} and Figure \ref{fig:bilinear-games-time-normal}, we run the four algorithms for \tmm{} = 10 seconds, and we show the duality gap as a function of the time of computation. We average all the results over 50 randomly generated instances. Note that both axis are in logarithmic scale.

\paragraph{Results and discussion}
When we compare the duality gap as a function of the number of iterations (Figure \ref{fig:bilinear-games-duality-uniform} and Figure \ref{fig:bilinear-games-duality-normal}), we note that \spcbap{} performs on par with \rmp, and both algorithms vastly outperform AdaHedge and AdaFTRL. However, each iteration of \spcbap{} on the simplex requires solving $O\left(n \log(n)\right)$ arithmetic operations (see Section \ref{sec:projection-simplex}), whereas each iteration of \rmp{} can be performed in $O(n)$ operations. Therefore, when we compare the duality gap as a function of the computation time (Figure \ref{fig:bilinear-games-time-uniform} and Figure \ref{fig:bilinear-games-time-normal}), we note that \rmp{} outperforms \spcbap{}, even though after roughly ten seconds of computation, the performances of \spcbap{} and \rmp{} are equivalent.
\begin{figure}[hbt]
\begin{center}
   \begin{subfigure}{0.35\textwidth}
         \includegraphics[width=1.0\linewidth]{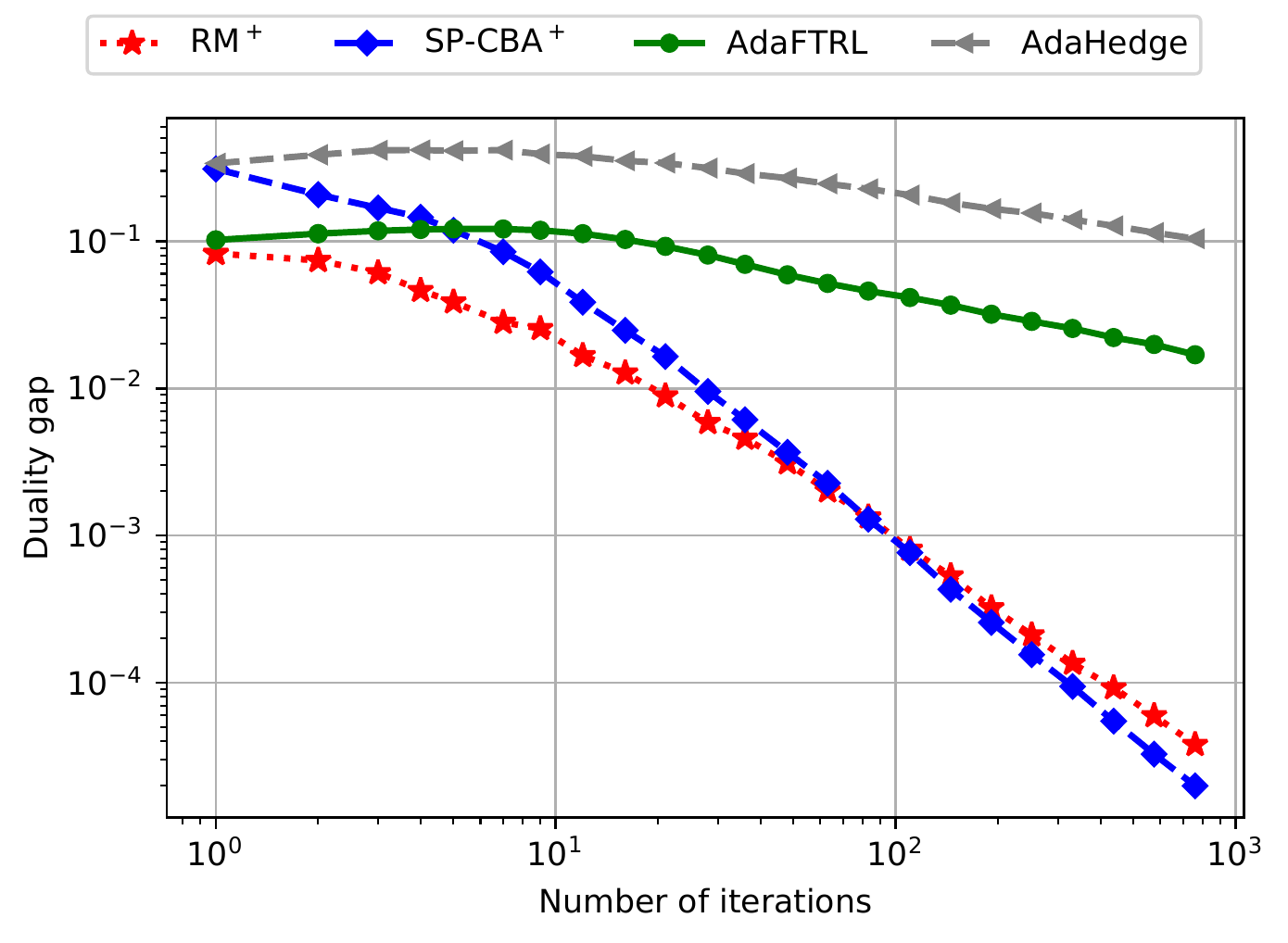}
         \caption{Uniform}
          \label{fig:bilinear-games-duality-uniform}
  \end{subfigure}
   \begin{subfigure}{0.35\textwidth}
         \includegraphics[width=1.0\linewidth]{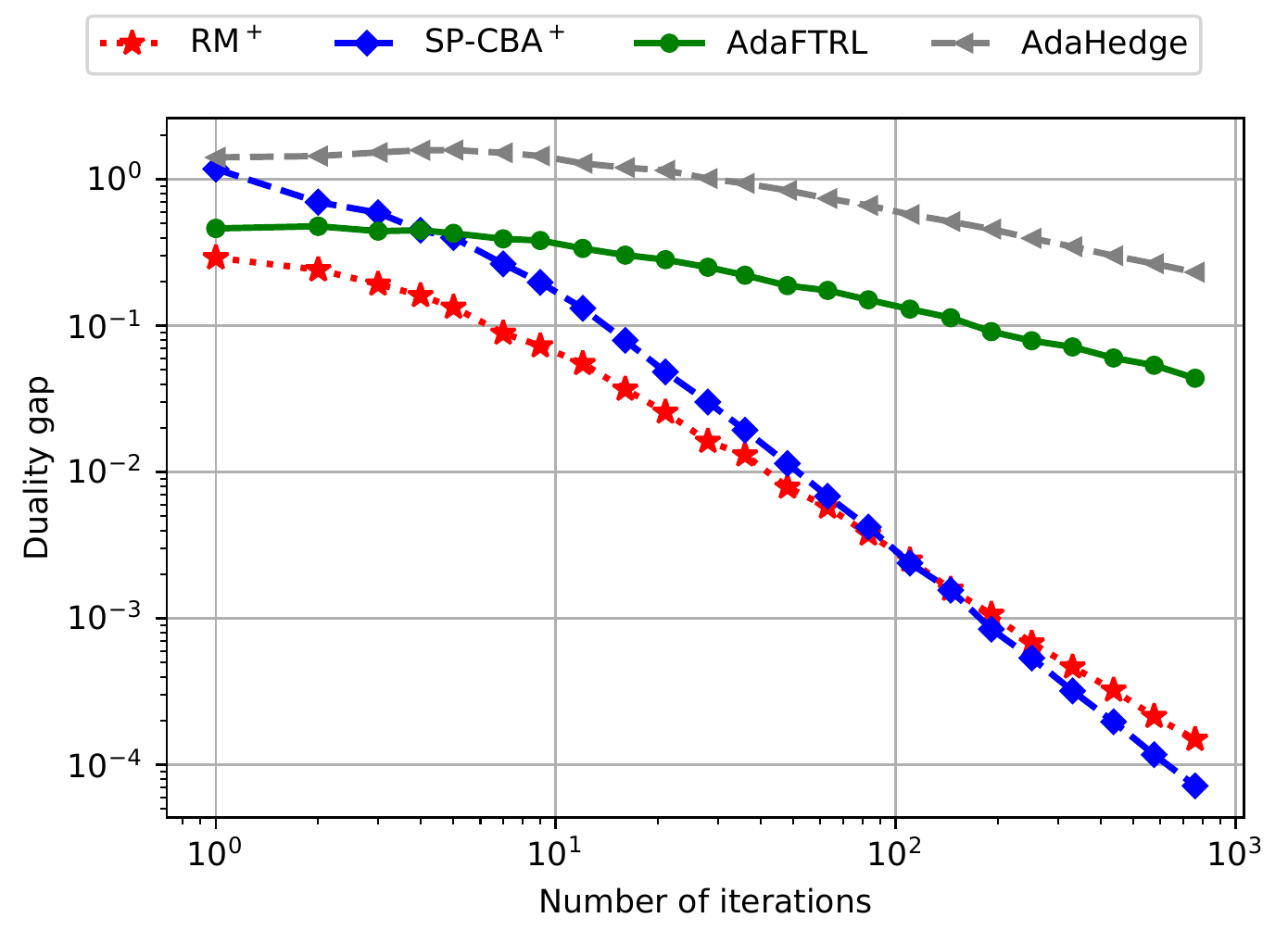}
         \caption{Normal}
 \label{fig:bilinear-games-duality-normal}
  \end{subfigure}
\end{center}
  \caption{Comparison of \spcbap, \rmp, AdaHedge and AdaFTRL on instances of matrix games, with respect to the number of iterations. The payoffs are chosen randomly, with uniform distribution in Figure \ref{fig:bilinear-games-duality-uniform} and normal distribution in Figures \ref{fig:bilinear-games-duality-normal}.}
  \label{fig:bilinear-games-period}
\end{figure}
\begin{figure}[hbt]
\begin{center}
   \begin{subfigure}{0.35\textwidth}
         \includegraphics[width=1.0\linewidth]{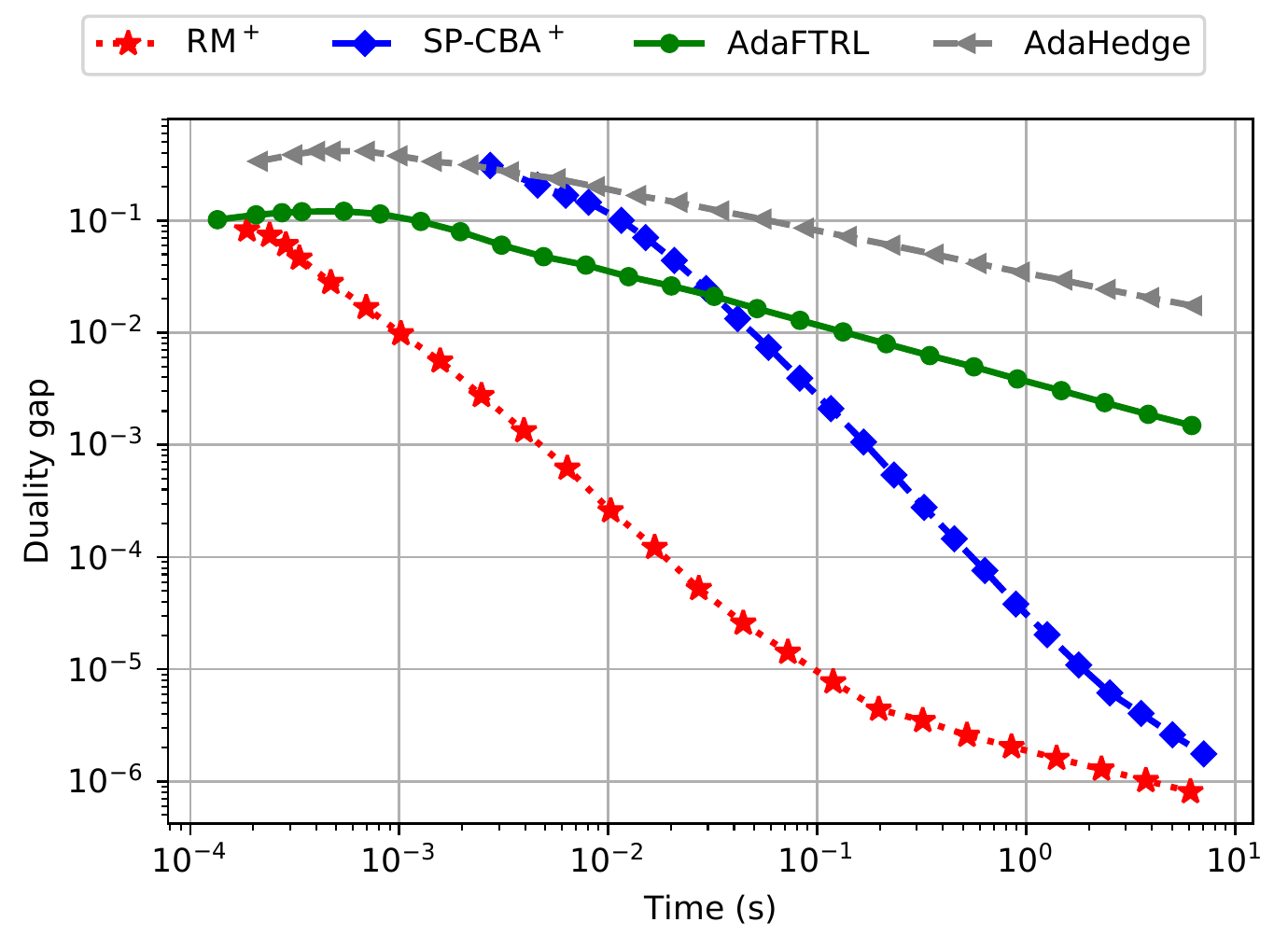}
         \caption{Uniform}
             \label{fig:bilinear-games-time-uniform}
  \end{subfigure}
   \begin{subfigure}{0.35\textwidth}
         \includegraphics[width=1.0\linewidth]{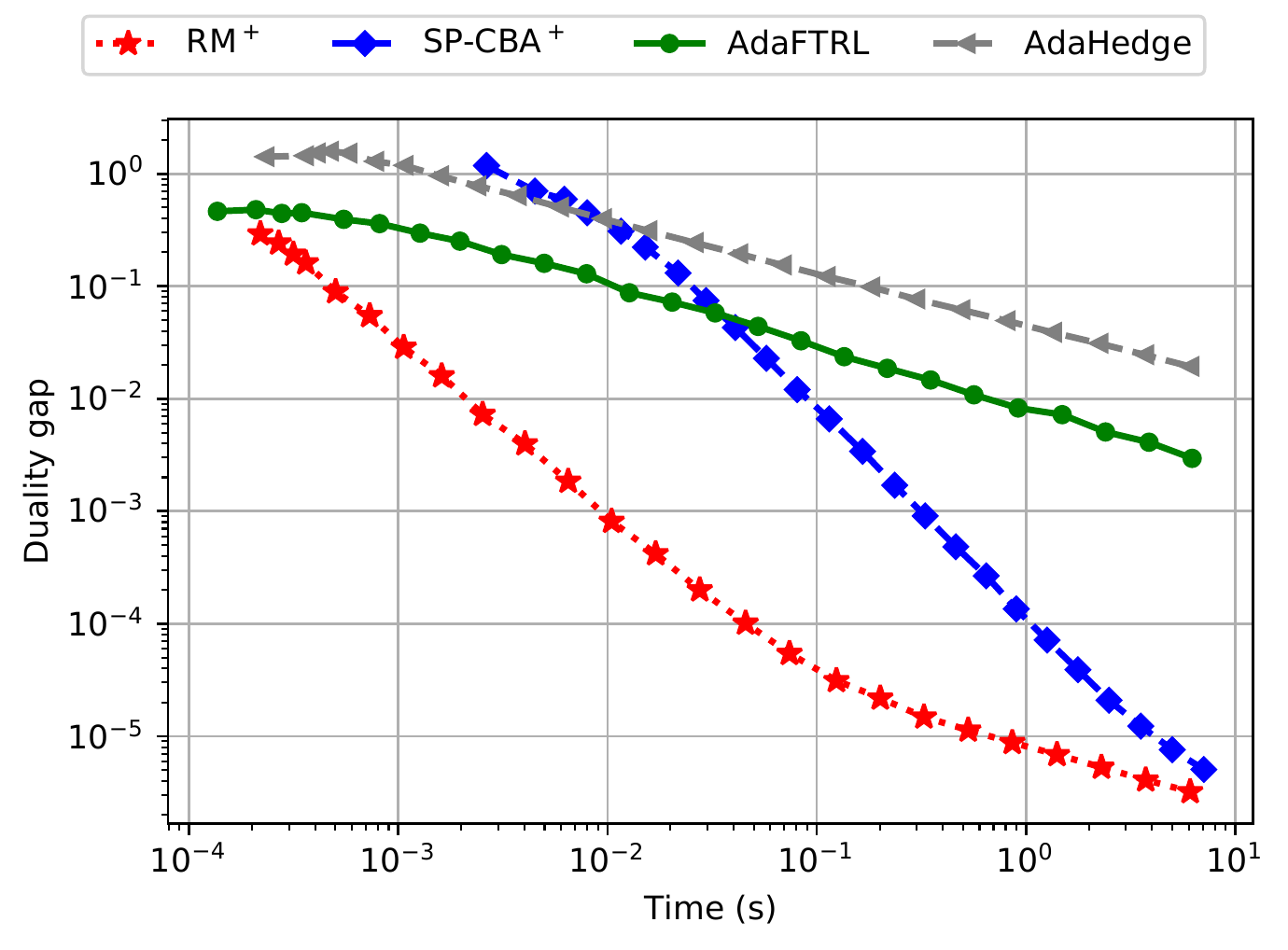}
\caption{Normal}
\label{fig:bilinear-games-time-normal}
  \end{subfigure}
\end{center}
  \caption{Comparison of \spcbap, \rmp, AdaHedge and AdaFTRL on instances of matrix games, with respect to computation time. The payoffs are chosen randomly, with uniform distribution in Figure \ref{fig:bilinear-games-time-uniform}, and normal distribution in Figure \ref{fig:bilinear-games-time-normal}.}
  \label{fig:bilinear-games-times}
\end{figure}

\subsection{Extensive-form games}\label{sec:simu-efg}
Extensive-form games (EFGs,~\citep{stengel1996efficient,zinkevich2007regret}) are used to model sequential games with imperfect information. For example, they were used for superhuman poker AIs in games such as Texas hold'em~\citep{tammelin2015solving,brown2018superhuman,brown2019superhuman,moravvcik2017deepstack}. EFGs can be written as saddle-point problems, with a bilinear objective functions and polytopes $\mcX,\mcY$ encoding the players' decision spaces~\citep{stengel1996efficient}. Based on the counterfactual regret minimization (CFR) framework~\citep{zinkevich2007regret}, EFGs can be solved via decomposition into a set of simplex-based regret minimization problems. We point the reader to \citep{farina2019online,farina2019regret} for more details. 

\paragraph{Experimental setup}
For solving EFGs, we combine the CFR decomposition with \cbap{} as a regret minimizer on the simplex. For the sake of simplicity, we will still call the resulting algorithm \spcbap{} (since we use alternation and linear averaging on the decisions), even though the algorithm relies on the CFR decomposition for EFGs (which is not necessary for solving the other saddle-point instances from Section \ref{sec:simu-bspp}, Section \ref{sec:simu-dro} and Section \ref{sec:simu-mdp}). We compare \spcbap{} with \cfrp~\citep{bowling2015heads}, the algorithm with the strongest empirical performance for solving EFGs. 
Note that both \spcbap{} and \cfrp{} guarantee a $O(1/\sqrt{T})$ rate of convergence to a Nash equilibrium. We compare \spcbap{} and \cfrp{} on several Leduc poker benchmark instances, a search game, and sheriff; we refer to \citep{farina2021faster} for details about the instances. Similarly as in Section \ref{sec:simu-bspp}, we compare the progress of \spcbap{} and \cfrp{} both as a function of computation time and number of iterations in the repeated game framework. We run the algorithms for \tmm{} = 100 seconds and $T=1500$ iterations; note that we choose \tmm{} and $T$ larger for EFGs than for matrix games because the EFG instances are way larger than the matrix games from Section~\ref{sec:simu-bspp}.

\paragraph{Results and discussion}
If we only consider the duality gap as a function of the number of iterations (Figure \ref{fig:efgs-times}), \spcbap{} performs on par with \cfrp, and significantly outperforms \cfrp{} on some EFGs instances (Figure \ref{fig:efgs-steps-search-game-4} and Figure \ref{fig:efgs-steps-battleship}). However, when we consider the progress made by each algorithm during \tmm{} = 100 seconds (Figure \ref{fig:efgs-times}), \cfrp{} enjoys better numerical performances than \spcbap. This is because the updates are  closed-form in \cfrp, whereas each update of \spcbap{} requires to solve an equation, a situation similar as for matrix games over the simplex (Section \ref{sec:simu-efg}). 
It is interesting to note that for EFGs, the difference in per-iteration computation time has a bigger impact than for matrix games; it is possible that this is due to our python-based implementation of \spcbap. Better implementations of \spcbap\ for EFGs could potentially lead to better results. To conclude this section, we note that \cfrp{} enjoys the best empirical performances for solving EFGs, and it is not concerning that \spcbap{} can not outperform \cfrp{} on EFGs (in terms of computation time). Instead, we will see in the next section how \spcbap{} carries over these very strong empirical results to saddle-point instances where \cfrp{} does not apply and where \spcbap{} can be implemented more efficiently.
\begin{figure}[hbt]
\begin{center}
   \begin{subfigure}{0.24\textwidth}
         \includegraphics[width=1.0\linewidth]{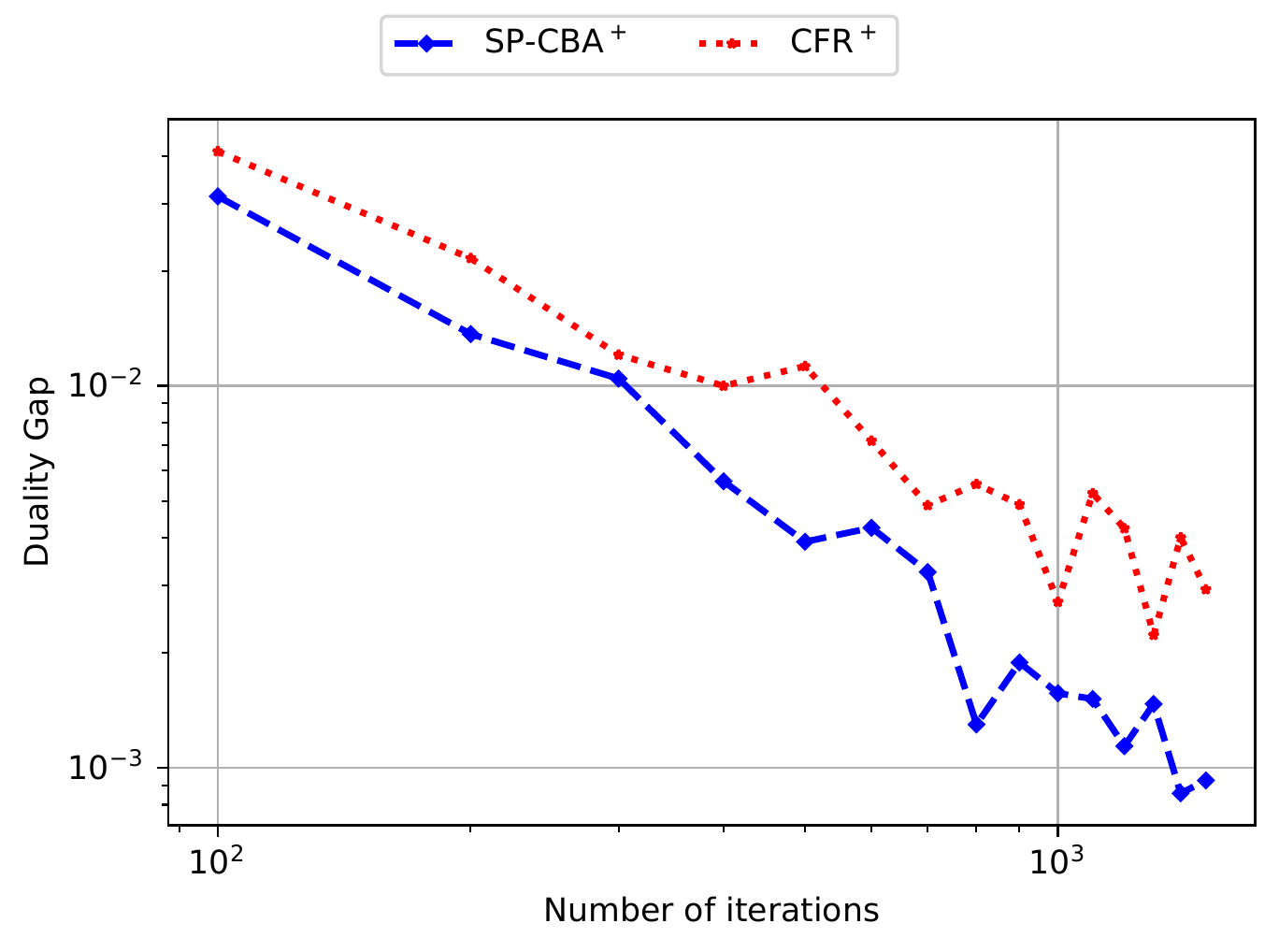}
         \caption{Search game (4 turns)}
             \label{fig:efgs-steps-search-game-4}
  \end{subfigure}
    \begin{subfigure}{0.24\textwidth}
         \includegraphics[width=1.0\linewidth]{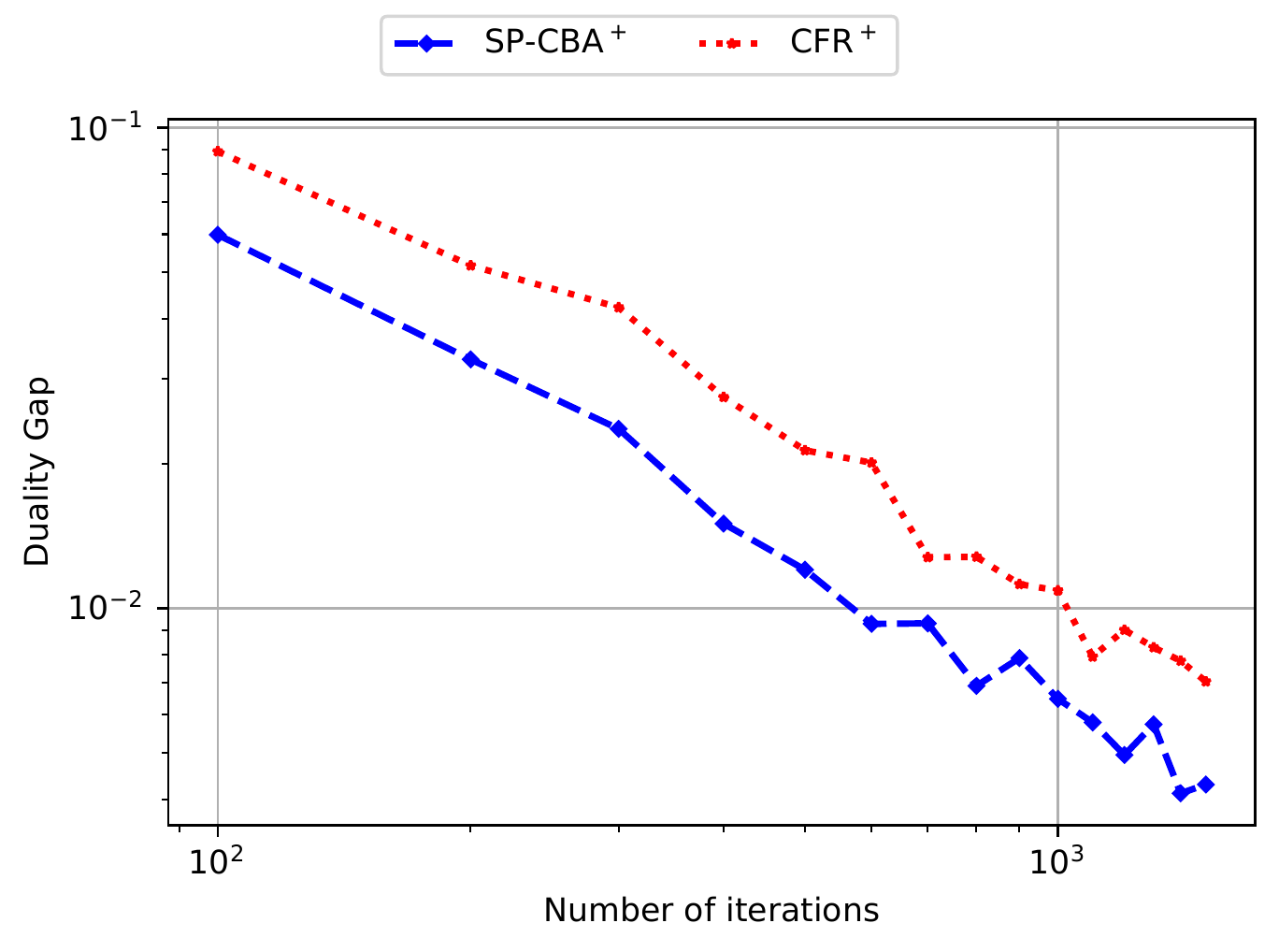}
         \caption{Battleship (3 turns)}
             \label{fig:efgs-steps-battleship}
  \end{subfigure}
   \begin{subfigure}{0.24\textwidth}
         \includegraphics[width=1.0\linewidth]{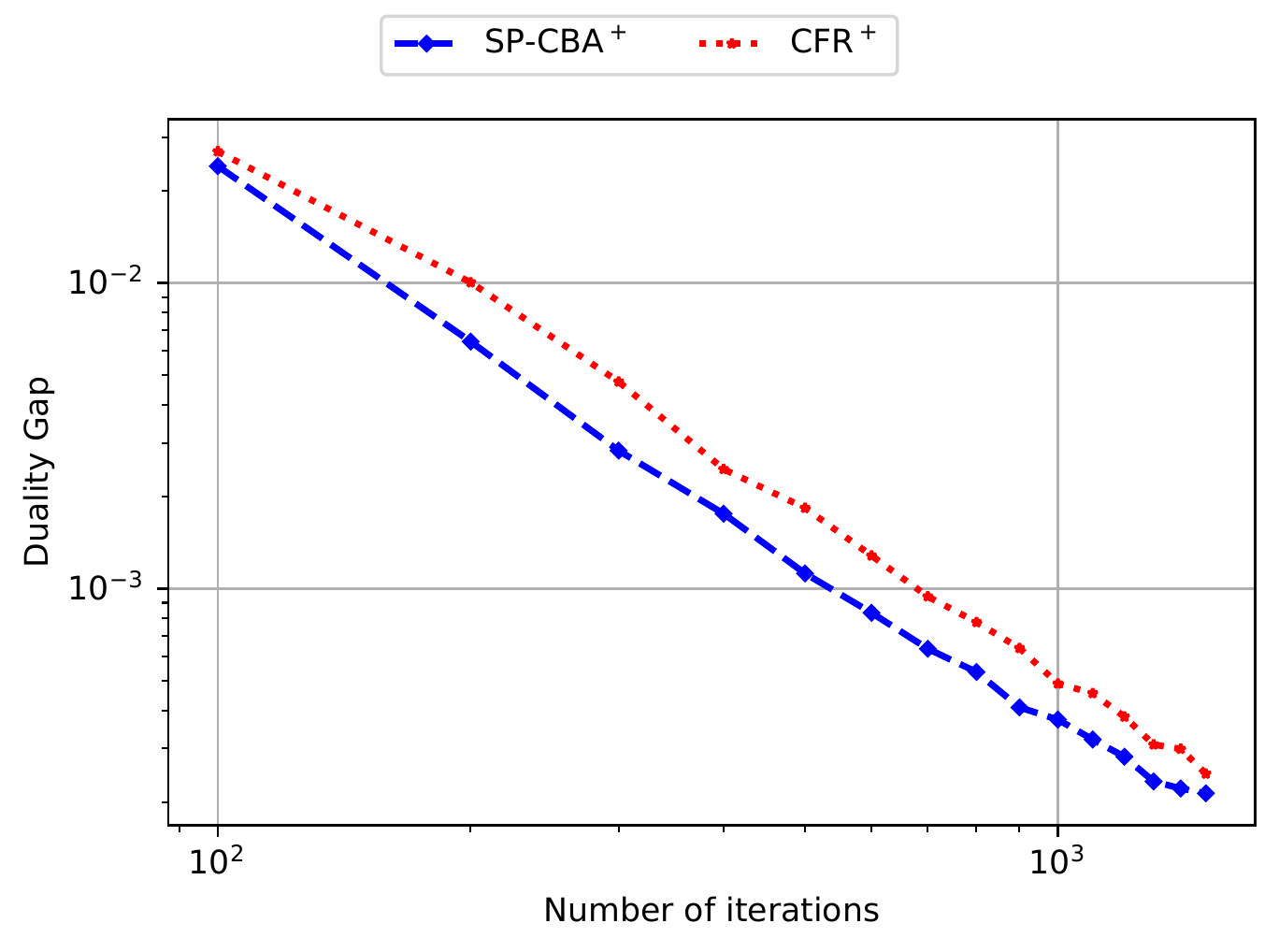}
         \caption{Leduc, 2pl. 3 ranks}
             \label{fig:efgs-steps-leduc_2pl_3ranks}
  \end{subfigure}
   \begin{subfigure}{0.24\textwidth}
         \includegraphics[width=1.0\linewidth]{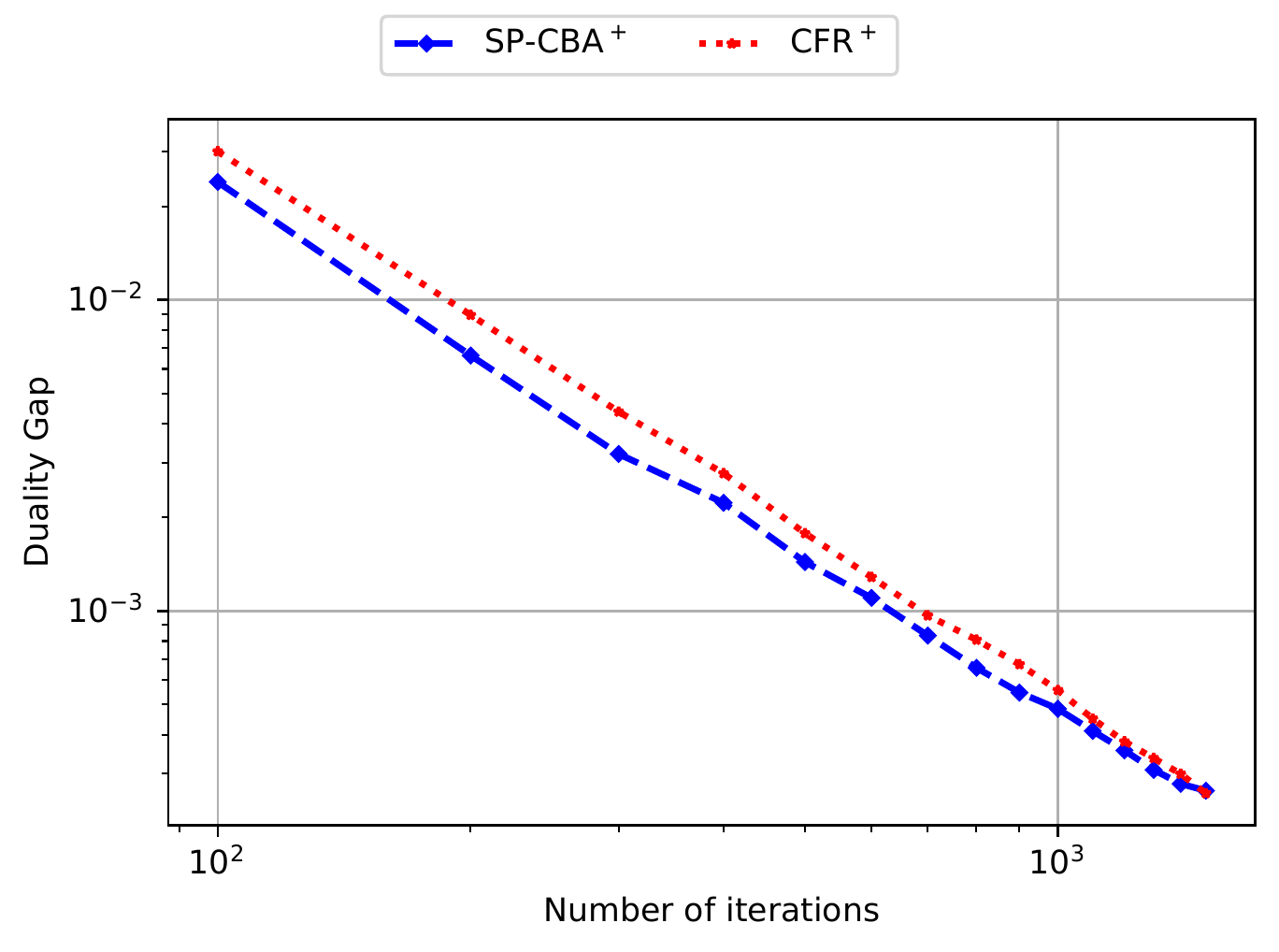}
         \caption{Leduc, 2pl. 5 ranks}
             \label{fig:efgs-steps-leduc_2pl_5ranks}
  \end{subfigure}
\end{center}
  \caption{Comparison of \spcbap{} and \cfrp{} for solving extensive-form games, as regards the number of iterations.}
  \label{fig:efgs-steps}
\end{figure}
\begin{figure}[hbt]
\begin{center}
   \begin{subfigure}{0.24\textwidth}
         \includegraphics[width=1.0\linewidth]{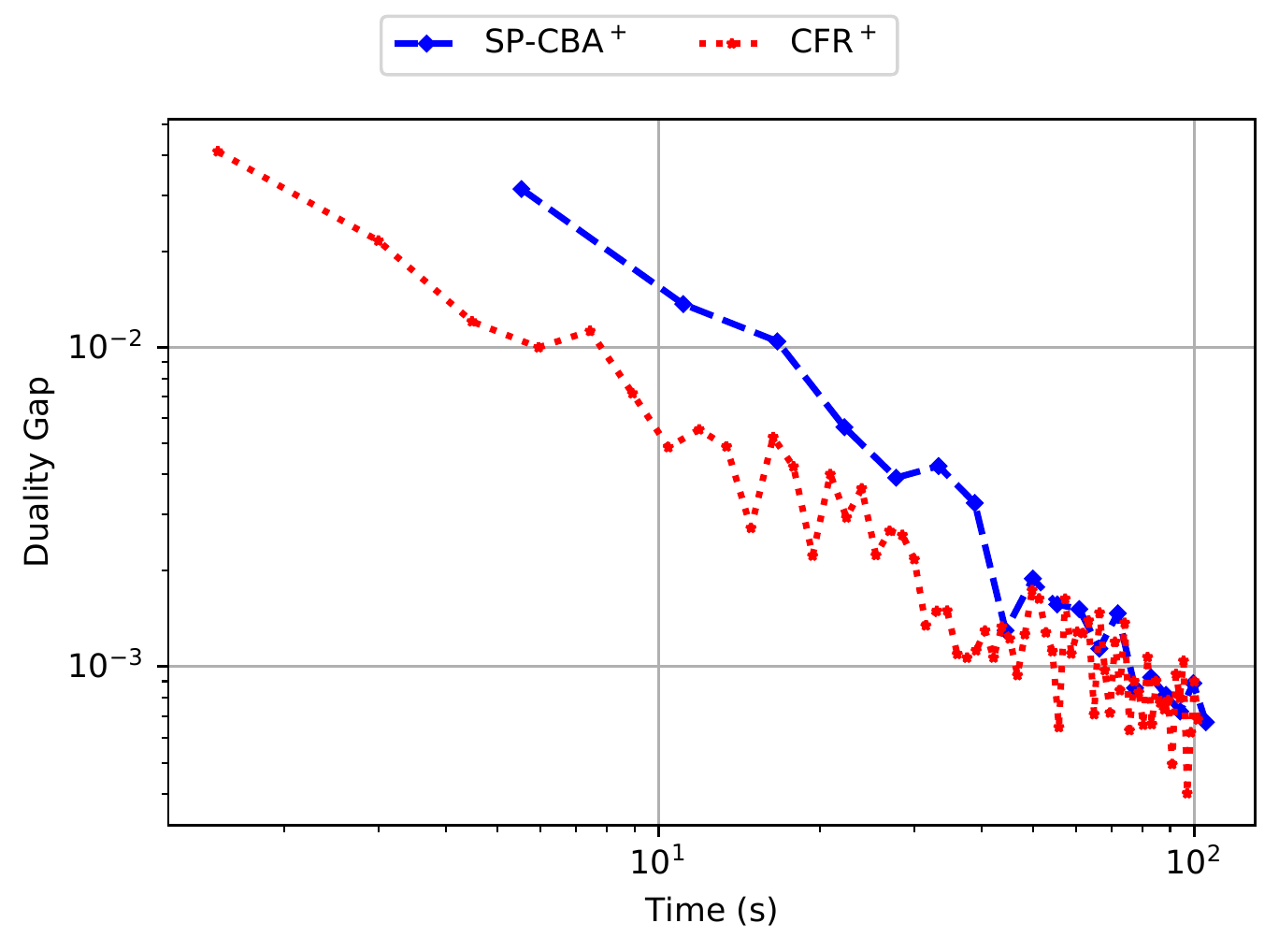}
         \caption{Search game (4 turns)}
             \label{fig:efgs-time-search-game-4}
  \end{subfigure}
    \begin{subfigure}{0.24\textwidth}
         \includegraphics[width=1.0\linewidth]{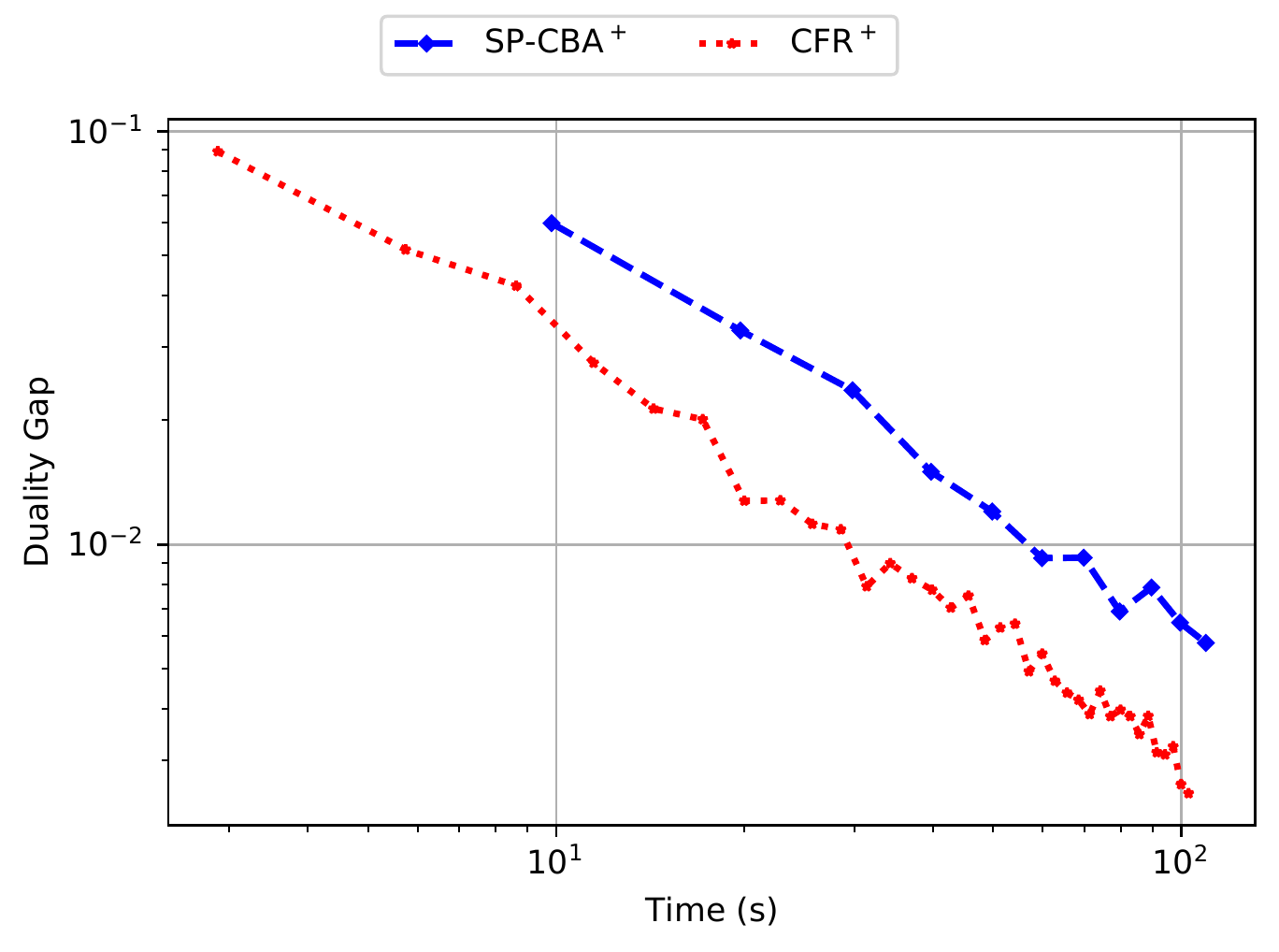}
         \caption{Battleship (3 turns)}
             \label{fig:efgs-time-battleship}
  \end{subfigure}
   \begin{subfigure}{0.24\textwidth}
         \includegraphics[width=1.0\linewidth]{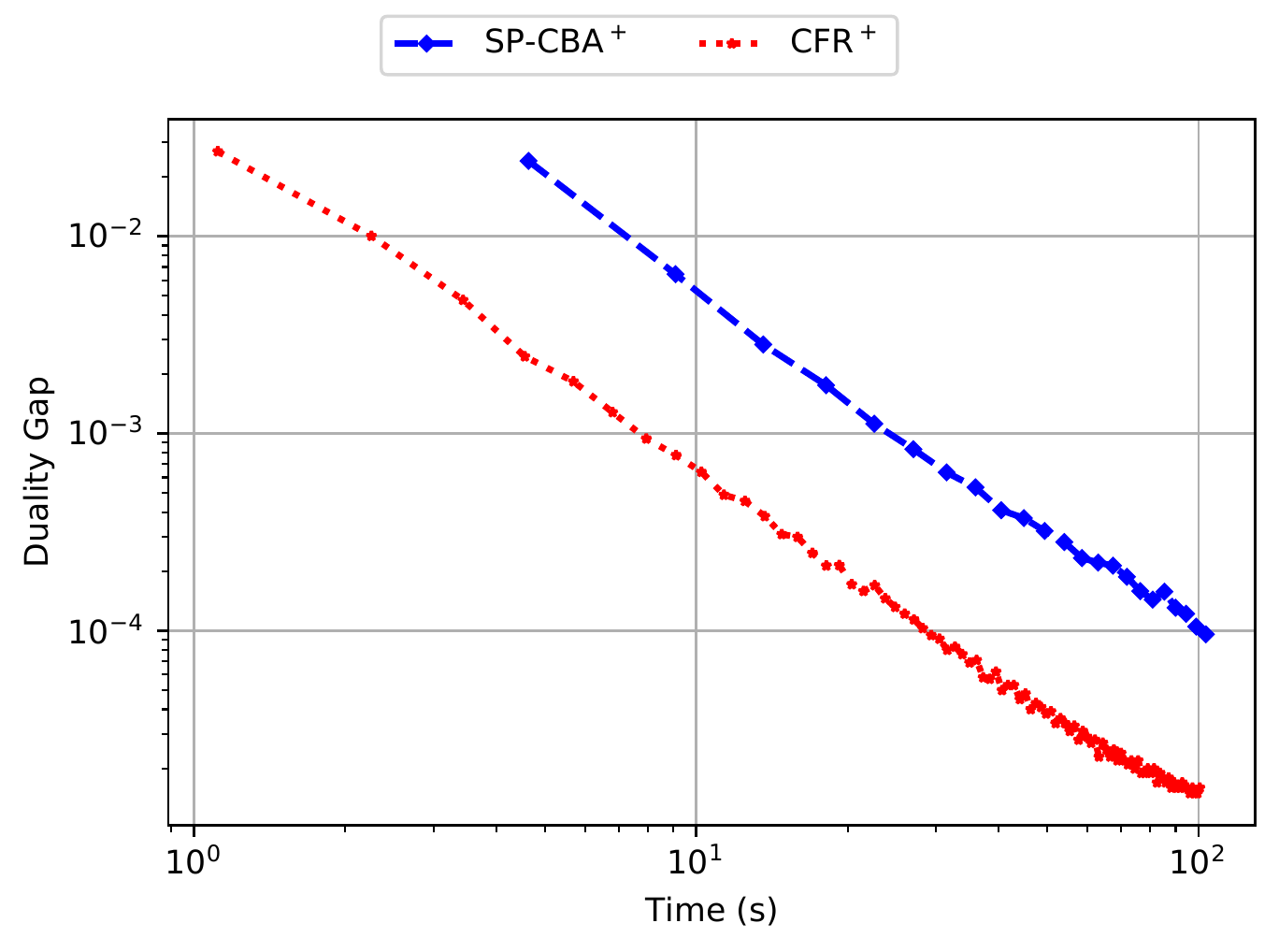}
         \caption{Leduc, 2pl. 3 ranks}
             \label{fig:efgs-time-leduc_2pl_3ranks}
  \end{subfigure}
   \begin{subfigure}{0.24\textwidth}
         \includegraphics[width=1.0\linewidth]{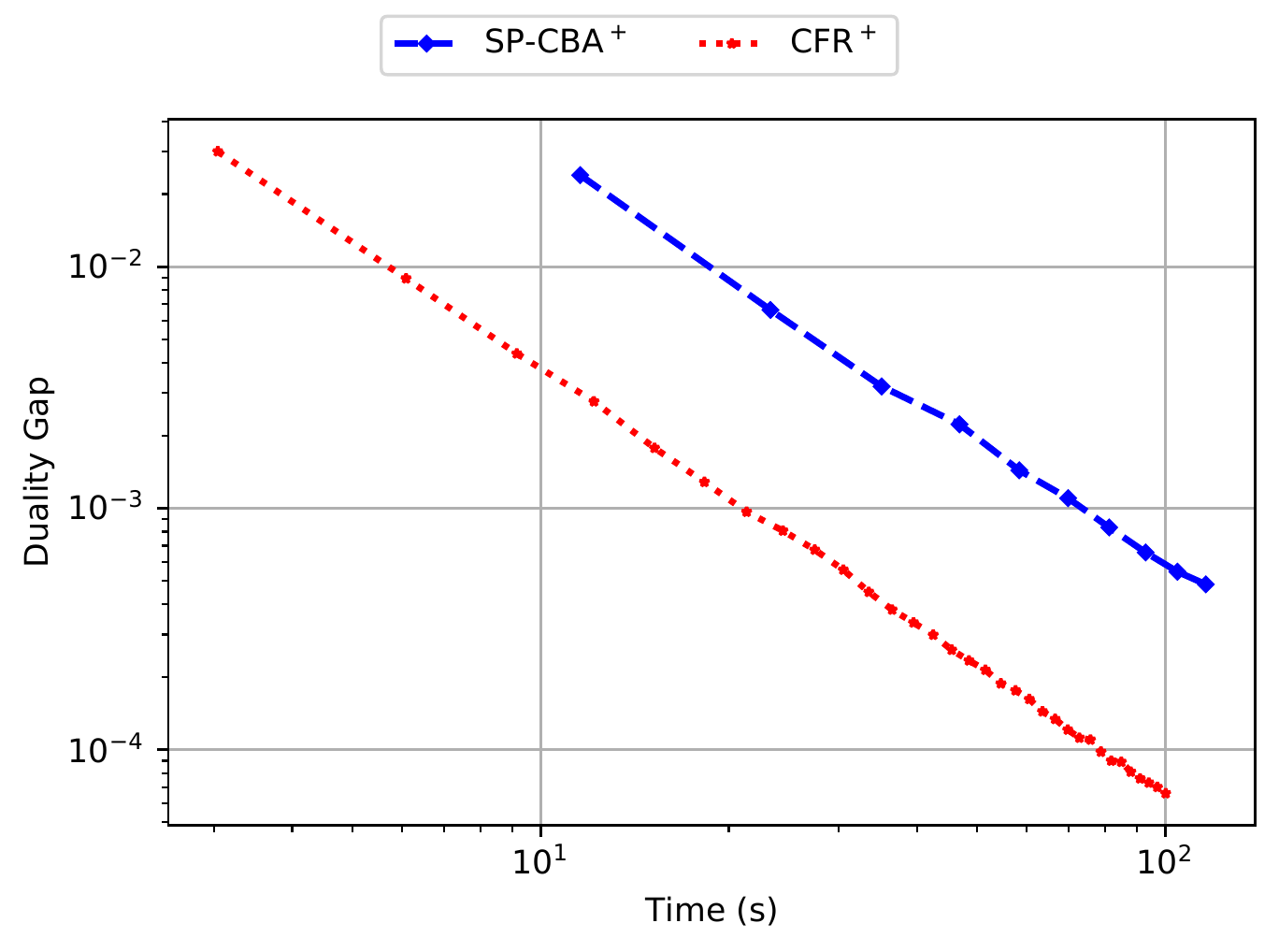}
         \caption{Leduc, 2pl. 5 ranks}
             \label{fig:efgs-time-leduc_2pl_5ranks}
  \end{subfigure}
\end{center}
  \caption{Comparison of \spcbap{} and \cfrp{} for solving extensive-form games, based on the computation time.}
  \label{fig:efgs-times}
\end{figure}
\subsection{Distributionally robust logistic regression}\label{sec:simu-dro}
Distributionally robust optimization exploits knowledge of the statistical properties of the model parameters to obtain risk-averse optimal solutions~\citep{rahimian2019distributionally}.
We focus on the following instance of distributionally robust logistic regression \citep{namkoong2016stochastic,ben2015oracle}.  There are $m$ observed feature-label pairs $\left( \bm{a}_{i},b_{i} \right) \in \mbR^{n} \times \{-1,1\}$,  and we want to solve 
\begin{equation}\label{eq:dro}
\min_{\bm{x} \in \mbR^{n}, \| \bm{x} - \bm{x}_{0}\|_{2} \leq \epsilon_{x}} \max_{\bm{y} \in \Delta(m), \| \bm{y} - \bm{y}_{0} \|_{2} \leq \epsilon_{y}} \sum_{i=1}^{m} y_{i} \ell_{i}(\bm{x})+\frac{\mu}{2}\| \bm{x}\|_{2}^{2}
\end{equation}
where $\ell_{i}(\bm{x}) = \log(1+\exp(-b_{i}\bm{a}^{\top}_{i}\bm{x}))$ and $\mu \geq 0$. The formulation \eqref{eq:dro} takes a worst-case approach to put more weight on misclassified observations and provides some statistical guarantees, e.g., it can be seen as a convex regularization of standard empirical risk minimization instances~\citep{duchi2021statistics}.
\paragraph{Experimental setup}
We compare \spcbap{} with four classical first-order methods (FOMs): Online Mirror Descent (\ref{alg:OMD}), Optimistic OMD (\ref{alg:pred-OMD}), Follow-The-Regularized-Leader (\ref{alg:FTRL}) and Optimistic FTRL (\ref{alg:pred-FTRL}). We provide a detailed presentation of our implementations of these algorithms and our experimental setting in Appendix \ref{app:details-omd}; we use the $\ell_{2}$ norm as the Bregman divergence. 
We compare the performances of these algorithms with \spcbap{} on two synthetic datasets and two real data sets.  
We use parameters $\bm{x}_{0} = \left(1,...,1\right)/n,\epsilon_{x}=10,\bm{y}_{0} = \left(1,...,1\right)/m,\epsilon_{y} = 1/2m,\mu=0.1$ in \eqref{eq:dro}, and we initialize all algorithms at $\bm{x}_{0},\bm{y}_{0}$.
For the synthetic classification instances, we generate a vector $\bm{x}^{*} \in \mbR^{n}$,  we sample some vectors $\bm{a}_{i} \in \mbR^{n}$ at random for $i \in \{1,...,m\}$,  set labels $b_{i} = \text{sign} (\bm{a}^{\top}_{i}\bm{x}^{*})$, and then  we flip $10 \%$ of the labels. We consider two types of synthetic instances: one where $a_{ij}$ is sampled from a uniform distribution in $[0,1]$, and one where $a_{ij}$ is sampled from a normal distribution with mean $0$ and variance $1$.
For the real classification instances, we use the following datasets from the {\sf libsvm} website\footnote{https://www.csie.ntu.edu.tw/$\sim$cjlin/libsvmtools/datasets/}: \textit{adult} and {\em splice}. 

One of the main motivations for \spcbap{} is to obtain a \textit{parameter-free} algorithm.
In contrast, the other FOMs considered in this section require choosing step sizes $\eta_{t}$ at every iteration $t$. 
This is a major limitation in practice: if the step sizes are too small, the iterates may be very conservative, while the algorithms may diverge with very large step sizes. We will compare the performances of the FOMs for both the fixed, theoretically-correct step sizes, and for tuned step sizes. The computation of the theoretically-correct step sizes is presented in Appendix \ref{app:simu-dro-step-size}. 
To tune the FOMs, we run them for the first $10$ iterations, with step sizes $\eta_{t} = \alpha / \sqrt{t+1}$ for \ref{alg:OMD} and \ref{alg:FTRL} and step size $\eta_{t} = \alpha$ for \ref{alg:pred-OMD} and \ref{alg:pred-FTRL}, and we search for the best $\alpha \in \{0.01,0.1,1,10,100\}$. We then choose the value of $\alpha$ that lead to the smallest duality gap after $10$ iterations, and use this value for the remaining $T=1000$ iterations. Note that the tuning time and iterations (where the first $10$ iterations are repeated with various values of $\alpha$) are counted in the total computation time and number of iterations of the FOMs.
We acknowledge that this tuning method is only one possibility and that the multiplicative factor $\alpha$ could be chosen in many different ways. However, any other tuning framework would still be resource-demanding and uncertain. In contrast, \spcbap{} does not require any tuning, and, as we will see, outperforms even the tuned FOMs. Finally, on the $y$-axis we only report the worst-case loss of the current average $\bar{\bm{x}}_{T}$; in particular, we do not compute the duality gap at every iteration, because for a fixed value of $\bm{y}$, computing the optimal $\bm{x}$ requires solving a (regularized) nominal logistic regression, which would be computationally intensive to do at every iteration.

\paragraph{Proximal updates for the first-order methods}
Note that in \eqref{eq:dro}, \spcbap{} is instantiated on an $\ell_{2}$ ball (for the first player) and the intersection of an $\ell_{2}$ ball and the simplex (for the second player). As shown in Section \ref{sec:projection-ell-p-balls} and Section \ref{sec:projection-confidence-regions}, this leads to closed-form updates for \spcbap{} at every iteration. In contrast, \ref{alg:OMD}, \ref{alg:FTRL},  \ref{alg:pred-OMD}, and \ref{alg:pred-FTRL}  require binary searches for the decision of the second player at each iteration, see Appendix \ref{app:details-omd}.  The functions used in the binary searches themselves require solving an optimization program (an orthogonal projection onto the simplex) at each evaluation.  Even though computing the orthogonal projection of a vector onto the simplex of size $m$ can be done in $O(m \log(m))$, this results in slower overall running time, compared to \spcbap{} with closed-form updates at each iteration.  The situation is even worse for \ref{alg:pred-OMD}, which requires two proximal updates at each iteration.
\paragraph{Results and discussion}
In Figure \ref{fig:dro-period-theoretical}, we show the progress of all algorithms toward solving \eqref{eq:dro} as a function of the number of iterations, when the theoretical step sizes are used for the FOMs. We notice that all FOMs are progressing very slowly toward an optimal solution. This is because the theoretical step sizes are very small, relying on upper bounds on the Lipschitz constants of the objective function of \eqref{eq:dro}. In contrast, \spcbap{} quickly converges to an optimal solution, even though we see in Figure \ref{fig:dro-period-uniform-theoretical} that during the first few iterations, \spcbap{} may increase the objective function. In Figure \ref{fig:dro-period-tuned}, we tune the FOMs for the first $10$ iterations, before running them (with the tuned step sizes). We note that depending on the datasets, the tuned FOMs may perform very well (e.g., \ref{alg:OMD} in Figure \ref{fig:dro-period-uniform-tuned}, all FOMS in Figure \ref{fig:dro-period-normal-tuned}, \ref{alg:pred-OMD} in Figure \ref{fig:dro-period-australian-tuned}), but may also fail to converge to an optimal solution, even after very good performances during the first iterations (e.g., \ref{alg:pred-FTRL} in Figure \ref{fig:dro-period-australian-tuned}). This is because the convergence guarantees of the FOMs may fail to hold, for large choices of the multiplicative factor $\alpha$. In Figure \ref{fig:dro-time-theoretical} and Figure \ref{fig:dro-time-tuned}, we present the same experiments but where we record the computation time on the $x$-axis. Recall that the per-iteration computation time  of \spcbap{} is shorter than for the FOMs, because \spcbap{} has closed-form updates in this setting. Therefore, we still observe in Figures \ref{fig:dro-time-theoretical}-\ref{fig:dro-time-tuned} that \spcbap{} outperforms the classical FOMs.
\begin{figure}[hbt]
\begin{center}
   \begin{subfigure}{0.24\textwidth}
         \includegraphics[width=1.0\linewidth]{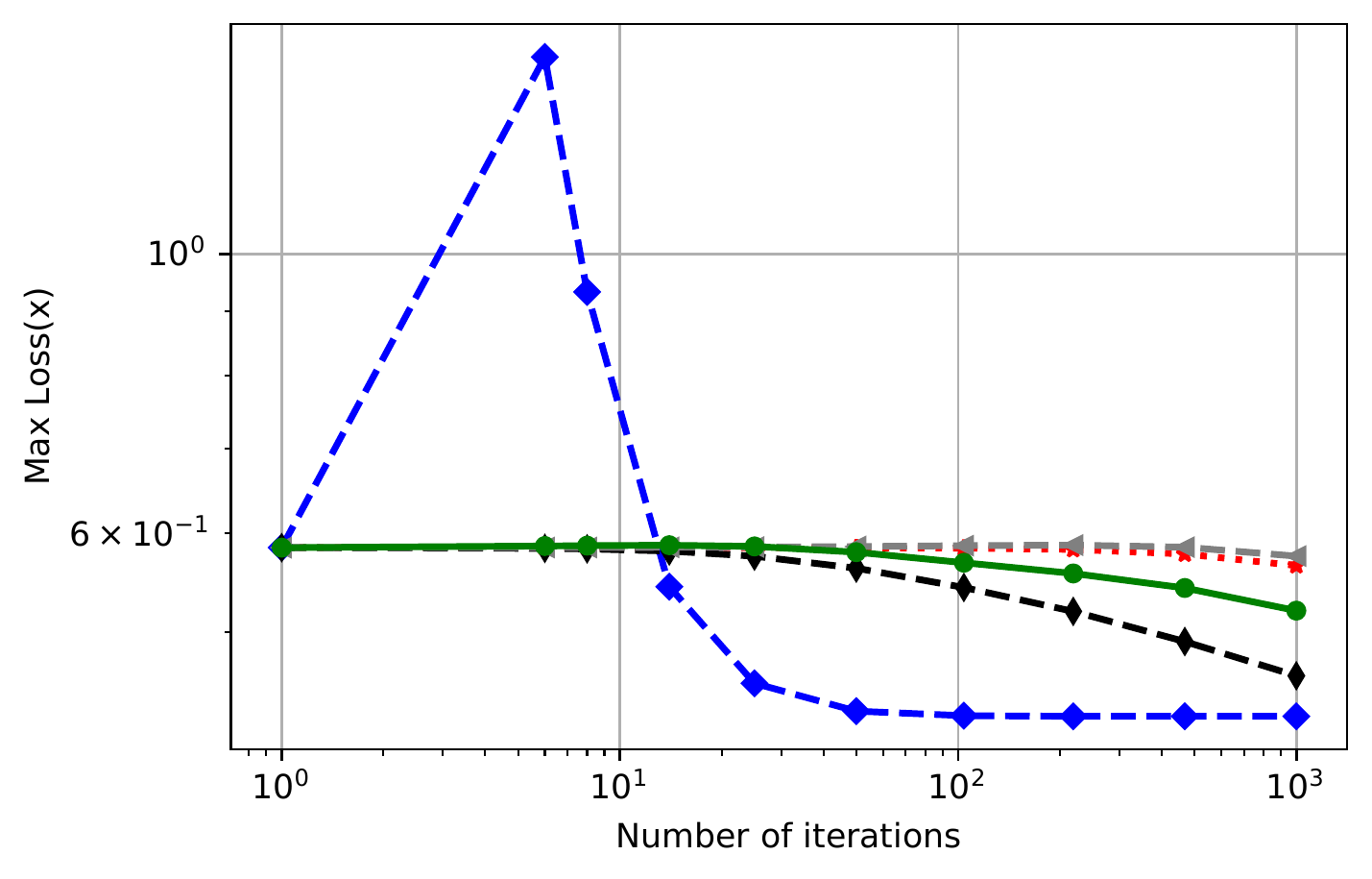}
         \caption{Uniform}
          \label{fig:dro-period-uniform-theoretical}
  \end{subfigure}
   \begin{subfigure}{0.24\textwidth}
         \includegraphics[width=1.0\linewidth]{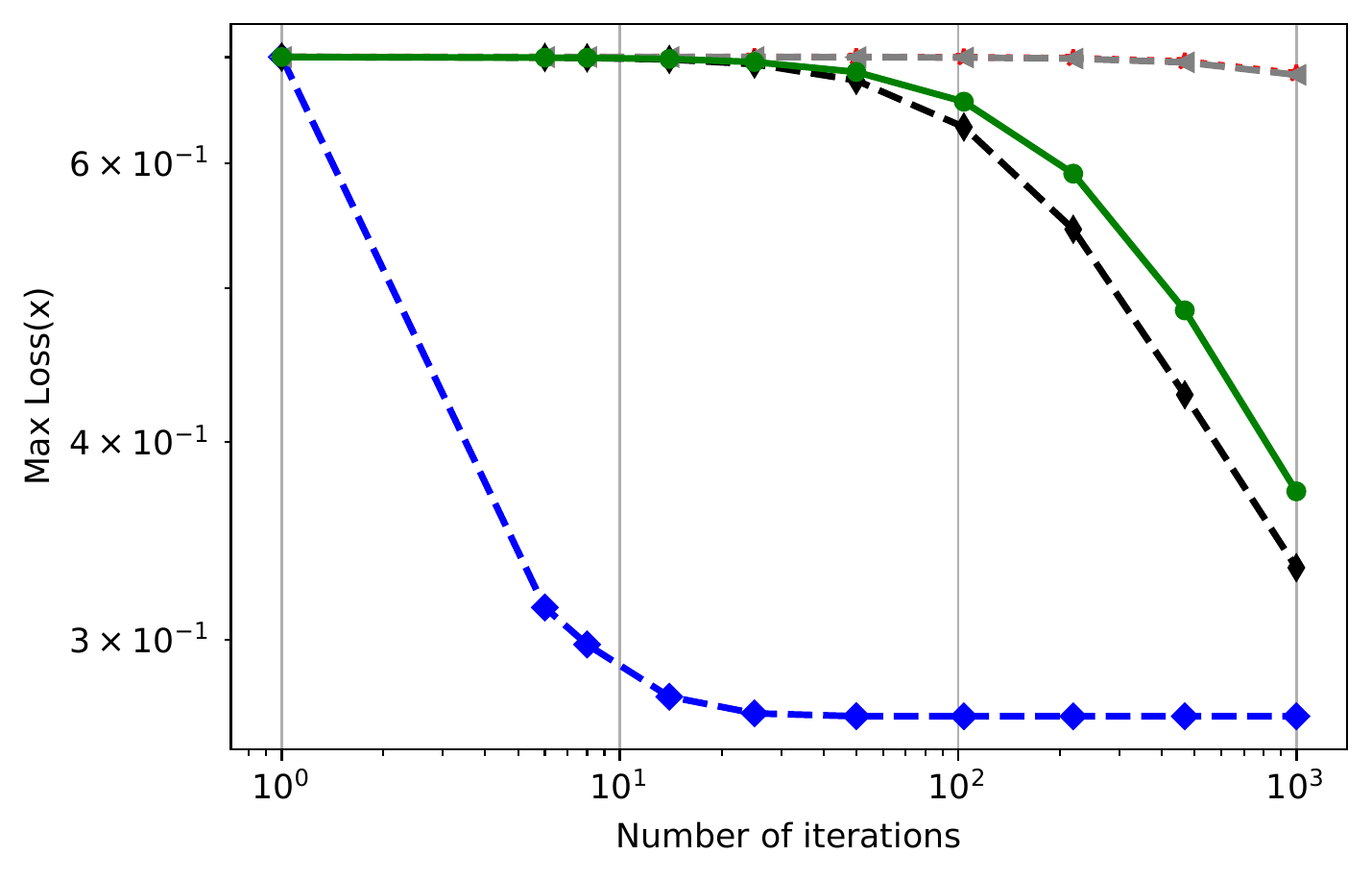}
         \caption{Normal}
          \label{fig:dro-period-normal-theoretical}
  \end{subfigure}
     \begin{subfigure}{0.24\textwidth}
         \includegraphics[width=1.0\linewidth]{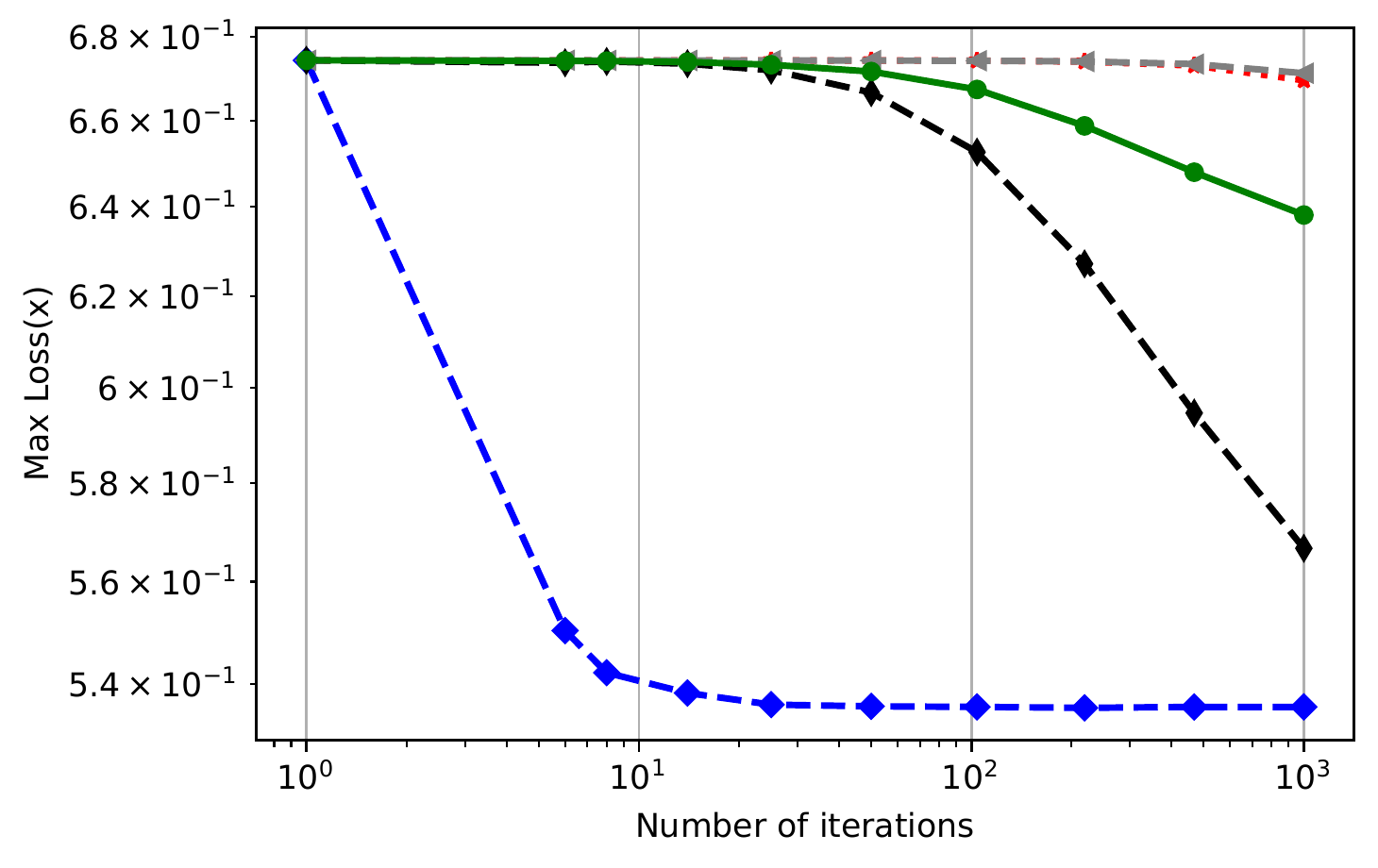}
         \caption{Australian}
          \label{fig:dro-period-australian-theoretical}
  \end{subfigure}
   \begin{subfigure}{0.24\textwidth}
         \includegraphics[width=1.0\linewidth]{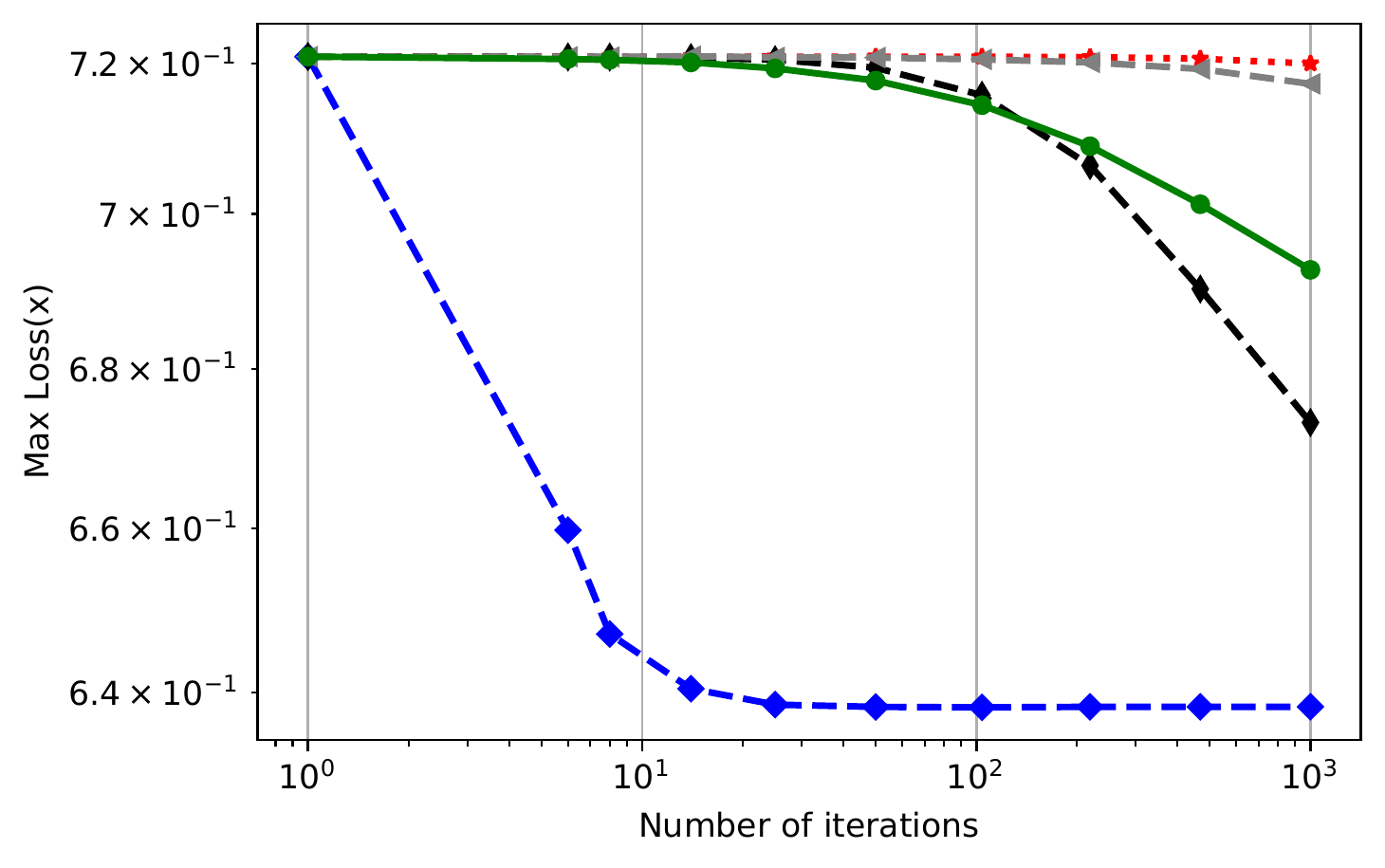}
         \caption{Splice}
          \label{fig:dro-period-splice-theoretical}
  \end{subfigure}
\end{center}
  \caption{Comparisons of \cbap, OMD, FTRL, OOMD and OFTRL to compute a solution to the distributionally robust logistic regression problem~\eqref{eq:dro}, based on the number of iterations. The theoretical choices of step sizes are used in the first-order methods.}
  \label{fig:dro-period-theoretical}
\end{figure}
\begin{figure}[hbt]
\begin{center}
   \begin{subfigure}{0.24\textwidth}
         \includegraphics[width=1.0\linewidth]{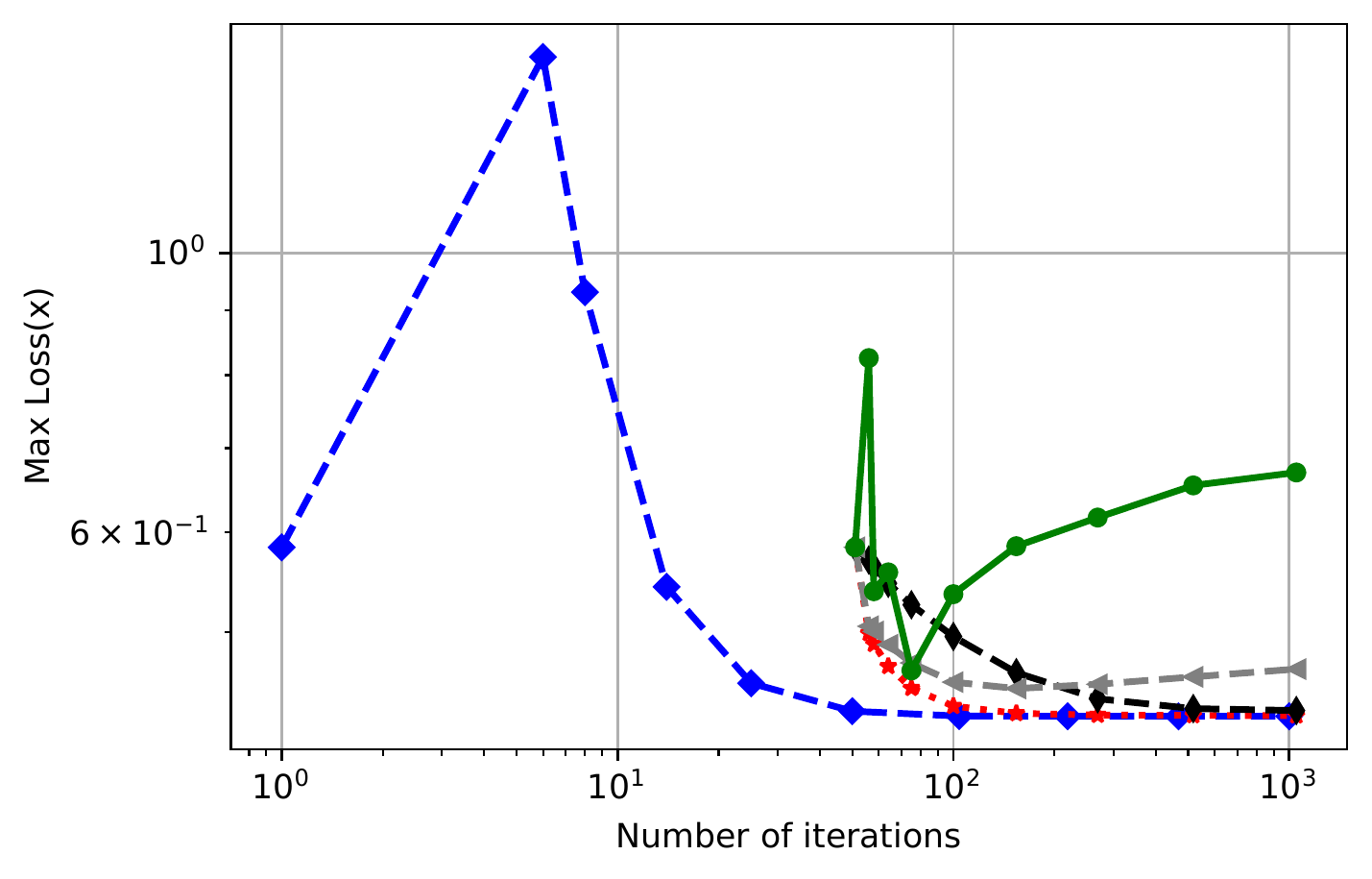}
         \caption{Uniform}
          \label{fig:dro-period-uniform-tuned}
  \end{subfigure}
   \begin{subfigure}{0.24\textwidth}
         \includegraphics[width=1.0\linewidth]{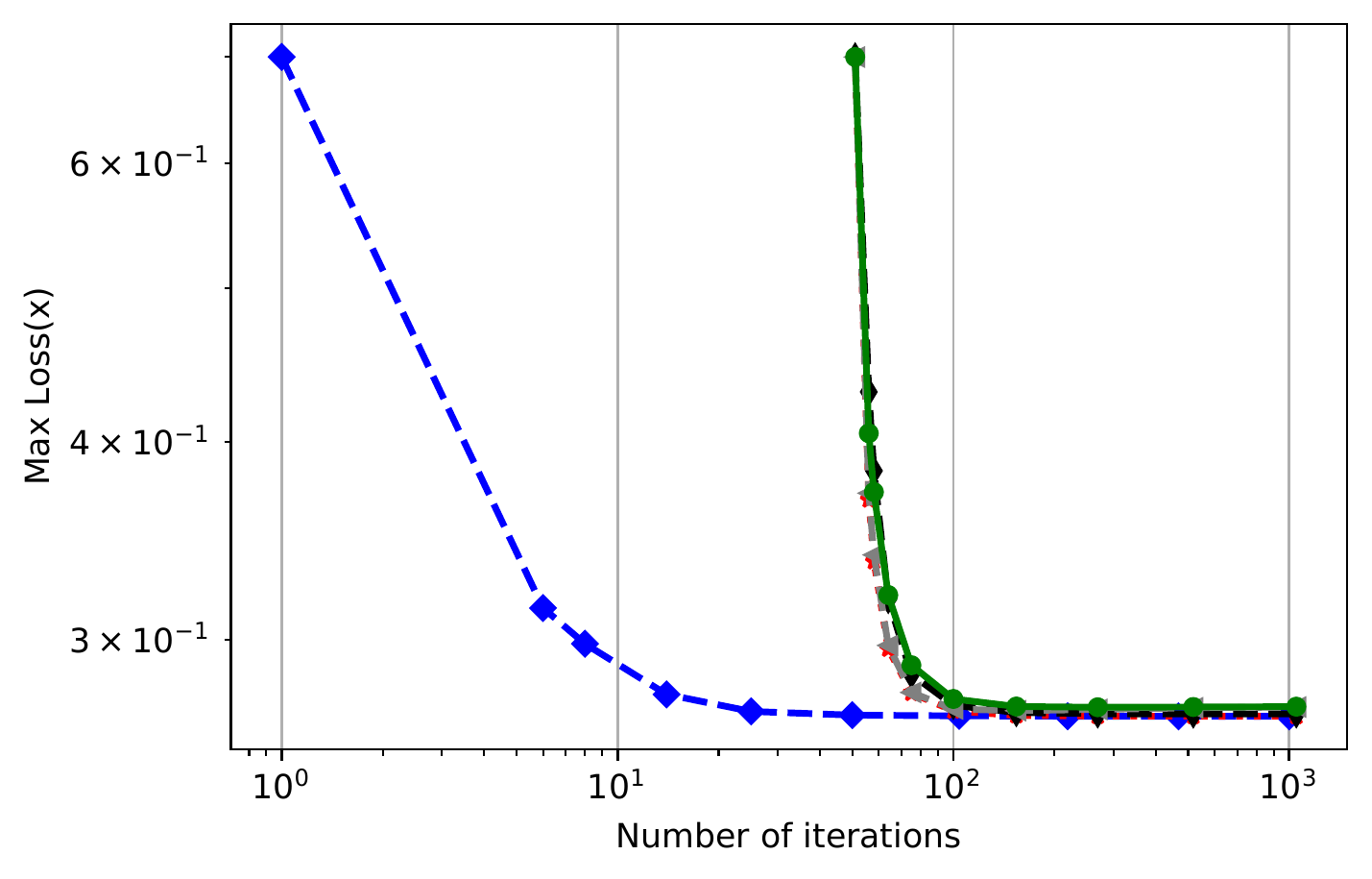}
         \caption{Normal}
          \label{fig:dro-period-normal-tuned}
  \end{subfigure}
     \begin{subfigure}{0.24\textwidth}
         \includegraphics[width=1.0\linewidth]{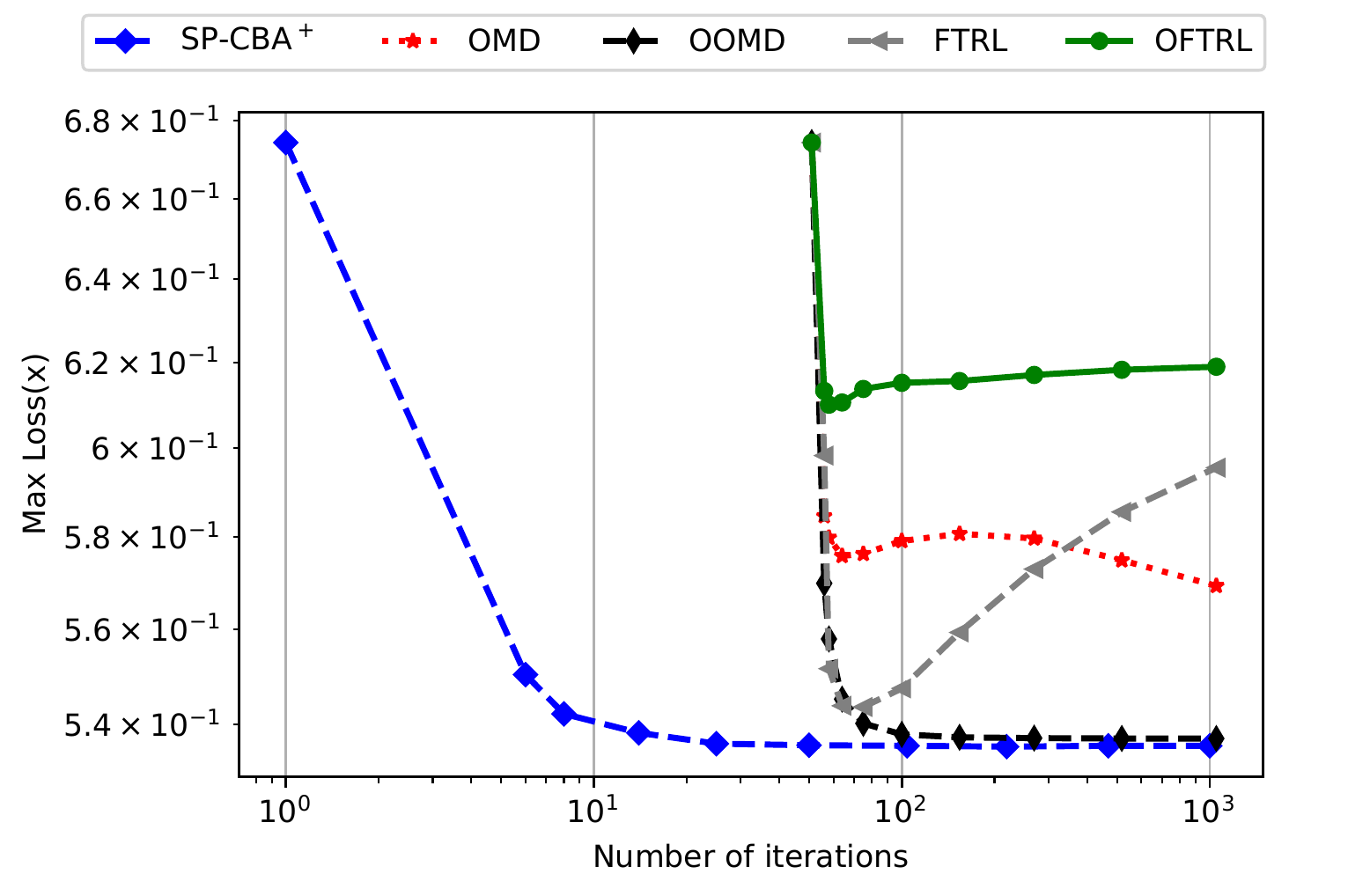}
         \caption{Australian}
          \label{fig:dro-period-australian-tuned}
  \end{subfigure}
   \begin{subfigure}{0.24\textwidth}
         \includegraphics[width=1.0\linewidth]{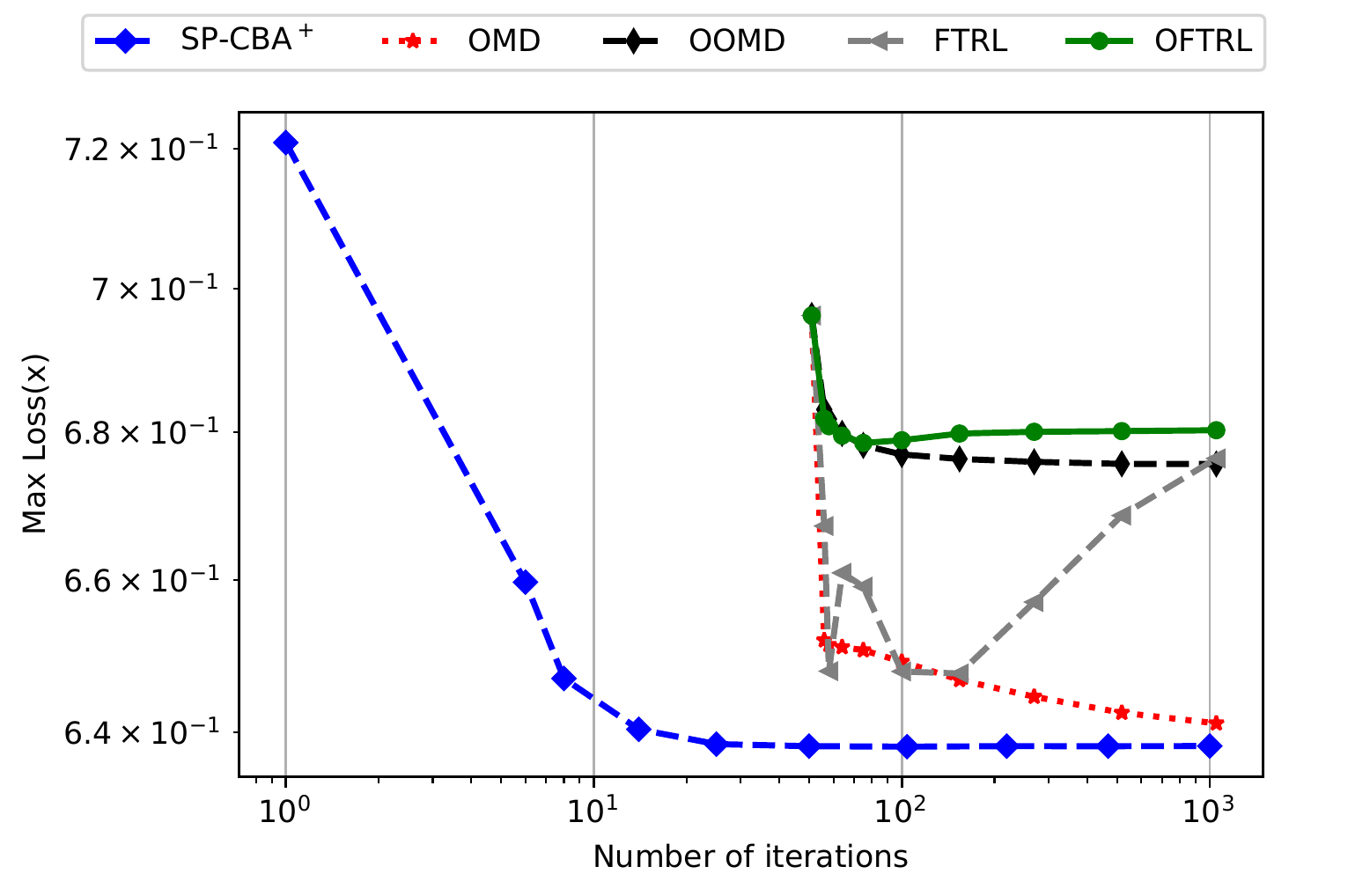}
         \caption{Splice}
          \label{fig:dro-period-splice-tuned}
  \end{subfigure}
\end{center}
  \caption{Comparisons of \cbap, OMD, FTRL, OOMD and OFTRL to compute a solution to the distributionally robust logistic regression problem~\eqref{eq:dro}, based on the number of iterations. The tuned step sizes are used in the first-order methods.}
  \label{fig:dro-period-tuned}
\end{figure}

\begin{figure}[hbt]
\begin{center}
   \begin{subfigure}{0.24\textwidth}
         \includegraphics[width=1.0\linewidth]{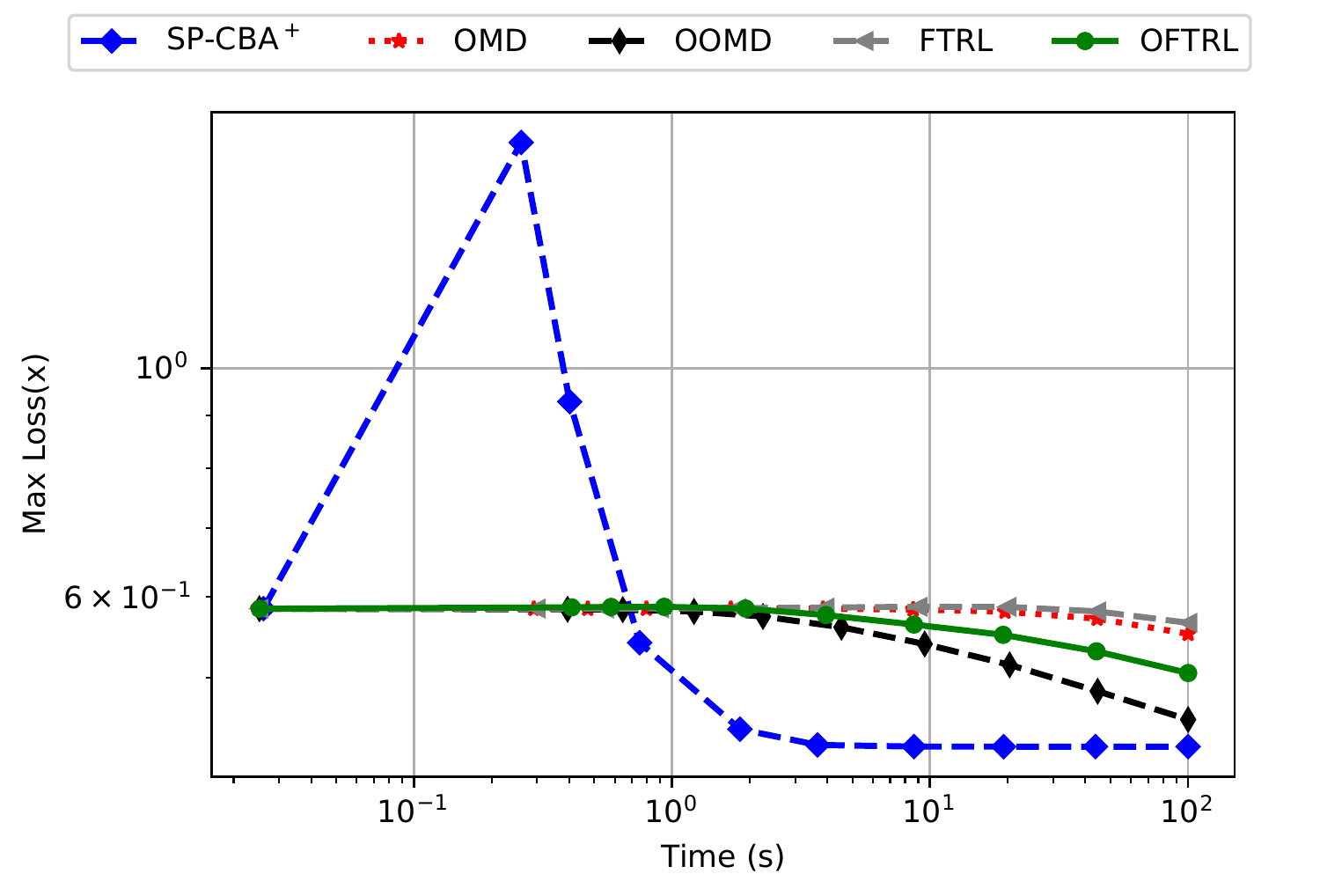}
         \caption{Uniform}
          \label{fig:dro-time-uniform-theoretical}
  \end{subfigure}
   \begin{subfigure}{0.24\textwidth}
         \includegraphics[width=1.0\linewidth]{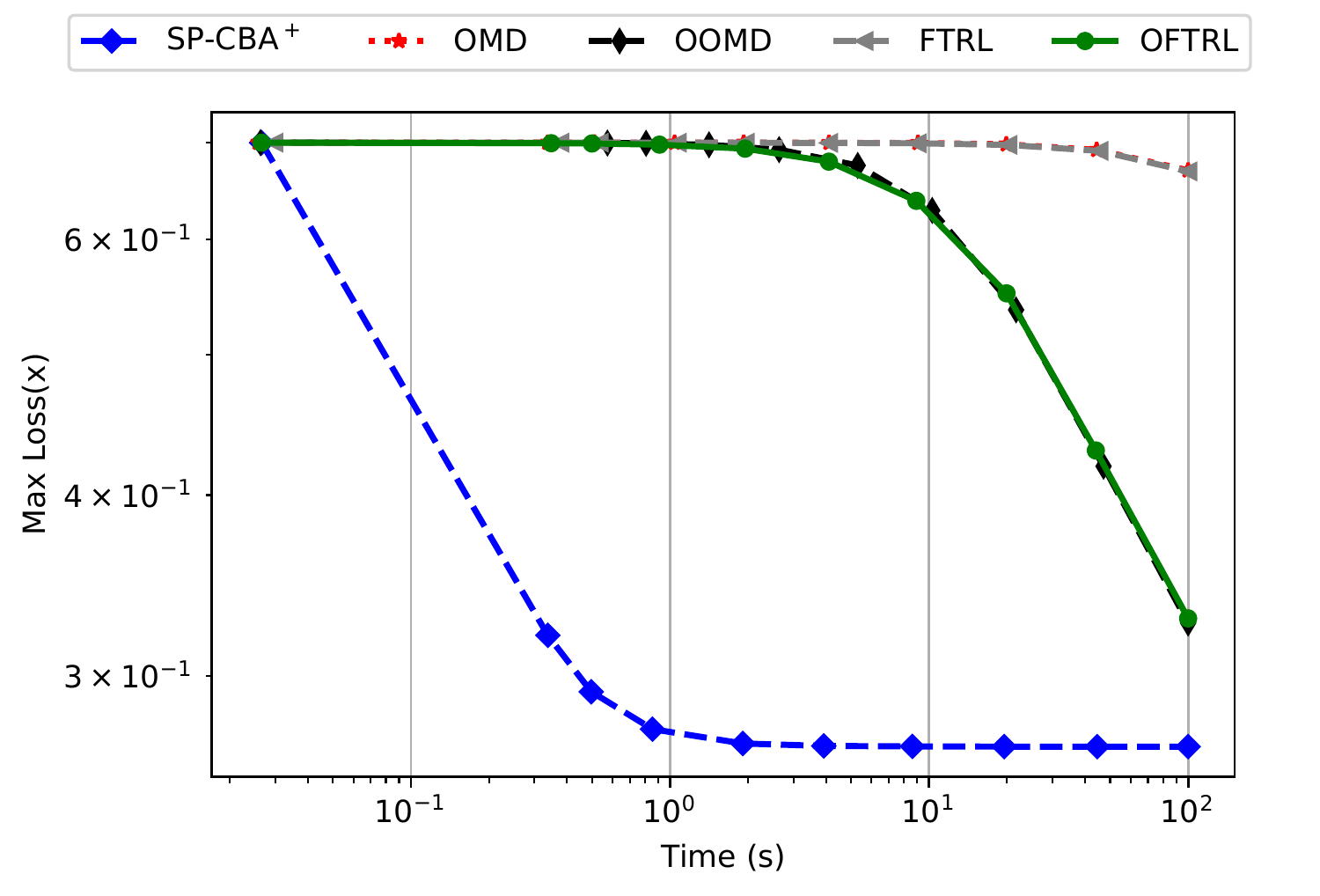}
         \caption{Normal}
          \label{fig:dro-time-normal-theoretical}
  \end{subfigure}
     \begin{subfigure}{0.24\textwidth}
         \includegraphics[width=1.0\linewidth]{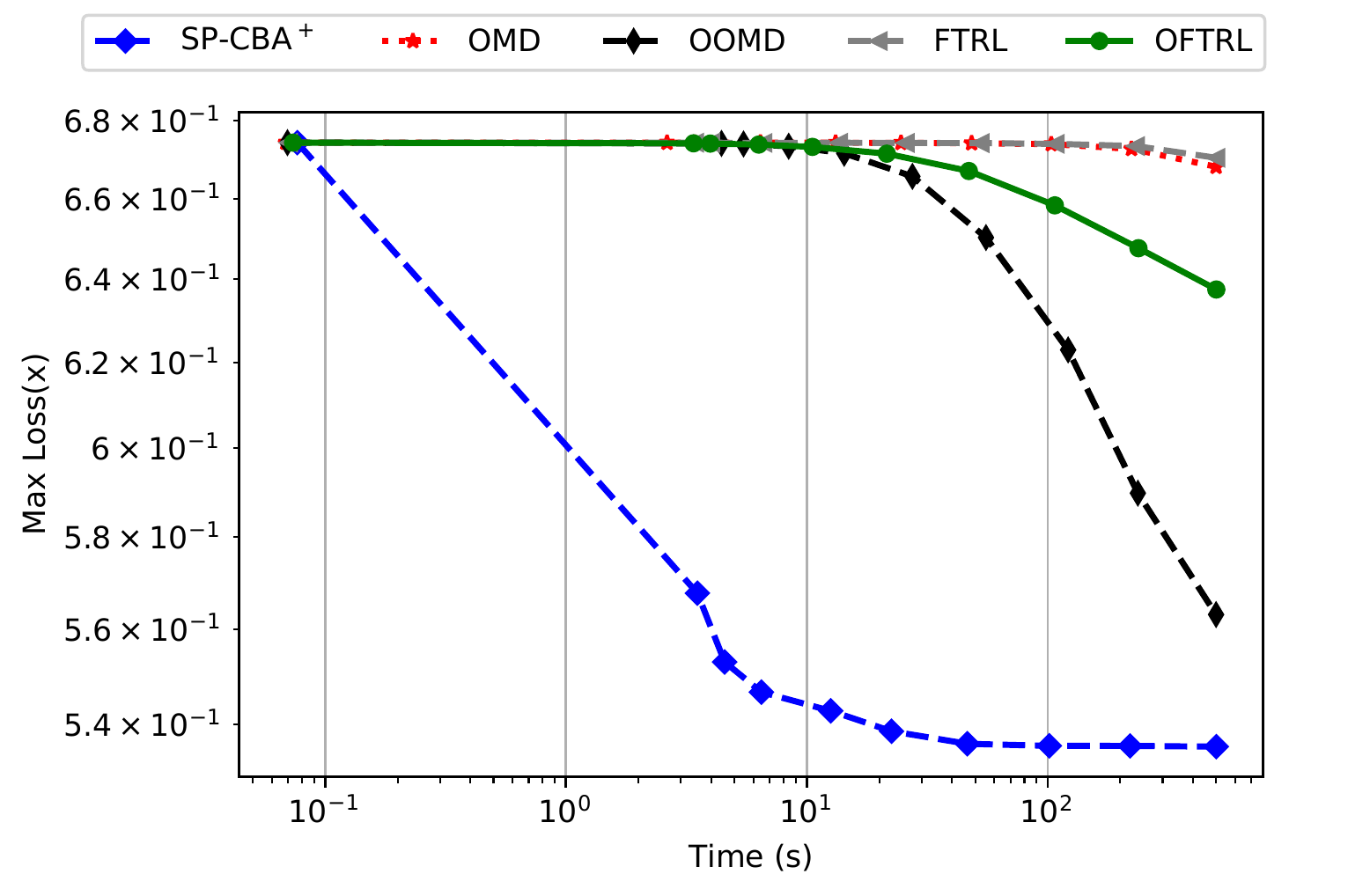}
         \caption{Australian}
          \label{fig:dro-time-australian-theoretical}
  \end{subfigure}
   \begin{subfigure}{0.24\textwidth}
         \includegraphics[width=1.0\linewidth]{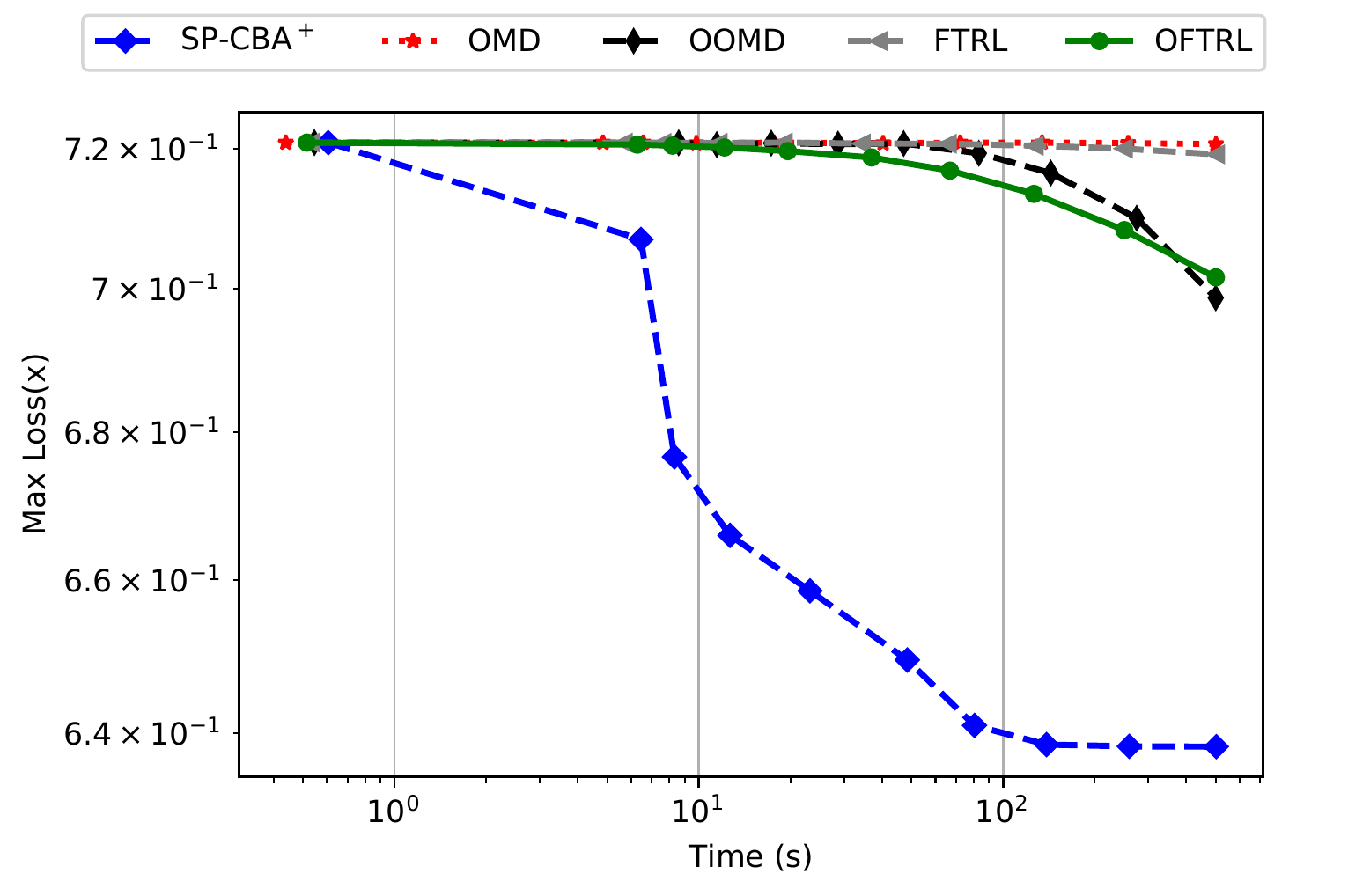}
         \caption{Splice}
          \label{fig:dro-time-splice-theoretical}
  \end{subfigure}
\end{center}
  \caption{Comparisons of \cbap, OMD, FTRL, OOMD and OFTRL to compute a solution to the distributionally robust logistic regression problem~\eqref{eq:dro}, based on the computation time. The theoretical choices of step sizes are used in the first-order methods.}
  \label{fig:dro-time-theoretical}
\end{figure}

\begin{figure}[hbt]
\begin{center}
   \begin{subfigure}{0.24\textwidth}
         \includegraphics[width=1.0\linewidth]{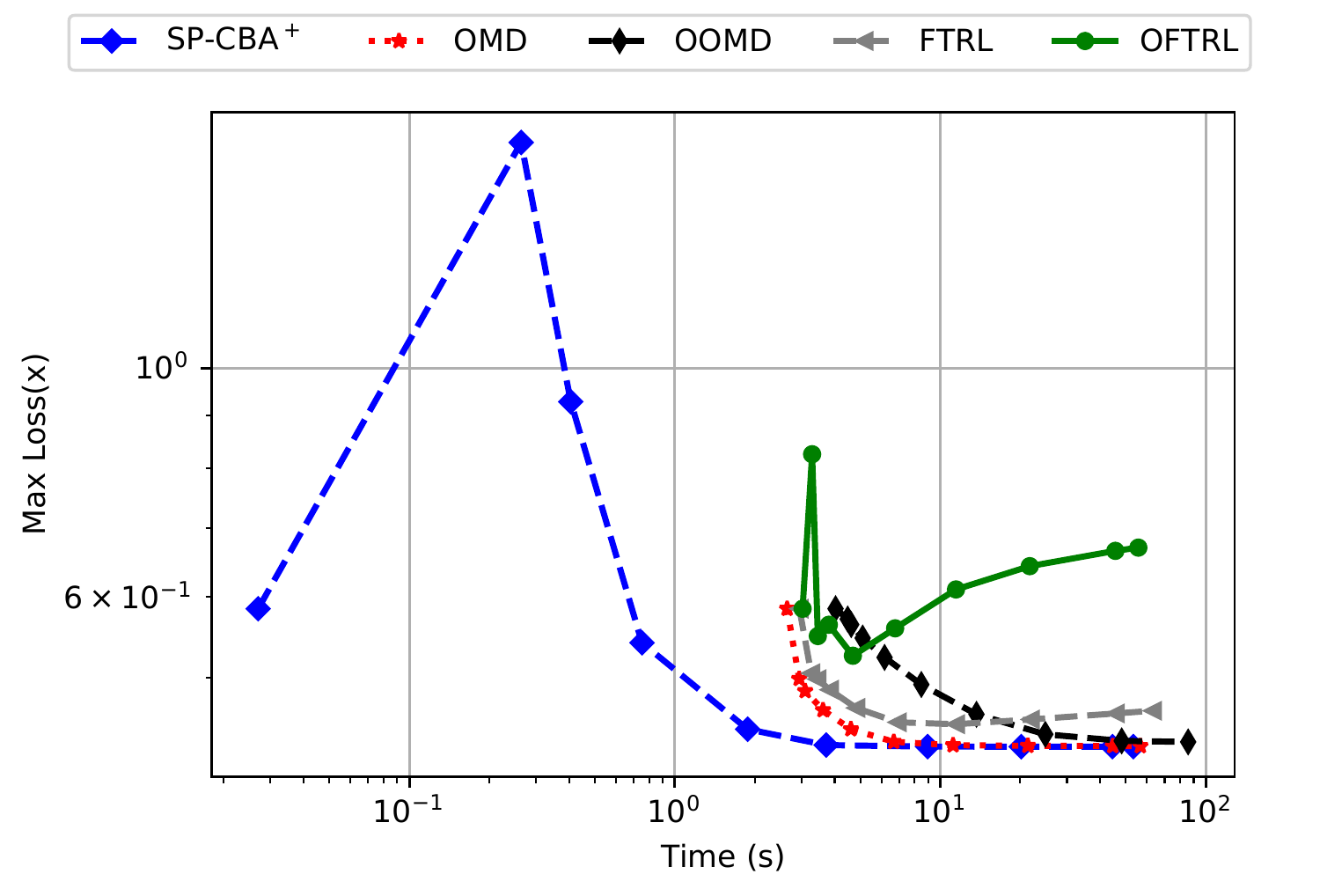}
         \caption{Uniform}
          \label{fig:dro-time-uniform-tuned}
  \end{subfigure}
   \begin{subfigure}{0.24\textwidth}
         \includegraphics[width=1.0\linewidth]{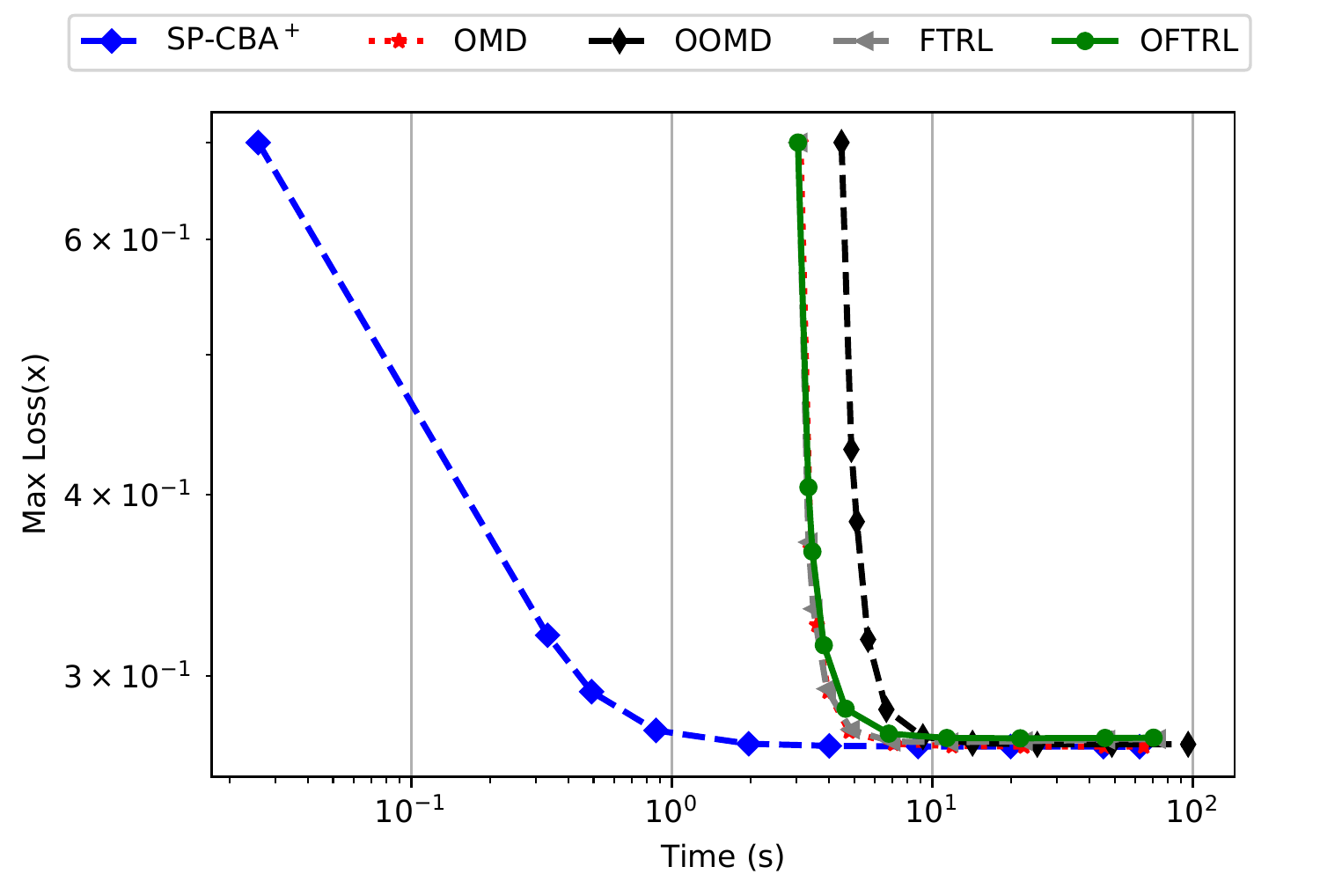}
         \caption{Normal}
          \label{fig:dro-time-normal-tuned}
  \end{subfigure}
     \begin{subfigure}{0.24\textwidth}
         \includegraphics[width=1.0\linewidth]{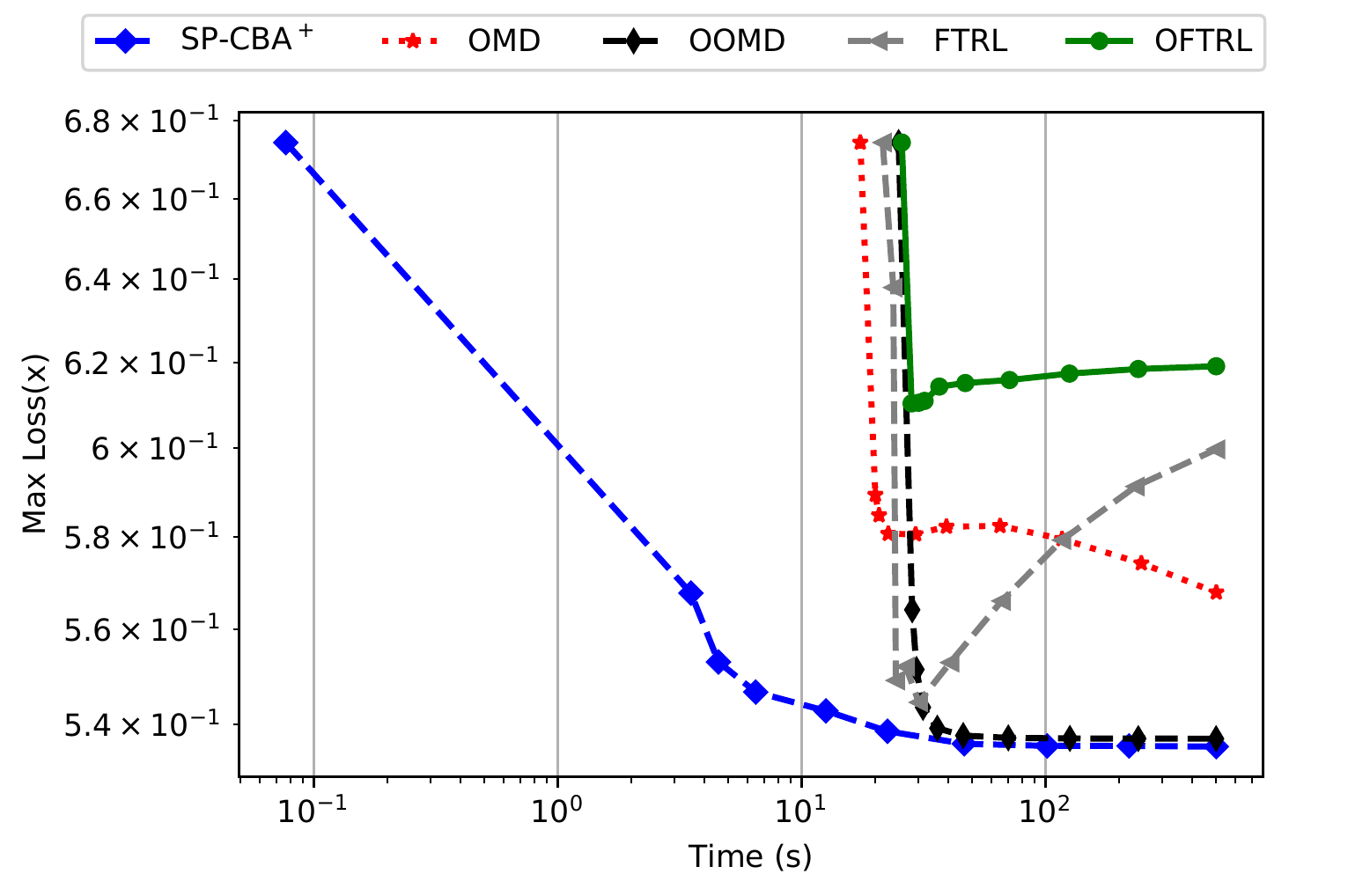}
         \caption{Australian}
          \label{fig:dro-time-australian-tuned}
  \end{subfigure}
   \begin{subfigure}{0.24\textwidth}
         \includegraphics[width=1.0\linewidth]{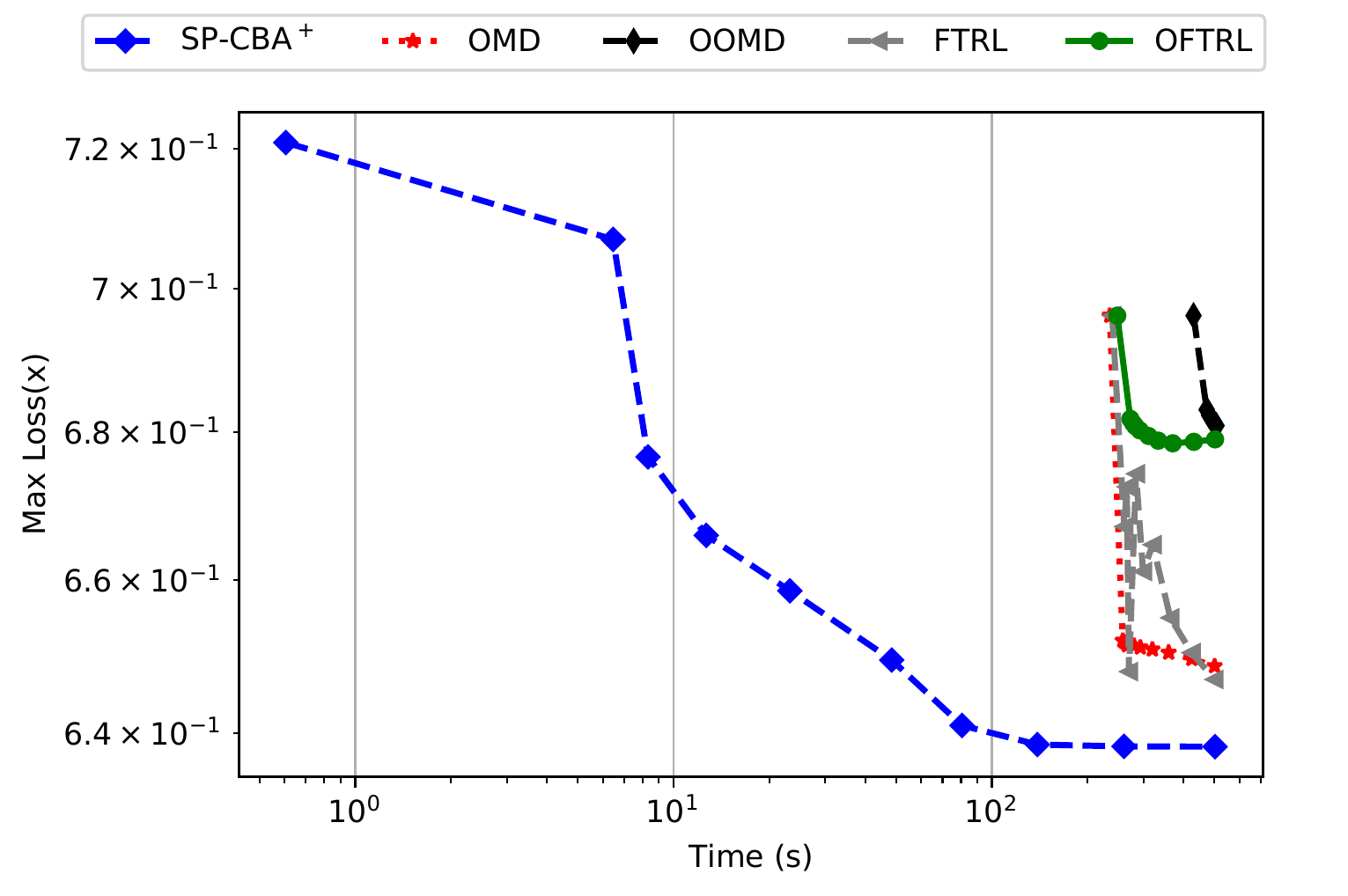}
         \caption{Splice}
          \label{fig:dro-time-splice-tuned}
  \end{subfigure}
\end{center}
  \caption{Comparisons of \cbap, OMD, FTRL, OOMD and OFTRL to compute a solution to the distributionally robust logistic regression problem~\eqref{eq:dro}, based on the computation time. The tuned step sizes are used in the first-order methods.}
  \label{fig:dro-time-tuned}
\end{figure}

\subsection{Markov decision processes}\label{sec:simu-mdp}
Markov Decision Processes (MDPs) are used as a modeling tool for sequential decision-making problems~\citep{Puterman}, and have found applications in game learning~\citep{mnih2013playing} and healthcare~\citep{grand2020robust,mdp-med-1,mdp-steimle}. In a finite MDP, the set of states is $[n]$ and there are $A$ actions. For each state-action pair $(s,a)$, there is an associated instantaneous reward $r_{sa}$ as well as a distribution $\bm{P}_{sa} \in \Delta(n)$ over the possible next states in $[n]$. We write $r_{\infty} = \max_{s,a} r_{sa}$ and we assume, without loss of generality, that $r_{sa} \geq 0, \forall \; (s,a) \in [n] \times [A]$. Given a discount factor $\lambda \in (0,1)$ and an initial probability distribution $\bm{p}_{0} \in \Delta(n)$, the goal of the decision-making is to maximize the infinite-horizon discounted cumulated reward. This leads to the following linear programming formulation~\citep{Puterman}:
\[ \min \{ (1-\lambda) \bm{p}_{0}^{\top}\bm{v} \; \vert \; v_{s} \geq r_{sa} + \lambda \bm{P}_{sa}^{\top}\bm{v}, \forall \; (s,a) \in [n] \times [A]\}\]
which can be rewritten as a saddle-point problem~\citep{jin2020efficiently}:
\begin{equation}\label{eq:mdp-saddle-point}
\min_{\bm{v} \in \mbR^{n},\| \bm{v} \|_{2} \leq \sqrt{n} r_{\infty}/(1-\lambda)} \max_{\mu \in \Delta(n \times A)} \; (1-\lambda) \bm{p}_{0}^{\top}\bm{v} + \sum_{s=1}^{n} \sum_{a=1}^{A} \mu_{sa} \left( r_{sa}+\lambda \bm{P}_{sa}^{\top}\bm{v} - v_{s}\right),
\end{equation}
where we add the constraint $\|\bm{v}\|_{2} \leq \sqrt{n}r_{\infty}/(1-\lambda)$ because \cbap{} requires a bounded decision set $\mcX$; this is a valid constraint for the optimal solution  $\bm{v}^{*} \in \mbR^{n}$ to the MDP problem, because $\bm{v}^{*}$ satisfies $0 \leq v^{*}_{s} \leq r_{\infty}/(1-\lambda), \forall \; s \in [n]$.

\paragraph{Experimental setup}
We test the performances of \spcbap{} for solving \eqref{eq:mdp-saddle-point} on random \textit{Garnet} MDPs (Generalized Average Reward Non-stationary Environment Test-bench, \citep{garnet,bhatnagar2007naturalgradient}), a class of random MDP instances widely used for benchmarking sequential decision-making algorithms. Garnet MDPs are parametrized by a branching factor $n_{b}$, which represents the proportion of reachable next states from each state-action pair $(s,a)$.  We choose $S=100,A=50,n_{b} = 50 \%,\lambda=0.95$. We average the performances of our algorithm over 10 random instances of Garnet MDPs, where the reward parameters are drawn at random uniformly in $[0,10]$. We compare \spcbap{} with the same first-order methods as in the previous section: \ref{alg:OMD}, \ref{alg:FTRL}, and their optimistic variants, with the same tuning method. The computation of the upper bounds $L_{v}$ and $L_{\mu}$ are detailed in Appendix \ref{app:constant-mdp}. We acknowledge that at the scale of the instances considered in this paper, MDPs can be solved efficiently using policy iteration. This algorithm is specialized to solving MDPs and differs greatly from \spcbap{} which is based on the repeated game framework; for this reason, we compare \spcbap{} with first-order methods that are widely applicable and that have been developed for larger MDP instances, e.g. online mirror descent for MDPs~\citep{jin2020efficiently}.

 \paragraph{Results and discussion}
 Similarly as in the two previous section, we note that \spcbap{} outperforms \ref{alg:OMD}, \ref{alg:FTRL}, as well as the optimistic variants, even after they are tuned. We note that in our tuning method, choosing the best step sizes after observing the first 10 iterations may lead to algorithms that choose step sizes that are too large and algorithms that fail to converge, such as \ref{alg:pred-OMD} in Figures \ref{fig:mdp-period-tuned}-\ref{fig:mdp-time-tuned}. Also, we note in Figure \ref{fig:mdp-time-tuned} that tuning the FOMs may require a lot of computation time. In contrast, \spcbap{} does not need to be tuned and all the computation time in \spcbap{} is used to make progress toward solving \eqref{eq:mdp-saddle-point}.
\begin{figure}[hbt]
\centering
   \begin{subfigure}{0.4\textwidth}
\centering
         \includegraphics[width=1.0\linewidth]{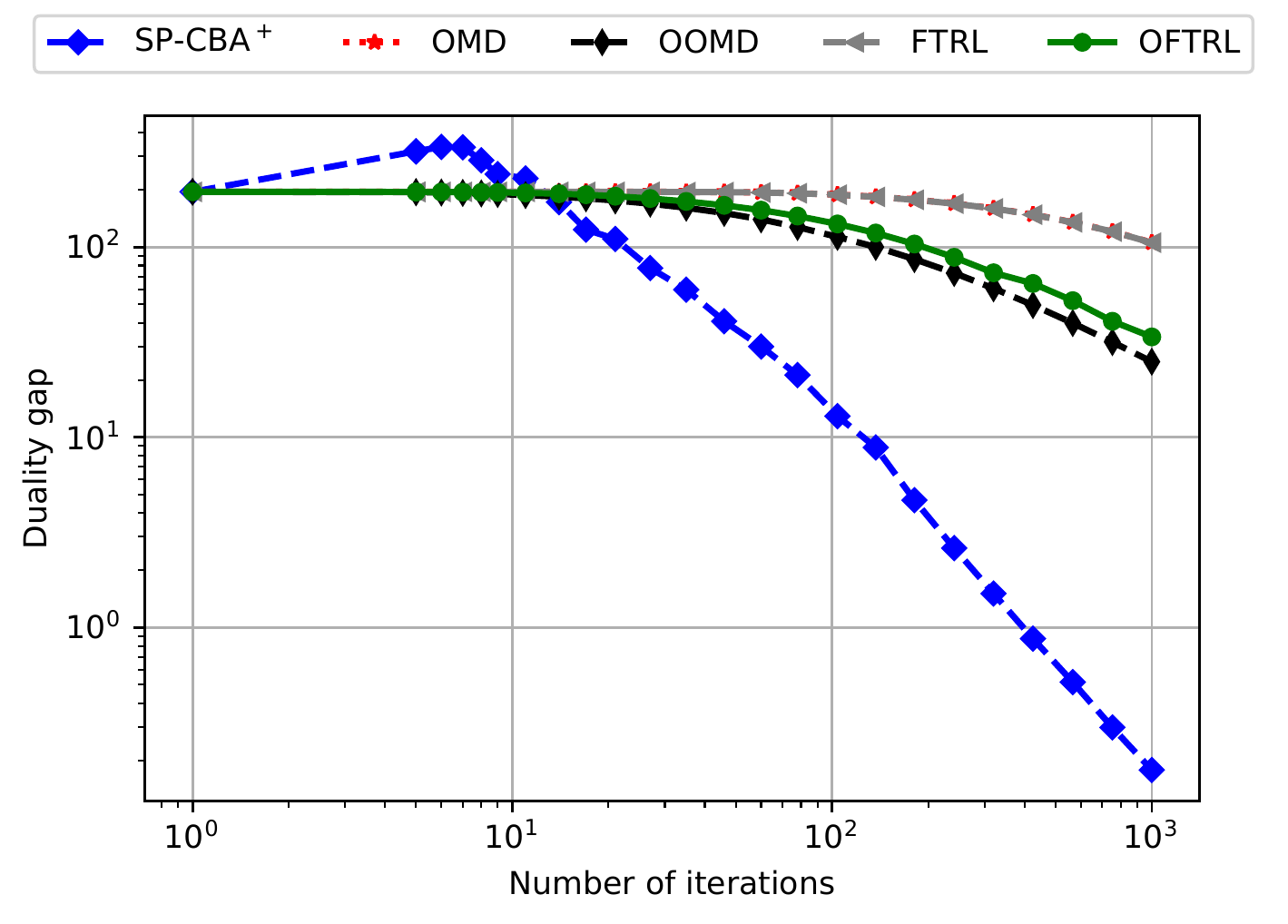}
         \caption{Number of iterations}
         \label{fig:mdp-period-theoretical}
  \end{subfigure}
     \begin{subfigure}{0.4\textwidth}
\centering
         \includegraphics[width=1.0\linewidth]{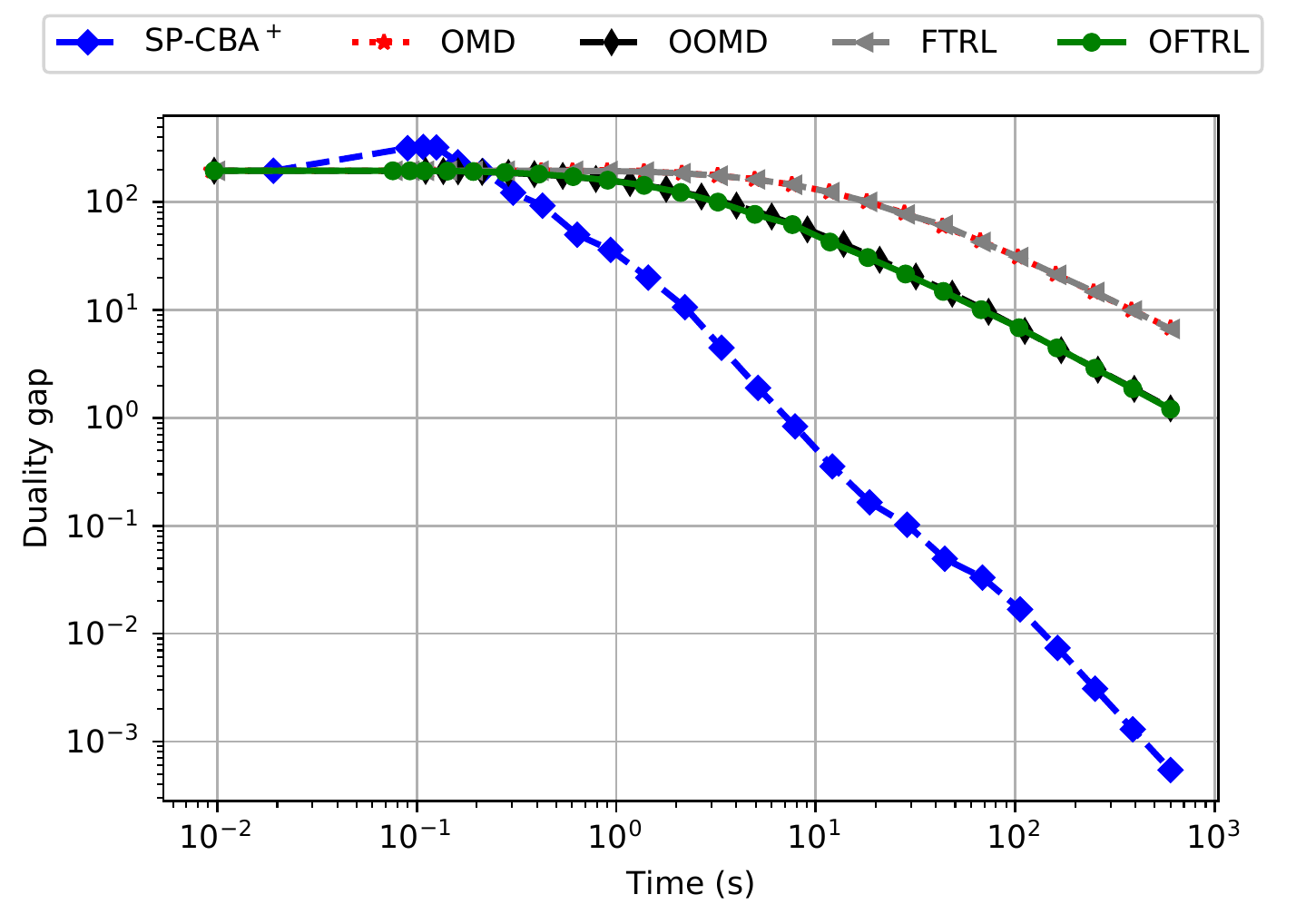}
         \caption{Computation time}
         \label{fig:mdp-time-theoretical}
  \end{subfigure}
  \caption{Comparisons of \cbap, OMD, FTRL, OOMD and OFTRL to compute a solution to the MDP problem~\eqref{eq:mdp-saddle-point}. The theoretical choices of step sizes are used in the first-order methods.}
  \label{fig:mdp-theoretical}
\end{figure}
\begin{figure}[hbt]
\centering
   \begin{subfigure}{0.4\textwidth}
\centering
         \includegraphics[width=1.0\linewidth]{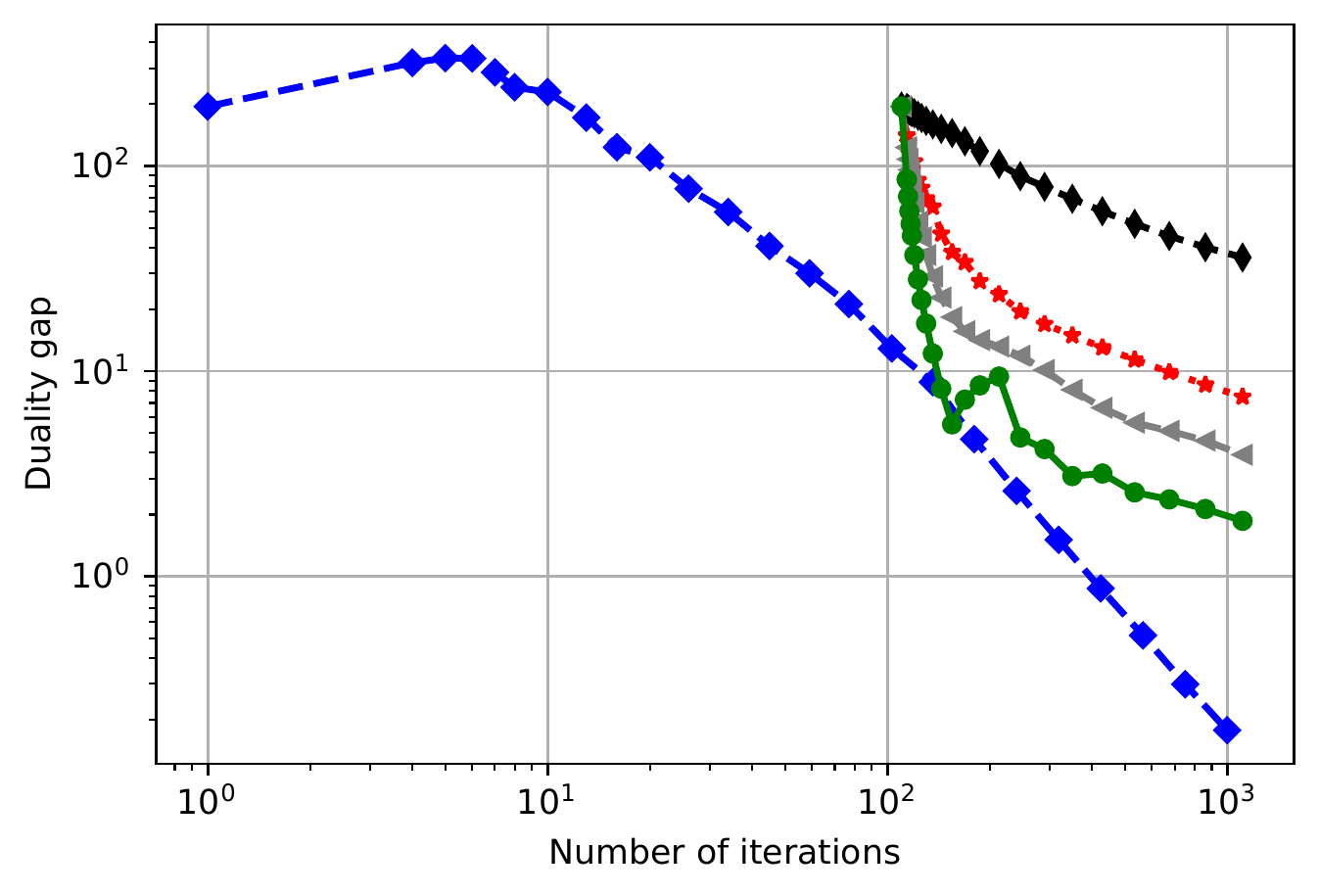}
         \caption{Number of iterations}
         \label{fig:mdp-period-tuned}
  \end{subfigure}
     \begin{subfigure}{0.4\textwidth}
\centering
         \includegraphics[width=1.0\linewidth]{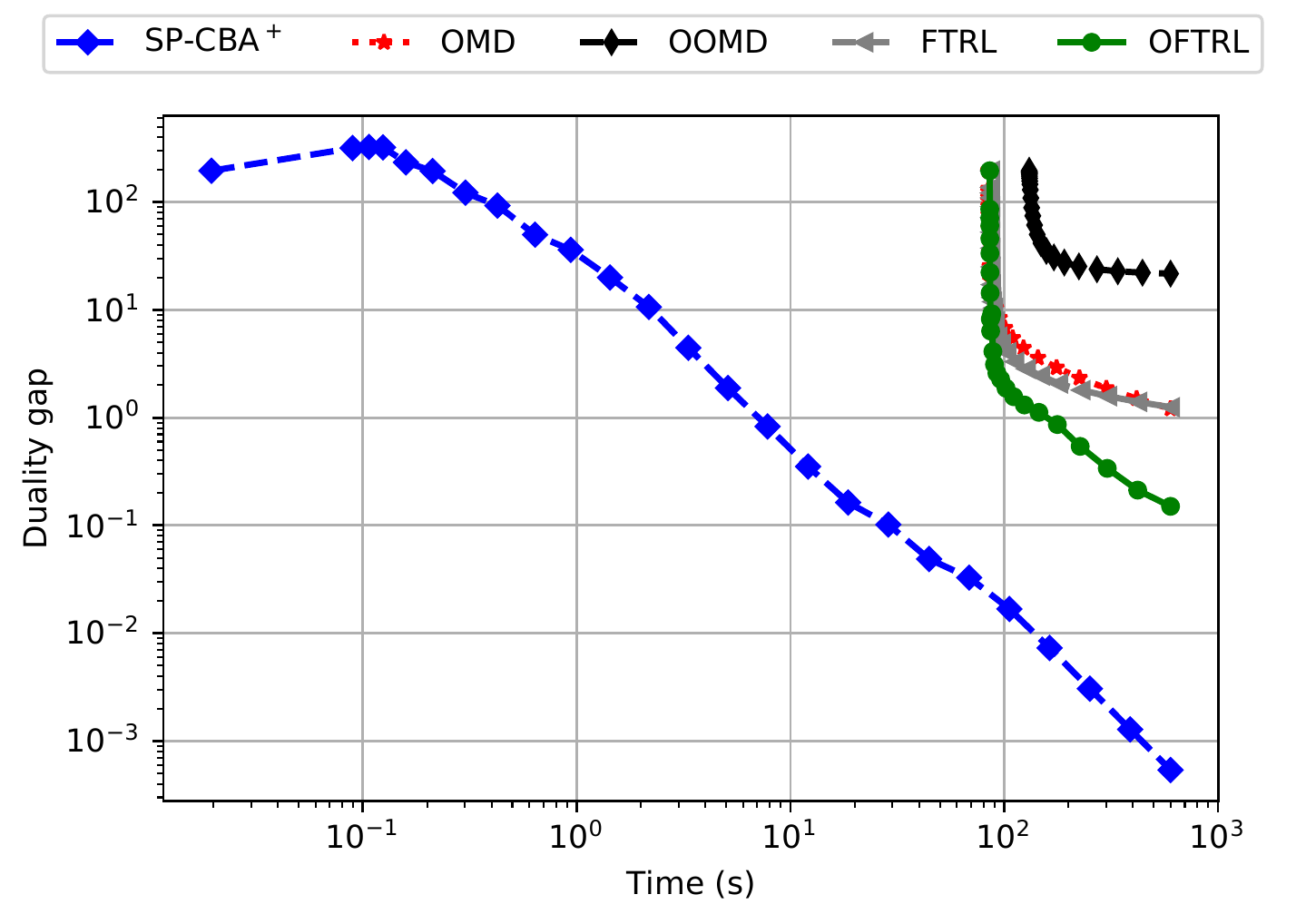}
         \caption{Computation time}
         \label{fig:mdp-time-tuned}
  \end{subfigure}
  \caption{Comparisons of \cbap, OMD, FTRL, OOMD and OFTRL to compute a solution to the MDP problem~\eqref{eq:mdp-saddle-point}. The tuned step sizes are used in the first-order methods.}
  \label{fig:mdp-tuned}
\end{figure}

\section{Conclusion}
We have proposed \spcbap, an algorithm based on Blackwell approachability for solving classical instances of saddle-point optimization. Our algorithm is  1) simple to implement for many practical decision sets,  2) completely parameter-free and does not attempt to learn any step sizes, and 3) competitive with, or even better than, state-of-the-art approaches with both theoretical and tuned parameters.  Interesting future directions of research include designing efficient implementations for other widespread decision sets (e.g., based on Kullback-Leibler divergence or $\phi$-divergence), extending \spcbap{} to unbounded decision sets, and developing novel accelerated versions based on strong convex-concavity or optimism.
{
\small 
\bibliographystyle{plainnat} 
\bibliography{ref}
}
\normalsize
\appendix

\section{Proof of Lemma \ref{lem:conic-opt}}\label{app:proof-lemma-conic-opt}
\begin{proof}[Proof of Lemma \ref{lem:conic-opt}]
\begin{enumerate}
\item The fact that $\bm{u} - \pi_{\mcC^{\circ}}(\bm{u}) = \pi_{\mcC}(\bm{u}) \in \mcC$,$\langle \bm{u} - \pi_{\mcC^{\circ}}(\bm{u}),\pi_{\mcC^{\circ}}(\bm{u}) \rangle = 0$ follows from Moreau's Decomposition Theorem \citep{combettes2013moreau}. The fact that $\| \bm{u} - \pi_{\mcC^{\circ}}(\bm{u}) \|_{2} \leq \| \bm{u} \|_{2}$ is a straightforward consequence of $\langle \bm{u} - \pi_{\mcC^{\circ}}(\bm{u}),\pi_{\mcC^{\circ}}(\bm{u}) \rangle = 0$.
\item For any $\bm{w} \in \mcC^{\circ} \cap B_{2}(1)$ we have
\[ \langle \bm{u},\bm{w} \rangle \leq  \langle \bm{u} - \pi_{\mcC}(\bm{u}),\bm{w} \rangle \leq \| \bm{w} \|_{2} \|  \bm{u} - \pi_{\mcC}(\bm{u}) \|_{2} \leq \|  \bm{u} - \pi_{\mcC}(\bm{u})\|_{2}.\]
Conversely,  since $\left( \bm{u} - \pi_{\mcC}(\bm{u}) \right)/\| \bm{u} - \pi_{\mcC}(\bm{u}) \|_{2} \in \mcC^{\circ}$, we have 
\[  \max_{\bm{w} \in \mcC^{\circ} \cap B_{2}(1) } \langle\bm{u},\bm{w} \rangle \geq \| \bm{u} - \pi_{\mcC}(\bm{u})\|_{2}.\]
This shows that 
\[ \max_{\bm{w} \in \mcC^{\circ} \cap B_{2}(1) } \langle\bm{u},\bm{w} \rangle = \| \bm{u} - \pi_{\mcC}(\bm{u})\|_{2} = d(\bm{u},\mcC).\]
\item For any $\bm{u} \in \mbR^{n+1}$,  by definition we have  $d(\bm{u},\mcC^{\circ}) = \| \bm{u} - \pi_{\mcC^{\circ}}(\bm{u})\|_{2}$. Now if $\bm{u} \in \mcC$ we have $ \pi_{\mcC^{\circ}}(\bm{u}) = 0$ so $d(\bm{u},\mcC^{\circ}) = \| \bm{u} \|_{2}$.
\item Let $\bm{u} \in \mcC$. Then $\bm{u} = \alpha (\kappa,\bm{x})$ for $\alpha \geq 0, \bm{x} \in \mcX$. We will show that $-\bm{u} \in \mcC^{\circ}$. We have
\begin{align*}
- \bm{u} \in \mcC^{\circ} &  \iff \langle - \bm{u} , \bm{u}' \rangle \leq 0, \forall \; \bm{u}' \in \mcC \\
& \iff \langle - \alpha (\kappa,\bm{x}), \alpha' (\kappa,\bm{x}') \rangle \leq 0, \forall \; \alpha' \geq 0, \forall \; \bm{x}' \in \mcX \\
& \iff   \kappa^{2} + \langle \bm{x},\bm{x}' \rangle \geq 0 \\
& \iff -  \langle \bm{x},\bm{x}' \rangle \leq \kappa^{2},
\end{align*}
and $- \langle \bm{x},\bm{x}' \rangle \leq \kappa^{2}$ is true by Cauchy-Schwartz and the definition of $\kappa = \max_{\bm{x} \in \mcX} \| \bm{x} \|_{2}$.
\item We start by proving \eqref{eq:partial-order-additivity}.  Let $ \bm{x},\bm{x}',\bm{y},\bm{y}' \in \mbR^{n+1}, $ and assume that $ \bm{x} \leq_{\mcC^{\circ}} \bm{y}, \bm{x}' \leq_{\mcC^{\circ}} \bm{y}'$. Then $\bm{y} - \bm{x} \in \mcC^{\circ},\bm{y}'- \bm{x}' \in \mcC^{\circ}$. Because $\mcC^{\circ}$ is a convex set, and a cone, we have $2 \cdot \left( \dfrac{\bm{y} - \bm{x}}{2} +  \dfrac{\bm{y}' - \bm{x}'}{2} \right) \in \mcC^{\circ}$. Therefore, $\bm{y} + \bm{y}' - \bm{x} - \bm{x}' \in \mcC^{\circ}$, i.e.,  $\bm{x} + \bm{x}' \leq_{\mcC^{\circ}} \bm{y} + \bm{y}'$.

We now prove \eqref{eq:partial-order-minus-x-prime}.  Let $\; \bm{x},\bm{y} \in \mbR^{n+1}, \bm{x}' \in \mcC^{\circ}$ and assume that $\bm{x} + \bm{x}'  \leq_{\mcC^{\circ}} \bm{y}$. Then by definition $\bm{y}  - \bm{x} -  \bm{x}'  \in \mcC^{\circ}$. Additionally, $\bm{x}' \in \mcC^{\circ}$ by assumption.  Since $\mcC^{\circ}$ is convex, and is a cone, $2 \cdot \left( \dfrac{\bm{y}  - \bm{x} -  \bm{x}'}{2}+ \dfrac{\bm{x}'}{2} \right) \in \mcC^{\circ}$, i.e., $\bm{y} - \bm{x} \in \mcC^{\circ}$. Therefore, $\bm{x} \leq_{\mcC^{\circ}} \bm{y}.$
\item Let $\bm{x},\bm{y} \in \mbR^{n+1}$ such that $\bm{x} \leq_{\mcC^{\circ}} \bm{y}$. Then $\bm{y} - \bm{x} \in \mcC^{\circ}$. We have $ d(\bm{y},\mcC^{\circ}) = \min_{\bm{z} \in \mcC^{\circ}} \| \bm{y} - \bm{z} \|_{2} \leq \| \bm{y} - ( \bm{y} - \bm{x}) \|_{2} = \| \bm{x} \|_{2}.$
\end{enumerate}
\end{proof}
\section{Proof of Theorem \ref{th:folk-final}}\label{app:proof-folk-final}
\begin{proof}[Proof of Theorem \ref{th:folk-final}]
  We prove the theorem for each part separately.
\begin{enumerate}
\item Let
\[  \bar{\bm{x}}_{T} = \frac{1}{S_{T}} \sum_{t=1}^{T} \omega_{t}\bm{x}_{t}, \bar{\bm{y}}_{T} = \frac{1}{S_{T}} \sum_{t=1}^{T} \omega_{t}\bm{y}_{t}.\] 
Since $F$ is convex-concave,  we first have
\[ \max_{\bm{y} \in \mcY}F(\bar{\bm{x}}_{T},\bm{y}) - \min_{\bm{x} \in \mcX} F(\bm{x},\bar{\bm{y}}_{T}) \leq \frac{1}{S_{T}} \left(  \max_{\bm{y} \in \mcY}  \sum_{t=1}^{T} \omega_{t}F(\bm{x}_{t},\bm{y}) - \min_{\bm{x} \in \mcX} \sum_{t=1}^{T} \omega_{t}F(\bm{x},\bm{y}_{t}) \right).\]
Now,
\begin{align*}
 \max_{\bm{y} \in \mcY}  \sum_{t=1}^{T} \omega_{t}F(\bm{x}_{t},\bm{y}) - \min_{\bm{x} \in \mcX} \sum_{t=1}^{T} \omega_{t}F(\bm{x},\bm{y}_{t})  & =  \max_{\bm{y} \in \mcY}  \sum_{t=1}^{T} \omega_{t}F(\bm{x}_{t},\bm{y}) -  \sum_{t=1}^{T} \omega_{t} F(\bm{x}_{t},\bm{y}_{t}) \\
 & + \sum_{t=1}^{T} \omega_{t} F(\bm{x}_{t},\bm{y}_{t})  - \min_{\bm{x} \in \mcX} \sum_{t=1}^{T} \omega_{t}F(\bm{x},\bm{y}_{t}) .
\end{align*}
Now since $F$ is convex-concave, we can  upper bound each pair of terms using the subgradient inequality:
\begin{align*}
\max_{\bm{y} \in \mcY}  \sum_{t=1}^{T} \omega_{t}F(\bm{x}_{t},\bm{y}) -  \sum_{t=1}^{T} \omega_{t} F(\bm{x}_{t},\bm{y}_{t}) &  \leq \max_{\bm{y} \in \mcY}  \omega_{t} \sum_{t=1}^{T} \langle \bm g_t ,\bm{y} \rangle - \sum_{t=1}^{T}  \omega_{t} \langle \bm g_t,\bm{y}_{t} \rangle, \\
 \sum_{t=1}^{T} \omega_{t} F(\bm{x}_{t},\bm{y}_{t})  - \min_{\bm{x} \in \mcX} \sum_{t=1}^{T} \omega_{t}F(\bm{x},\bm{y}_{t}) & \leq  \sum_{t=1}^{T}  \omega_{t} \langle \bm f_t, \bm x_t \rangle - \min_{\bm{x} \in \mcX} \sum_{t=1}^{T} \omega_{t} \langle \bm f_t, \bm x \rangle,
\end{align*}
where $\bm f_t \in \partial_{\bm x} F(\bm x_t,\bm y_t),\bm g_t \in \partial_{\bm y} F(\bm x_t,\bm y_t)$ (recall the repeated game framework presented  at the beginning of Section \ref{sec:game-setup}).
We recognize the right-hand side as the regrets in the repeated game framework.
For \cba{} with weights on both payoffs and decisions (Theorem \ref{th:cba-linear-averaging-both}), we have shown that
\begin{align*}
\frac{1}{S_{T}} \max_{\bm{y} \in \mcY} \sum_{t=1}^{T} \omega_{t} \langle \bm g_t ,\bm{y} \rangle - \sum_{t=1}^{T} \langle  \omega_{t} \bm g_t,\bm{y}_{t} \rangle & = O \left( \kappa L  \frac{\sqrt{\sum_{t=1}^{T} \omega_{t}^{2}}}{\sum_{t=1}^{T} \omega_{t} } \right), \\
\frac{1}{S_{T}} \sum_{t=1}^{T}  \omega_{t} \langle \bm f_t, \bm x_t \rangle - \min_{\bm{x} \in \mcX} \sum_{t=1}^{T}  \omega_{t} \langle \bm f_t, \bm x \rangle & = O \left( \kappa L \frac{\sqrt{\sum_{t=1}^{T} \omega_{t}^{2}}}{\sum_{t=1}^{T} \omega_{t} } \right).
\end{align*}
Recall that $\omega_{t}=t^p$.
Since $t \mapsto t^p$ is an increasing function, we have
\[ \int_{0}^{k} t^{p}dt \leq \sum_{t=1}^{k} t^{p} \leq \int_{0}^{k+1} t^{p}dt.\]
Therefore, we can conclude that
\begin{align*}
\sum_{t=1}^{T} \omega_{t}^{2} & = O \left( \frac{1}{p+1}T^{2p+1} \right), \\
\frac{1}{p+1}T^{p+1} & \leq \sum_{t=1}^{T} \omega_{t}.
\end{align*}
Overall, we obtain that 
\[ O \left( \kappa L \frac{\sqrt{\sum_{t=1}^{T} \omega_{t}^{2}}}{\sum_{t=1}^{T} \omega_{t} } \right) = O \left(  \frac{\kappa L \sqrt{p+1}}{\sqrt{T}}\right).\]
\item This proof is mostly similar to the first part. We have
\begin{align*}
\theta_{T} & = T^q, \\
 \frac{T^{q+1}}{q+1} & \leq \sum_{t=1}^{T} \theta_{t},\\
 \sqrt{\sum_{t=1}^{T} \omega^{2}_{t}} & = O \left( \frac{1}{\sqrt{p+1}}T^{p+1/2}\right), \\
 \omega_{T} & = T^{p}.
\end{align*}
Combining all this we obtain that an upper bound of
\[
O\left(\kappa L \frac{(q+1)T^q}{T^{q+1}} \frac{T^{p+1/2}}{\sqrt{p+1}T^{p}}\right)
\]
which is equal to 
$O \left(\frac{\kappa L(q+1)}{\sqrt{p+1}\sqrt{T}} \right).$
\end{enumerate}
\end{proof}
\section{Proof for Theorem \ref{th:folk-theorem-alternation}}\label{app:proof-folk-alternation}
\begin{proof}[Proof of Theorem \ref{th:folk-theorem-alternation}]
The proof of Theorem \ref{th:folk-theorem-alternation} is similar to the proof of Theorem \ref{th:folk-final} presented in Appendix \ref{app:proof-folk-final}. 
Let
\[  \bar{\bm{x}}_{T} = \frac{1}{S_{T}} \sum_{t=1}^{T} \theta_{t+1}\bm{x}_{t+1}, \bar{\bm{y}}_{T} = \frac{1}{S_{T}} \sum_{t=1}^{T} \theta_{t+1}\bm{y}_{t}.\] 
Since $F$ is convex-concave,  we first have
\[ \max_{\bm{y} \in \mcY}F(\bar{\bm{x}}_{T},\bm{y}) - \min_{\bm{x} \in \mcX} F(\bm{x},\bar{\bm{y}}_{T}) \leq \frac{1}{S_{T}} \left(  \max_{\bm{y} \in \mcY}  \sum_{t=1}^{T} \theta_{t+1}F(\bm{x}_{t+1},\bm{y}) - \min_{\bm{x} \in \mcX} \sum_{t=1}^{T} \theta_{t+1}F(\bm{x},\bm{y}_{t}) \right).\]
Now we can rewrite 
\[ \max_{\bm{y} \in \mcY}  \sum_{t=1}^{T} \theta_{t+1}F(\bm{x}_{t+1},\bm{y}) - \min_{\bm{x} \in \mcX} \sum_{t=1}^{T} \theta_{t+1}F(\bm{x},\bm{y}_{t}) \]
as 
\begin{align*}
 &  \max_{\bm{y} \in \mcY}  \sum_{t=1}^{T} \theta_{t+1}F(\bm{x}_{t+1},\bm{y}) -  \sum_{t=1}^{T} \theta_{t+1} F(\bm{x}_{t+1},\bm{y}_{t}) \\
 & + \sum_{t=1}^{T} \theta_{t+1} F(\bm{x}_{t+1},\bm{y}_{t})  - \sum_{t=1}^{T} \theta_{t+1} F(\bm{x}_{t},\bm{y}_{t}) \\ 
  & + \sum_{t=1}^{T} \theta_{t+1} F(\bm{x}_{t},\bm{y}_{t}) - \min_{\bm{x} \in \mcX} \sum_{t=1}^{T} \theta_{t+1}F(\bm{x},\bm{y}_{t}).
\end{align*}
Now since $F$ is convex-concave, we can use the following upper bound:
\begin{align*}
\max_{\bm{y} \in \mcY}  \sum_{t=1}^{T} \theta_{t+1}F(\bm{x}_{t+1},\bm{y}) -  \sum_{t=1}^{T} \theta_{t+1} F(\bm{x}_{t+1},\bm{y}_{t}) &  \leq \max_{\bm{y} \in \mcY}  \theta_{t+1} \sum_{t=1}^{T} \langle \bm g_t ,\bm{y} \rangle - \sum_{t=1}^{T}  \theta_{t+1} \langle \bm g_t,\bm{y}_{t} \rangle, \\
 \sum_{t=1}^{T} \theta_{t+1} F(\bm{x}_{t},\bm{y}_{t})  - \min_{\bm{x} \in \mcX} \sum_{t=1}^{T} \theta_{t+1}F(\bm{x},\bm{y}_{t}) & \leq  \sum_{t=1}^{T}  \theta_{t+1} \langle \bm f_t, \bm x_t \rangle - \min_{\bm{x} \in \mcX} \sum_{t=1}^{T} \theta_{t+1} \langle \bm f_t, \bm x \rangle,
\end{align*}
where $\bm f_t \in \partial_{\bm x} F(\bm x_{t},\bm y_t),\bm g_t \in \partial_{\bm y} F(\bm x_{t+1},\bm y_t)$.
This concludes the proof of Theorem \ref{th:folk-theorem-alternation}.
\end{proof} 
\section{Proof of Theorem \ref{th:alternation-works-bilinear-case}}\label{app:proof-alternation-is-improving}
We start with the following lemma. It shows that once a non-degenerate update has been chosen ($\bm{u}_{t} \neq \bm{0}$ for \cbap{} and $\pi_{\mcC}(\bm{u}_{t}) \neq \bm{0}$ for \cba), all the future updates are also non-degenerate.
\begin{lemma}\label{lem:non-degenerate-update}
\begin{enumerate}
\item Let $\left(\bm{u}_{t}\right)_{t \geq 1} \in \left(\mbR^{n+1}\right)^{\mbN}$ the sequence of payoffs generated by \cba{} with weights $\left(\omega_{t}\right)_{t \geq 1}$ on the payoffs.
Let $t \geq 1$. If $\pi_{\mcC}\left(\bm{u}_{t}\right) \neq \bm{0}$, then for all $t' \geq t$ we also have $\pi_{\mcC}\left(\bm{u}_{t'}\right) \neq \bm{0}$.
\item  Let $\left(\bm{u}_{t}\right)_{t \geq 1} \in \left(\mbR^{n+1}\right)^{\mbN}$ the sequence of payoffs generated by \cbap{} with weights $\left(\omega_{t}\right)_{t \geq 1}$ on the payoffs.
Let $t \geq 1$. If $\bm{u}_{t} \neq \bm{0}$, then for all $t' \geq t$ we also have $\bm{u}_{t'} \neq \bm{0}$.
\end{enumerate}  
\end{lemma}
\begin{proof}[Proof of Lemma \ref{lem:non-degenerate-update}]
\begin{enumerate}
\item 
Assume that $\pi_{\mcC}\left(\bm{u}_{t}\right) \neq \bm{0}$. Let $\bm{\pi}_{t} = \left(\tilde{\pi}_{t},\hat{\bm{\pi}}_{t}\right)$ such that $\bm{\pi}_{t} = \pi_{\mcC}\left(\bm{u}_{t}\right)$. 
In this case, we can define $\bm{x}_{t+1} = \left(\kappa / \tilde{\pi}_{t} \right) \hat{\bm{\pi}}_{t}$. By definition of the updates in \cba, we have
\[ \bm{u}_{t+1} =\bm{u}_{t} + \omega_{t+1} \bm{v}_{t+1},\]
for $\bm{v}_{t+1} =  \left(\frac{\langle \bm{f}_{t+1},\bm{x}_{t+1} \rangle}{\kappa},-\bm{f}_{t+1}\right).$
We will show that $\bm{u}_{t+1} \notin \mcC^{\circ}$.
By definition, 
\[ \bm{u}_{t+1} \notin \mcC^{\circ} \iff \exists \; \bm{z} \in \mcC, \langle \bm{z},\bm{u}_{t+1}\rangle > 0.\]
If we take $\bm{z} = \bm{\pi}_{t}$, we have
\[ \langle \bm{\pi}_{t}, \bm{u}_{t+1} \rangle = \langle \bm{\pi}_{t}, \bm{u}_{t} + \omega_{t+1} \bm{v}_{t+1}\rangle = \langle \bm{\pi}_{t}, \bm{u}_{t} \rangle\]
since that by definition of $\bm{x}_{t+1}$, we have $\langle \bm{\pi}_{t} ,\bm{v}_{t+1} \rangle = 0$. Now 
\[ \langle \bm{\pi}_{t}, \bm{u}_{t} \rangle = \langle \pi_{\mcC}\left(\bm{u}_{t}\right), \pi_{\mcC}\left(\bm{u}_{t}\right)+\pi_{\mcC^{\circ}}\left(\bm{u}_{t}\right) \rangle = \langle \pi_{\mcC}\left(\bm{u}_{t}\right), \pi_{\mcC}\left(\bm{u}_{t}\right) \rangle   = \| \pi_{\mcC}\left(\bm{u}_{t}\right) \|_{2}^{2} >0.\]
This shows that $\bm{u}_{t+1} \notin \mcC^{\circ}$. Since $\bm{u}_{t+1} = \pi_{\mcC}(\bm{u}_{t+1}) + \pi_{\mcC^{\circ}}(\bm{u}_{t+1})$, this also shows that $\pi_{\mcC}(\bm{u}_{t+1}) \neq \bm{0}$. By induction, we have shown that $\pi_{\mcC}\left(\bm{u}_{t}\right) \neq \bm{0} \Rightarrow \pi_{\mcC}\left(\bm{u}_{t'}\right) \neq \bm{0}, \forall \; t' \geq t$.
\item The proof is very similar to the proof of the first statement.
Suppose that $\bm{u}_{t} \neq \bm{0}$. In this case, we can define $\bm{x}_{t+1} = \left(\kappa / \tilde{u}_{t} \right) \hat{\bm{u}}_{t}$. Note that by definition of the updates in \cbap, we have
\[ \bm{u}_{t+1} = \pi_{\mcC}\left(\bm{u}_{t} + \omega_{t+1} \bm{v}_{t+1}\right).\]
We will show that 
\[\bm{u}_{t} + \omega_{t+1} \bm{v}_{t+1} \notin \mcC^{\circ}.\]
By definition of $\mcC^{\circ}$,
\[ \bm{u}_{t} + \omega_{t+1} \bm{v}_{t+1} \notin \mcC \iff \exists \; \bm{z} \in \mcC, \langle \bm{z}, \bm{u}_{t} + \bm{v}_{t+1} \rangle >0.\]
For $\bm{z} = \bm{u}_{t}$, we obtain
\[\langle \bm{u}_{t},\bm{u}_{t} + \omega_{t+1}\bm{v}_{t+1} \rangle  
= \langle \bm{u}_{t},\bm{u}_{t} \rangle = \| \bm{u}_{t} \|_{2}^{2} > 0,\]
where
\[ \langle \bm{u}_{t} ,\bm{v}_{t+1} \rangle = 0\]
follows from the choice of $\bm{x}_{t+1}$ as in Blackwell approachability framework (see \eqref{eq:blackwell-forcing} in the proof of Theorem \ref{th:cba-linear-averaging-both} for more details.) 
Therefore, for any $t\geq 1$, we have $\bm{u}_{t} \neq \bm{0} \Rightarrow \bm{u}_{t+1} \neq \bm{0}$. This concludes the proof of Lemma \ref{lem:non-degenerate-update} by induction.
\end{enumerate}
\end{proof}
We are now ready to prove Theorem \ref{th:alternation-works-bilinear-case}.
\begin{proof}[Proof of Theorem \ref{th:alternation-works-bilinear-case}]
Assume that $\left(\bm{x},\bm{y} \right) \mapsto F\left(\bm{x},\bm{y} \right)$ is linear in $\bm{x}$.
\begin{enumerate}
\item We want to prove that
\begin{equation}\label{eq:x-t-plus-one-improving-objective}
F(\bm{x}_{t},\bm{y}_{t})  \geq F(\bm{x}_{t+1},\bm{y}_{t})+ \frac{\kappa}{ \omega_{t} \| \bm{u}_{t}\|_{\infty}} \| \bm{u}_{t} - \bm{u}_{t-1} \|_{2}^{2}.
\end{equation}
Let $t \geq 1$. Recall that
\begin{align*}
\bm{x}_{t} & = \chp_{\cbap}(\bm{u}_{t-1}),\\
\bm{x}_{t+1} & = \chp_{\cbap}(\bm{u}_{t}).
\end{align*}
We consider the following two cases.
\begin{enumerate}
\item Case 1: $\bm{u}_{t} = \bm{0}$.
From Lemma \ref{lem:non-degenerate-update}, we must have $\bm{u}_{t-1} = \bm{0}$, in which case $\bm{x}_{t+1} = \bm{x}_{t} = \bm{x}_{0}$ (the default value for the decisions of the first player), so that \eqref{eq:x-t-plus-one-improving-objective} holds because every term is $0$, with the convention that $0/0=0$ (in case $\bm{u}_{t+1}=\bm{0}$).
\item Case 2: $\bm{u}_{t} \neq \bm{0}$. 
We start from
\[\bm{u}_{t} = \pi_{\mcC}\left(\bm{u}_{t-1} + \omega_{t} \bm{v}_{t} \right)\]
with $\bm{v}_{t} = \left(\frac{\langle \bm{f}_{t},\bm{x}_{t} \rangle}{\kappa},-\bm{f}_{t}\right)$.
The optimality condition for the projection on $\mcC$ shows that
\[ \langle \bm{u}_{t} - \bm{u}_{t-1} - \omega_{t} \bm{v}_{t}, \bm{u}_{t} - \bm{z} \rangle \leq 0, \forall \; \bm{z} \in \mcC.\]
We can apply this with $\bm{z} = \bm{u}_{t-1}$ to obtain
\[ \langle  \bm{u}_{t} - \bm{u}_{t-1} - \omega_{t} \bm{v}_{t}, \bm{u}_{t} - \bm{u}_{t-1} \rangle \leq 0.\]
This shows that
\[ \| \bm{u}_{t} - \bm{u}_{t-1} \|_{2}^{2}  \leq \langle \omega_{t} \bm{v}_{t},
\bm{u}_{t} - \bm{u}_{t-1} \rangle.\]
Recall that by definition of $\bm{x}_{t}$ and $\bm{v}_{t}$, we have 
\[\langle \bm{v}_{t},
\bm{u}_{t-1} \rangle=0.\]
Recall that $\bm{u}_{t}=\alpha_{t+1}\left(\kappa,\bm{x}_{t+1}\right)$, with $\alpha_{t+1}>0$ because $\bm{u}_{t} \neq \bm{0}$.
This implies that 
\begin{align*}
\langle \omega_{t} \bm{v}_{t},
\bm{u}_{t} - \bm{u}_{t-1} \rangle
& = \langle \omega_{t} \bm{v}_{t},\bm{u}_{t} \rangle \\
& =  \omega_{t}\langle\left(\frac{\langle \bm{f}_{t},\bm{x}_{t} \rangle}{\kappa},-\bm{f}_{t}\right), \alpha_{t+1}\left(\kappa,\bm{x}_{t+1}\right) \rangle \\
& = \omega_{t}  \alpha_{t+1} \left( \langle \bm{f}_{t},\bm{x}_{t} \rangle - \langle \bm{f}_{t},\bm{x}_{t+1} \rangle \right).
\end{align*}
Overall, we have obtained
\[ \langle \bm{f}_{t},\bm{x}_{t} \rangle \geq \langle \bm{f}_{t},\bm{x}_{t+1} \rangle + \frac{1}{ \omega_{t} \alpha_{t+1}} \| \bm{u}_{t} - \bm{u}_{t-1} \|_{2}^{2}.\]
Recall that by definition, $\bm{u}_{t} = \alpha_{t+1}\left(\kappa,\bm{x}_{t+1}\right)$, with $\kappa = \max \{ \| \bm{x} \|_{2} \vert \bm{x} \in \mcX\}$. Therefore,
\[ \| \bm{u}_{t} \|_{\infty} = \alpha_{t+1} \max \{ \kappa, \| \bm{x}_{t+1}\|_{\infty}\} = \alpha_{t+1} \kappa,\]
where the last inequality follows from $\| \bm{x}_{t+1}\|_{\infty}\leq \| \bm{x}_{t+1}\|_{2} \leq \kappa$.
Overall, we have shown that \[ \langle \bm{f}_{t},\bm{x}_{t} \rangle \geq \langle \bm{f}_{t},\bm{x}_{t+1} \rangle + \frac{\kappa}{ \omega_{t} \| \bm{u}_{t} \|_{\infty}} \| \bm{u}_{t} - \bm{u}_{t-1} \|_{2}^{2}.\]
\end{enumerate}
Recall that in the repeated game framework with alternation, we have $\bm{f}_{t} = \partial_{\bm{x}} F(\bm{x}_{t},\bm{y}_{t})$. 
For an objective function that is linear in $\bm{x}$, we obtain 
\begin{align*}
\langle \bm{f}_{t},\bm{x}_{t+1} \rangle & = F(\bm{x}_{t+1},\bm{y}_{t}), \\
\langle \bm{f}_{t},\bm{x}_{t} \rangle & = F(\bm{x}_{t},\bm{y}_{t}).
\end{align*}
In this case, we have shown that 
\[ F(\bm{x}_{t},\bm{y}_{t})  \geq F(\bm{x}_{t+1},\bm{y}_{t})+ \frac{\kappa}{\omega_{t}  \| \bm{u}_{t} \|_{\infty}} \|  \bm{u}_{t} - \bm{u}_{t-1} \|_{2}^{2}.\]
This concludes the proof of the first statement of Theorem \ref{th:alternation-works-bilinear-case}.
\item The proof is identical to the first claim of this theorem. For the sake of conciseness, we omit it in this paper.
\end{enumerate}
\end{proof}
\section{Proofs for the efficient projections of Section \ref{sec:efficient-implementation}}\label{app:efficient-implementation}
\subsection{Proofs for the simplex}
\begin{proof}[Proof of Lemma \ref{lem:simplex-C-polar}] For $\mcX = \Delta(n)$, we can choose $\kappa = \max \{ \| \bm{x} \|_{2} \; \vert \; \bm{x} \in \mcX \} = 1$. Therefore, $\mcC = \{ \alpha \left(1,\bm{x}\right) \; \vert \; \bm{x} \in \Delta(n), \alpha \geq 0 \}$.
For $\bm{y}=(\tilde{y},\hat{\bm{y}}) \in \mbR^{n+1}$ we have
\begin{align*}
\bm{y} \in \mcC^{\circ} &  \iff \langle \bm{y},\bm{z} \rangle \leq 0, \forall \; \bm{z} \in \mcC \\
& \iff \langle (\tilde{y},\hat{\bm{y}}), \alpha(1,\bm{x}) \rangle \leq 0, \forall \; \bm{x} \in \Delta(n), \forall \; \alpha \geq 0 \\
& \iff \tilde{y} + \langle \hat{\bm{y}},\bm{x} \rangle \leq 0, \forall \; \bm{x} \in \Delta(n) \\
& \iff \max_{\bm{x} \in \Delta(n)} \langle \hat{\bm{y}},\bm{x} \rangle \leq - \tilde{y} \\
& \iff \max_{i=1,...,n} \hat{y}_{i} \leq - \tilde{y}.
\end{align*}
\end{proof}
\begin{proof}[Proof of Proposition \ref{prop:proj-simplex}]

Let us fix $\tilde{y} \in \mbR$ and let us first solve
\begin{equation}\label{eq:y-hat-of-y-tilde}
\begin{aligned}
\min \;  &  \| \hat{\bm{y}}-\hat{\bm{u}} \|_{2}^{2} \\
&\hat{\bm{y}} \in \mbR^{n}, \\
& \max_{i \in [n]} \hat{y}_{i} \leq - \tilde{y}.
\end{aligned}
\end{equation}
This is essentially the projection of $\hat{\bm{u}}$ on $(-\infty,-\tilde{y}]^{n}$. So a solution to \eqref{eq:y-hat-of-y-tilde} is $ \hat{y}_{i}(\tilde{y}) = \min \{-\tilde{y}, \hat{u}_{i} \}, \forall \; i=1,...,n.$
Note that in this case we have
$ \hat{\bm{u}} - \hat{\bm{y}}(\tilde{y}) = \left( \hat{\bm{u}} + \tilde{y}\bm{e} \right)^{+}.$
So overall the orthogonal projection on $\mcC^{\circ}$ boils down to the optimization of the function $\phi: \mbR \mapsto \mbR_{+}$ such that
\begin{equation}\label{eq:function-phi}
\phi: \tilde{y} \mapsto  (\tilde{y}-\tilde{u} )^{2} + \| \left( \hat{\bm{u}} + \tilde{y}\bm{e} \right)^{+} \|_{2}^{2}.
\end{equation}
In principle, we could use binary search with a doubling trick to compute a $\epsilon$-minimizer of the convex function $\phi$ in $O\left( \log(\epsilon^{-1}) \right)$ calls to $\phi$. However, it is possible to find a minimizer $\tilde{y}^{*}$ of $\phi$ using the following remark.

By construction,  we know that $\bm{u} - \pi_{\mcC^{\circ}}(\bm{u}) \in \mcC$. Here,  $\mcC = \textrm{cone}\left( \{1\} \times \Delta(n) \right)$,  and $\bm{u} - \pi_{\mcC^{\circ}}(\bm{u}) = \left( \tilde{u}-\tilde{y}^{*},\left( \hat{\bm{u}} + \tilde{y}^{*}\bm{e} \right)^{+} \right).$ 
We first check if $\tilde{u}=\tilde{y}^{*}$. This is the case if and only if $\bm{u} - \pi_{\mcC^{\circ}}(\bm{u}) = \bm{0}$, i.e., if and only if $\bm{u} \in \mcC^{\circ}$, which is straightforward to check using Lemma \ref{lem:simplex-C-polar}. Now if $\tilde{u}\neq\tilde{y}^{*}$, we must have $\tilde{u}>\tilde{y}^{*}$, by definition of $\mcC$. This also implies that
\[ \dfrac{\left( \hat{\bm{u}} + \tilde{y}^{*}\bm{e} \right)^{+} }{\tilde{u}-\tilde{y}^{*}} \in \Delta(n),\]
which in turns imply that 
\begin{equation}\label{eq:simple-eq-y-tilde}
 \tilde{y}^{*}+ \sum_{i=1}^{n} \max \{ \hat{u}_{i} + \tilde{y}^{*},0 \}=\tilde{u}.
\end{equation}
We can use \eqref{eq:simple-eq-y-tilde} to efficiently compute $\tilde{y}^{*}$ without using any binary search. In particular, we can sort the coefficients of $\hat{\bm{u}}$ in $O\left( n \log(n) \right)$ arithmetic operations, and use \eqref{eq:simple-eq-y-tilde} to find $\tilde{y}^{*}$.
\end{proof}
\subsection{Proofs for $\ell_{p}$-balls}
\begin{proof}[Proof of Lemma \ref{lem:ball-p-C-polar}]
Let us write $B_{p}(1) = \{ \bm{z} \in \mbR^{n} \; \vert \; \| \bm{z} \|_{p} \leq 1\}$. Here we consider $\mcX = B_{p}(1)$. Recall that $\kappa = \max \{ \| \bm{x} \|_{2} \; \vert \; \bm{x} \in \mcX\}$. Therefore, by definition, $\mcC = \{ \alpha \left(\kappa,\bm{x}\right) \; \vert \; \bm{x} \in B_{p}(1), \alpha \geq 0 \}$.

We first provide the reformulation for $\mcC$. Let $\bm{y}=(\tilde{y},\hat{\bm{y}}) \in \mcC$.
Then $\tilde{y} = \alpha \kappa, \hat{\bm{y}} = \alpha \bm{x}$ with $\alpha \geq 0$ and with $\bm{x}$ such that $\| \bm{x} \|_{p} \leq 1$. For $\alpha >0$ we have $\| \bm{x} \|_{p} \leq 1 \iff \| \alpha \bm{x} \|_{p} \leq \alpha \iff \| \hat{\bm{y}} \|_{p} \leq \tilde{y}/\kappa$.

We now provide the reformulation for $\mcC^{\circ}$.
 Note that for $\bm{y}=(\tilde{y},\hat{\bm{y}}) \in \mbR^{n+1}$ we have
\begin{align*}
\bm{y} \in \mcC^{\circ} &  \iff \langle \bm{y},\bm{z} \rangle \leq 0, \forall \; \bm{z} \in \mcC \\
& \iff \langle (\tilde{y},\hat{\bm{y}}), \alpha(\kappa,\bm{x}) \rangle \leq 0, \forall \; \bm{x} \in B_{p}(1), \forall \; \alpha \geq 0 \\
& \iff \kappa \tilde{y} + \langle \hat{\bm{y}},\bm{x} \rangle \leq 0, \forall \; \bm{x} \in B_{p}(1), \\
& \iff \max_{\bm{x} \in B_{p}(1),} \langle \hat{\bm{y}},\bm{x} \rangle \leq - \kappa \tilde{y} \\
& \iff \| \hat{\bm{y}} \|_{q} \leq - \kappa \tilde{y},
\end{align*}
since $\| \cdot \|_{q}$ is the dual norm of $\| \cdot \|_{p}$.
\end{proof}
\begin{proof}[Proof of Proposition \ref{prop:proj-ball-1-C-polar}] For $p=1$, we have $\| \cdot \|_{q} = \| \cdot \|_{\infty}$ and we can choose $\kappa = 1$.
Let us compute the projection of $\left(\tilde{u},\hat{\bm{u}} \right)$ on $\mcC^{\circ}$ using the reformulation of Lemma \ref{lem:ball-p-C-polar}:
\begin{equation}\label{eq:ortho-proj-onto-polar-norm-1}
\begin{aligned}
 \min & \; (\tilde{y}-\tilde{u})^{2} + \| \hat{\bm{y}}-\hat{\bm{u}} \|_{2}^{2} \\
& \; \tilde{y} \in \mbR, \hat{\bm{y}} \in \mbR^{n}, \\
& \| \hat{\bm{y}} \|_{\infty} \leq -\tilde{y}.
\end{aligned}
\end{equation}
For a fixed $\tilde{y} \in \mbR$, we want to compute $
 \min \{ \| \hat{\bm{y}}-\hat{\bm{u}} \|_{2}^{2} \; \vert \; \hat{\bm{y}} \in \mbR^{n}, \| \hat{\bm{y}} \|_{\infty} \leq -\tilde{y} \}.$
 This projection can be computed in closed-form as
$
 \hat{\bm{y}}^{*}(\tilde{y})  =  \min \{ - \tilde{y}, \max \{ \tilde{y}, \hat{\bm{u}} \} \}$,
since this is simply the orthogonal projection of $\hat{\bm{u}}$ onto the $\ell_{\infty}$ ball of radius $-\tilde{y}$.  Let us call $\phi: \mbR \mapsto \mbR$ such that
\[ \phi(\tilde{y}) = \left( \tilde{y} - \tilde{u} \right)^{2} + \| \hat{\bm{y}}^{*}(\tilde{y}) - \hat{\bm{u}} \|_{2}^{2}.\]
Note that $\hat{\bm{y}}^{*}(\tilde{y}) - \hat{\bm{u}} = \left( \hat{\bm{u}} + \tilde{y}\bm{e} \right)^{+}$, so we have
 \[ \phi: \tilde{y} \mapsto  \left( \tilde{y} - \tilde{u} \right)^{2}  + \| \left( \hat{\bm{u}} + \tilde{y}\bm{e} \right)^{+} \|_{2}^{2}.\]
Assume that we have ordered the coefficients of $\hat{\bm{u}} \in \mbR^{n}$ in decreasing order. This can be done in $O\left(n\log(n)\right)$ arithmetic operations. Then on each of the $n+1$ intervals $\mcI_{1}=\left(-\infty,-\hat{u}_{1}\right),\mcI_{2}=\left(-\hat{u}_{1},-\hat{u}_{2}\right),...,\mcI_{n+1} =\left(-\hat{u}_{n},+\infty\right)$, the map $\phi$ is a second order polynomial in $\tilde{y}$, with a non-negative coefficient in front of $\tilde{y}^2$. Therefore, for each $i \in [n+1]$, we can find a closed-form expression for the minimum $\phi^*_{i}$ of $\phi$ on $\mcI_{i}$, and the scalar $\tilde{y}^*_{i}$ attaining this minimum. We can then simply search for a global minimum of $\phi$ among the scalars
 \[ \{ \tilde{y}^*_{i} \; \vert \; i \in [n+1]\} \bigcup \{ - \hat{u}_{i} \; \vert \; i \in [n]\}.\]
 Once we have found $\tilde{y}^{*}$ the minimizer of $\phi$, we obtain the solution of $\pi_{\mcC^{\circ}}(\bm{u})$ as $\pi_{\mcC^{\circ}}(\bm{u})=\left(\tilde{y}^{*},\hat{\bm{y}}^{*}(\tilde{y})\right)$, and we can recover $\pi_{\mcC}(\bm{u})$ from $\pi_{\mcC}(\bm{u})=\bm{u}-\pi_{\mcC^{\circ}}(\bm{u})$.

Let us now focus on the case $p=\infty$.
We know that $\| \cdot \|_{1}$ and $\| \cdot \|_{\infty}$ are dual norms to each other.
 Therefore, from Lemma \ref{lem:ball-p-C-polar}, it is as computationally demanding to compute orthogonal projections onto $\mcC^{\circ}$ (when $p=1$) and onto $\mcC$ (when $p=\infty$). 
Therefore, the method described in the first part of this proof for computing $\pi_{\mcC^{\circ}}(\bm{u})$ for $p=1$  can be applied for computing $\pi_{\mcC}(\bm{u})$ in the case $p=\infty$.
\end{proof}
\begin{proof}[Proof of Proposition \ref{prop:proj-ball-2-C-polar}]
First, we check if $\bm{u} \in \mcC$, i.e., we check if $\| \hat{\bm{u}} \|_{2} \leq \tilde{u}$. If this is the case, then $\pi_{\mcC}(\bm{u}) = \bm{u}$.
Second, we check if $\bm{u} \in \mcC^{\circ}$, i.e., we check if $\| \hat{\bm{u}} \|_{2} \leq - \tilde{u}$. If this is the case, then $\pi_{\mcC}(\bm{u}) = \bm{0}$. Else, we have $\| \hat{\bm{u}} \|_{2} > \vert \tilde{u}\vert$, and we can provide a closed-form solution to $\pi_{\mcC}(\bm{u})$. 
Let us fix $\tilde{y} \in \mbR$ and define $\hat{\bm{y}}^{*}(\tilde{y})$ the vector attaining the minimum in 
$\min \{\| \hat{\bm{y}}-\hat{\bm{u}} \|_{2}^{2} \; \vert \;  \hat{\bm{y}} \in \mbR^{n}, \| \hat{\bm{y}} \|_{2} \leq \tilde{y}\}.$ With this notation, we want to find the minimum of $\phi: \mbR \mapsto \mbR$ defined as
\[ \phi(\tilde{y}) = \left( \tilde{y} - \tilde{u} \right)^{2} + \| \hat{\bm{y}}^{*}(\tilde{y}) - \hat{\bm{u}} \|_{2}^{2}.\]
 If $\tilde{y} \geq \| \hat{\bm{u}} \|_{2}$, then $\hat{\bm{y}}^{*}(\tilde{y})=\hat{\bm{u}}$. This shows that the minimum of $\phi$ on $[\| \hat{\bm{u}} \|_{2},+\infty)$ is attained at $\tilde{y}_{1}=\|\hat{\bm{u}} \|_{2}$, at a value of $\phi(\tilde{y}_{1})) = \left(\|\hat{\bm{u}} \|_{2}-\tilde{u}\right)^{2}$.
When $\tilde{y} \in [0,\|\hat{\bm{u}} \|_{2}]$, we have 
$\hat{\bm{y}}^{*}(\tilde{y}) = \left(\tilde{y}/ \| \hat{\bm{u}} \|_{2}\right) \hat{\bm{u}}$.
Note that here, $\tilde{y} \mapsto \hat{\bm{y}}^{*}(\tilde{y})$ is differentiable.  Therefore, $\phi:\tilde{y} \mapsto \left( \tilde{y} - \tilde{u} \right)^{2} + \| \hat{\bm{y}}^{*}(\tilde{y}) - \hat{\bm{u}} \|_{2}^{2}$ is also differentiable. The first-order optimality conditions yield a closed-form solution for the minimum of $\phi$ on $[0,\| \hat{\bm{u}} \|_{2}]$, with $\tilde{y}_{2}= \dfrac{\tilde{u} +\| \hat{\bm{u}} \|_{2}}{2}$. For this value of $\tilde{y}_{2}$, we obtain $\phi(\tilde{y}_{2}) = (1/2) \left(\|\hat{\bm{u}} \|_{2}-\tilde{u}\right)^{2}.$ Therefore, the global minimum of $\phi$ on $[0,+\infty)$ is attained at $\tilde{y}_{2}$, yielding
\[ \pi_{\mcC}(\bm{u})  = \left( \dfrac{\tilde{u} +\| \hat{\bm{u}} \|_{2}}{2}, \dfrac{\tilde{u} +\| \hat{\bm{u}} \|_{2}}{2}\frac{\hat{\bm{u}}}{\| \hat{\bm{u}} \|_{2}}\right).\]
\end{proof}
\subsection{Proofs for confidence regions in the simplex}
\begin{proof}[Proof of Proposition \ref{prop:proj-confidence-regions}]
We can write $\mcX = \bm{x}_{0} + \epsilon \tilde{B},$ where
$ \tilde{B} = \{ \bm{z} \in \mbR^{n} \; \vert \; \bm{z}^{\top}\bm{e}=0, \| \bm{z} \|_{2} \leq 1\}.$

Suppose we made a sequence of decisions $\bm{x}_{1}, ...,\bm{x}_{T}$, which can be written as $\bm{x}_{t} = \bm{x}_{0} + \epsilon \bm{z}_{t}$ for $\bm{z}_{t} \in \tilde{B}.$ Then it is clear that for any sequence of payoffs $\bm{f}_{1},...,\bm{f}_{T}$, we have
\begin{align}\label{eq:first-reformulation-regret}
\sum_{t=1}^{T} \omega_{t} \langle \bm{f}_{t},\bm{x}_{t} \rangle - \min_{\bm{x} \in \mcX}  \sum_{t=1}^{T} \omega_{t} \langle \bm{f}_{t},\bm{x} \rangle = \epsilon_{x} \left( \sum_{t=1}^{T} \omega_{t} \langle \bm{f}_{t},\bm{z}_{t} \rangle - \min_{\bm{z} \in \tilde{B}}  \sum_{t=1}^{T}  \omega_{t}\langle \bm{f}_{t},\bm{z} \rangle \right).
\end{align}
Therefore,  if we run \cbap{} on the set $\tilde{B}$ to obtain $O \left( \sqrt{T} \right)$ growth of the right-hand side of \eqref{eq:first-reformulation-regret},  we obtain a no-regret algorithm for $\mcX$. We now show how to run \cbap{} for the set $\tilde{B}$.
Let $\mcV = \{ \bm{v} \in \mbR^{n} \; \vert \; \bm{v}^{\top}\bm{e}=0\}.$
 We use the following orthonormal basis of $\mcV$:
 let $\bm{v}_{1}, ..., \bm{v}_{n-1} \in \mbR^{n}$ be the vectors
$\bm{v}_{i} = \sqrt{i/(i+1)} \left(1/i, ..., 1/i, -1, 0,...,0 \right), \forall \; i=1,...,n-1,$ where the component $1/i$ is repeated $i$ times.
The vectors $\bm{v}_{1},...,\bm{v}_{n-1}$ are orthonormal and constitute a basis of $\mcV$ \citep{egozcue2003isometric}.  Writing $\bm{V} = \left(\bm{v}_{1},...,\bm{v}_{n-1} \right) \in \mbR^{n \times (n-1)}$,  and noting that $\bm{V}^{\top}\bm{V} = \bm{I}$, we can write $\tilde{B} = \{ \bm{Vs} \; \vert \; \bm{s} \in \mbR^{n-1},\| \bm{s} \|_{2} \leq 1 \}.$
Now, if $\bm{x} = \bm{x}_{0} + \epsilon_{x} \bm{z}_{t}$ with $\bm{z}_{t} \in \mcV$, we have $\bm{z}_{t} = \bm{Vs}_{t}$,  for $\bm{s}_{t} \in \mbR^{n-1}$ and $\| \bm{s} \|_{2} \leq 1$.  Finally, $\sum_{t=1}^{T} \omega_{t} \langle \bm{f}_{t},\bm{x}_{t} \rangle - \min_{\bm{x} \in \mcX}  \sum_{t=1}^{T} \omega_{t} \langle \bm{f}_{t},\bm{x} \rangle$ is equal to 
\begin{equation}\label{eq:second-reformulation-regret}\epsilon_{x} \left( \sum_{t=1}^{T} \omega_{t} \langle \bm{V}^{\top} \bm{f}_{t},\bm{s}_{t} \rangle - \min_{\bm{s} \in \mbR^{n-1},\|\bm{s}\|_{2} \leq 1}  \sum_{t=1}^{T} \omega_{t} \langle \bm{V}^{\top} \bm{f}_{t},\bm{s} \rangle \right).
\end{equation}
Therefore,  to obtain a regret minimizer for \eqref{eq:second-reformulation-regret} with observed payoffs $\left( \bm{f} \right)_{t \geq 1}$, we can run \cbap{} on the right-hand side,  where the decision set is an $\ell_{2}$ ball and the sequence of observed payoffs is $\left( \bm{V}^{\top}\bm{f}_{t} \right)_{t \geq 1}$.  
In the previous section we showed how to efficiently instantiate \cbap{} in this setting (see Proposition \ref{prop:proj-ball-2-C-polar}).
\end{proof}
\begin{remark}
In this section we have highlighted a sequence of reformulations of the regret, from \eqref{eq:first-reformulation-regret} to \eqref{eq:second-reformulation-regret}.  We essentially showed how to instantiate \cbap{} for settings where the decision set $\mcX$ is the intersection of an $\ell_{2}$ ball with a hyperplane for which we have an orthonormal basis.
\end{remark}
\section{Details on OMD, FTRL and optimistic variants}\label{app:details-omd}
\subsection{Algorithms}
For solving our instances of distributionally robust optimization, we compare \spcbap{} with the following four state-of-the-art algorithms: at iteration $t \geq 1$, for a step size $\eta_{t}>0$, the updates are:
\begin{enumerate}
\item Follow-The-Regularized-Leader (FTRL) \citep{abernethy2009competing,mcmahan2011follow}:
\begin{equation}\label{alg:FTRL}\tag{FTRL}
\bm{x}_{t+1} \in \arg \min_{ \bm{x} \in \mcX } \; \langle \sum_{\tau=1}^{t} \bm{f}_{\tau},\bm{x}\rangle + \dfrac{1}{ \eta_{t} } \| \bm{x} \|_{2}^{2}.
\end{equation}
Optimistic FTRL \citep{rakhlin2013online}: given estimation $\bm{m}^{t+1}$ of loss at iteration $t+1$, choose
\begin{equation}\label{alg:pred-FTRL}\tag{O-FTRL}
\bm{x}_{t+1} \in \arg \min_{ \bm{x} \in \mcX } \; \langle \sum_{\tau=1}^{t} \bm{f}_{\tau} + \bm{m}^{t+1},\bm{x}\rangle + \dfrac{1}{\eta_{t}} \|  \bm{x} \|_{2}^{2}.
\end{equation}
\item Online Mirror Descent (OMD) \citep{nemirovsky1983problem,beck2003mirror}:
\begin{equation}\label{alg:OMD}\tag{OMD}
\bm{x}_{t+1} \in \min_{\bm{x} \in \mcX } \langle \bm{f}_{t},\bm{x} \rangle + \dfrac{1}{ \eta_{t} } \| \bm{x}-\bm{x}_{t}\|_{2}^{2}.
\end{equation}
 Optimistic OMD \citep{chiang2012online}: given estimation $\bm{m}^{t+1}$ of loss at iteration $t+1$, 
 \begin{equation}\label{alg:pred-OMD}\tag{O-OMD}
  \begin{aligned}
 &\bm{z}_{t+1}  \in \min_{\bm{z} \in \mcX } \langle \bm{m}_{t+1},\bm{z} \rangle + \dfrac{1}{ \eta_{t} } \| \bm{z}-\bm{x}_{t}\|_{2}^{2}, \\
 & \text{Observe the loss }\bm{f}_{t+1} \text{ related to } \bm{z}_{t+1}, \\
 & \bm{x}_{t+1}  \in \min_{\bm{x} \in \mcX } \langle \bm{f}_{t+1},\bm{x} \rangle + \dfrac{1}{ \eta_{t}} \| \bm{x} - \bm{x}_{t}\|_{2}^{2}.
 \end{aligned}
 \end{equation}
\end{enumerate}
Note that these algorithms can be written more generally using Bregman divergence (e.g., \cite{BenTal-Nemirovski}). We choose to work with $\| \cdot \|_{2}$ instead of Kullback-Leibler divergence as this $\ell_{2}$-setup is usually associated with faster empirical convergence rates \citep{ChambollePock16,gao2021increasing}. Additionally,  following \cite{chiang2012online,rakhlin2013online}, we use the last observed loss as the predictor for the next loss, i.e., we set $\bm{m}^{t+1} = \bm{f}_{t}$. 
\subsection{Implementations}\label{app:OMD-implementation}
The proximal updates defined in the previous section need to be resolved for the decision sets of both players of the distributionally robust optimization problem \eqref{eq:dro}. We present the details of our implementation here. The results in the rest of this section are reminiscent to the novel tractable proximal setups presented in \cite{grand2020first,grand2020scalable}.

\paragraph{Computing the projection steps for the first player}
For $\mcX = \{ \bm{x} \in \mbR^{n} \; \vert \; \| \bm{x}  - \bm{x}_{0} \|_{2} \leq \epsilon_{x}\}$, $\bm{c}, \bm{x}' \in \mbR^{n}$ and a step size $\eta>0$, the prox-update becomes 
\[\min_{\| \bm{x} -\bm{x}_{0} \|_{2} \leq \epsilon_{x}} \langle \bm{c}, \bm{x} \rangle + \dfrac{1}{2 \eta} \| \bm{x} - \bm{x}'\|_{2}^{2}.\]
Using a change of variable, we find that the optimal solution $\bm{x}^{*}$ to the problem above is 
\[\bm{x}^{*} = \bm{x}_{0} + \epsilon_{x} \dfrac{\bm{x}'-\eta \bm{c} - \bm{x}_{0}}{\max\{\epsilon_{x},\| \bm{x}'-\eta \bm{c} - \bm{x}_{0} \|_{2} \}}.\]
\paragraph{Computing the projection steps for the second player}
For $\mcY = \{ \bm{y} \in \Delta(m) \; \vert \; \| \bm{y}  - \bm{y}_{0} \|_{2} \leq \epsilon_{y}\}$, the proximal update of the second player from a previous point $\bm{y}'$ and a step size of $\eta>0$ becomes
\begin{equation}\label{eq:prox-update-y-player}
\min_{\| \bm{y} -\bm{y}_{0} \|_{2} \leq \epsilon_{y}, \bm{y} \in \Delta(m)} \langle \bm{c}, \bm{y} \rangle + \dfrac{1}{2 \eta} \| \bm{y} - \bm{y}'\|_{2}^{2}.
\end{equation}
If we dualize the $\ell_{2}$ constraint with a Lagrangian multiplier $\mu \geq 0$ we obtain the relaxed problem $q(\mu)$ where
\begin{equation}\label{eq:q-mu}
q(\mu) = - (1/2) \epsilon_{y}^{2} \mu + \min_{\bm{y} \in \Delta(m)} \langle \bm{c}, \bm{y} \rangle + \dfrac{1}{2 \eta} \| \bm{y} - \bm{y}'\|_{2}^{2} + \dfrac{\mu}{2}\| \bm{y} - \bm{y}_{0} \|_{2}^{2}.
\end{equation}
Note that the $\arg \min$ in 
\[ \min_{\bm{y} \in \Delta(m)} \langle \bm{c}, \bm{y} \rangle + \dfrac{1}{2 \eta} \| \bm{y} - \bm{y}'\|_{2}^{2} + \dfrac{\mu}{2}\| \bm{y} - \bm{y}_{0} \|_{2}^{2} \]
is the same $\arg \min$ as in 
\begin{equation}\label{eq:intermediary-q-mu}
 \min_{\bm{y} \in \Delta(m)} \| \bm{y} - \dfrac{\eta}{\eta \mu + 1} \left( \dfrac{1}{\eta} \bm{y}' + \mu \bm{y}_{0} - \bm{c} \right) \|_{2}^{2}.
\end{equation}
Note that \eqref{eq:intermediary-q-mu} is an orthogonal projection onto the simplex.  Therefore, it can be solved efficiently \citep{euclidean-projection}. We call $\bm{y}(\mu)$ an optimal solution of \eqref{eq:intermediary-q-mu}. Then $q(\mu)$ can be rewritten
\[ q(\mu) = - (1/2) \epsilon_{y}^{2} \mu + \langle \bm{c}, \bm{y}(\mu) \rangle + \dfrac{1}{2 \eta} \| \bm{y}(\mu) - \bm{y}'\|_{2}^{2} + \dfrac{\mu}{2}\| \bm{y}(\mu) - \bm{y}_{0} \|_{2}^{2}.\]
We can therefore binary search $q(\mu)$ as in the previous expression.
An upper bound $\bar{\mu}$ for $\mu^{*}$ can be computed as follows. Note that 
\[ q(\mu) \leq - (1/2) \epsilon_{y}^{2} \mu + \langle \bm{c}, \bm{y}_{0} \rangle + \dfrac{1}{2 \eta} \| \bm{y}_{0}- \bm{y}'\|_{2}^{2}.\]
Since $\mu \mapsto q(\mu)$ is concave we can choose $\bar{\mu}$ such that $q(\bar{\mu}) \leq q(0)$. Using the previous inequality this yields
\[ \bar{\mu} = \frac{2}{\epsilon_{y}^{2}} \left(  \langle \bm{c}, \bm{y}_{0} \rangle + \dfrac{1}{2 \eta} \| \bm{y}_{0}- \bm{y}'\|_{2}^{2} - q(0) \right).\]
In our simulations, we search for an optimal $\mu$ using the {\sf minimize\_scalar} function from the {\sf sklearn} Python package, with an accuracy of $\epsilon=0.001$.
\subsection{Computing the theoretical fixed step sizes for Section \ref{sec:simu-dro}}\label{app:simu-dro-step-size}
For \ref{alg:OMD} and \ref{alg:FTRL}, in theory (e.g., \cite{BenTal-Nemirovski}),  for a player with decision set $\mcX$, we can choose $\eta_{\sf th} = \sqrt{2} \Omega/L\sqrt{T}$ with $\Omega = \max_{\bm{x},\bm{x}' \in \mcX} \| \bm{x}-\bm{x}'\|_{2}$, and $L$ an upper bound on the norm of any observed loss $ \bm{f}_{t}$: $ \| \bm{f}_{t} \|_{2} \leq L, \forall \; t \geq 1$. Note that this requires to know 1) the number of iterations $T$, and 2) the upper bound $L$ on the norm of any observed loss $\bm{f}_{t}$, before the losses are generated. For \ref{alg:pred-OMD}, we can choose $\eta_{\sf th} = 1/\sqrt{8}L$ (Corollary 6 in \cite{syrgkanis2015fast}), and for \ref{alg:pred-FTRL}, we can choose $\eta_{\sf th} = 1/2L$ (Corollary 8 in \cite{syrgkanis2015fast}).

We now show how to compute $L_{x}$ and $L_{y}$ (for the first player and the second player) for an instance of the distributionally robust logistic regression problem \eqref{eq:dro}.
\begin{enumerate}
\item For the first player we have $\bm{f}_{t} = \bm{A}^{t}\bm{y}_{t}$, where $\bm{A}^{t}$ is the matrix of subgradients of $\bm{x} \mapsto F(\bm{x},\bm{y}_{t})$ at $\bm{x}_{t}$:
\[ A^{t}_{ij} = \dfrac{-b_{i}a_{i,j}\exp(-b_{i}\bm{a}_{i}^{\top}\bm{x}_{t})}{1+\exp(-b_{i}\bm{a}_{i}^{\top}\bm{x}_{t})} + \mu x_{j}, \forall \; (i,j) \in \{ 1,...,m\} \times \{1,...,n\}.\]
Therefore, $\| \bm{f}_{t} \|_{2} \leq \| \bm{A}^{t} \|_{2} \| \bm{y}_{t} \|_{2} \leq \| \bm{A}^{t} \|_{2}$, because $\bm{y} \in \Delta(m)$. Now we have $\| \bm{A}^{t} \|_{2} \leq \| \bm{A}^{t} \|_{F} = \sqrt{\sum_{i,j} \vert A_{ij}^{t}\vert^{2}}$.  
Note that \[ \sqrt{\sum_{i,j} \vert A_{ij}^{t}\vert^{2}} \leq \sum_{i,j} \vert A_{ij}^{t}\vert.\]
We also have $\vert A^{t}_{ij} \vert \leq \vert b_{i}a_{i,j}\vert + \mu \vert x_{j} \vert$. Recall that we have $\bm{x} \in \mbR^{n}$ such that $\| \bm{x} - \bm{x}_{0} \|_{2} \leq \epsilon_{x}$. We obtain the following upper bound:
\[ L_{x} = \sum_{i,j}\vert b_{i}a_{i,j}\vert  + \mu \cdot m \cdot \left( \| \bm{x}_{0} \|_{1} + \sqrt{n} \epsilon_{x} \right).\]
\item For the second player, the loss $\bm{f}_{t}$ is $\bm{f}_{t} = \left( \ell_{i}(\bm{x}_{t}) \right)_{i \in [1,m]}$, with $\ell_{i}(\bm{x}) = \log(1+\exp(-b_{i}\bm{a}^{\top}_{i}\bm{x}))$. 
For each $i \in[1,m]$ we have 
\[\vert \ell_{i}(\bm{x}) \; \vert \; \leq \log(1+\exp( \vert b_{i}\vert \epsilon_{x} \| \bm{a}_{i} \|_{2} )),\] and we can conclude that
\[L_{y} = \sqrt{\sum_{i=1}^{m} \log(1+\exp( \vert b_{i}\vert \epsilon_{x} \| \bm{a}_{i} \|_{2} ))^{2} }.\]
\end{enumerate}

\section{Computing the theoretical step sizes for Section \ref{sec:simu-mdp}}\label{app:constant-mdp}
In the saddle-point formulation of MDP, the objective function is
\[F(\bm{v},\bm{\mu})= (1-\lambda) \bm{p}_{0}^{\top}\bm{v} + \sum_{s=1}^{n} \sum_{a=1}^{A} \mu_{sa} \left( r_{sa}+\lambda \bm{P}_{sa}^{\top}\bm{v} - v_{s}\right),\]
for $\bm{v} \in \mbR^{n},\| \bm{v} \|_{2} \leq \sqrt{n} r_{\infty}/(1-\lambda)$ and $\mu \in \Delta(n \times A)$.
The function $F$ is differentiable and we have $\nabla_{v} F(\bm{v},\bm{\mu}) \in \mbR^{n},\nabla_{\mu} F(\bm{v},\bm{\mu}) \in \mbR^{n \times A}$ with
\begin{align*}
 \left(\nabla_{v} F(\bm{v},\bm{\mu})\right)_{s'} & = (1-\lambda)p_{0s'} + \lambda \sum_{s,a} \mu_{sa} P_{sas'} - \sum_{a} \mu_{s'a}, \forall \; s' \in [n], \\
 \left(\nabla_{\mu} F(\bm{v},\bm{\mu})\right)_{sa} & = r_{sa} + \lambda \bm{P}_{sa}^{\top}\bm{v} - v_{s}, \forall \; (s,a) \in [n] \times [A].
\end{align*}
We now provide upper bounds $L_{v}$ and $L_{\mu}$ on $\| \nabla_{v} F(\bm{v},\bm{\mu}) \|_{2}$ and $\| \nabla_{\mu} F(\bm{v},\bm{\mu}) \|_{2}$.
Using the equivalence between $\| \cdot \|_{2}$ and $\| \cdot \|_{1}$, we have, for $\bm{\mu} \in \Delta(n \times A)$,
\begin{align*}
\| \nabla_{v} F(\bm{v},\bm{\mu}) \|_{2} & \leq \| \nabla_{v} F(\bm{v},\bm{\mu}) \|_{1} \\
& \leq (1-\lambda) + \lambda \sum_{s',a,s} \mu_{sa}P_{sas'} + \sum_{s',a} \mu_{s'a} \\
& \leq (1-\lambda) + \lambda + 1 \\
& \leq 2.
\end{align*}
For bounding $\|\nabla_{\mu} F(\bm{v},\bm{\mu}) \|_{2}$, we can rely on Cauchy-Schwarz's inequality and $\| \bm{v} \|_{2} \leq \sqrt{n} r_{\infty}/(1-\lambda)$ to obtain 
\[\| \nabla_{\mu} F(\bm{v},\bm{\mu}) \|_{2} \leq \| \bm{r} \|_{2} + \frac{ \sqrt{n} r_{\infty}}{1-\lambda} \left(A \left( \lambda  n +1\right)\right).\]

Overall, we can choose 
\[ L_{v} = 2, L_{\mu} = \|\bm{r} \|_{2} + \frac{ \sqrt{n} r_{\infty}}{1-\lambda} \left(A \left( \lambda  n +1\right)\right).\]

\end{document}